\documentclass[a4paper,reqno,final]{amsart}

\usepackage{amsmath,amssymb,amsthm}
\usepackage{fullpage,enumitem}
\usepackage{mathtools}
\mathtoolsset{showonlyrefs,showmanualtags}
\usepackage{multirow,array}
\usepackage{showkeys,version}
\usepackage{tikz} 
\usetikzlibrary{cd}
\usepackage{dynkin-diagrams}

\numberwithin{equation}{subsection}
\numberwithin{figure}{subsection}
\numberwithin{table}{subsection}
\allowdisplaybreaks
\newenvironment{Ack}%
{\par \vspace{\baselineskip}%
 \noindent \textbf{Acknowledgements.}}%
{\par \vspace{\baselineskip}}

\newtheorem{df}{Definition}[section]
\newtheorem{thm}[df]{Theorem}
\newtheorem{prop}[df]{Proposition}
\newtheorem{lem}[df]{Lemma}

\newtheorem{ex}[df]{Example}
\newtheorem{rmk}[df]{Remark}

\newcommand{\ep}{\varepsilon}

\DeclareMathOperator{\im}{Im}

\title{Homogeneous Ulrich bundles on homogenous varieties of certain exceptional types} 
\author{Yusuke Nakayama}
\date{}
\begin{document}
\maketitle
\begin{abstract}
This paper studies the Ulrich property of homogeneous vector bundles on rational homogenous varieties.
We provide a criterion for an initialized irreducible homogeneous vector bundle on a rational homogeneous variety with any Picard number to be Ulrich with respect to any polarizations.
This criterion extends Fonarev's result, which applies to rational homogeneous varieties with Picard number one.
As an application, we show that rational homogeneous varieties with Picard number at least two of certain exceptional algebraic groups do not admit such homogeneous Ulrich bundles with respect to the minimal ample class.
\end{abstract}
\section{Introduction}
Let $(X,\mathcal{O}_{X}(1))$ be a polarized smooth projective variety over $\mathbb{C}$.
A vector bundle $E$ on $X$ is said to be {\it Ulrich} if $H^{i}(X,E(-t))$ vanishes for all $i\ {\rm and}\ 1\leq t\leq \dim(X)$, where $E(t):=E\otimes\mathcal{O}_{X}(1)^{\otimes t}$.
These bundles were initially introduced in the context of commutative algebra in \cite{Ulrich}.
In fact, Ulrich studied Cohen--Macaulay modules with ``linear free resolution'' in the sense of \cite{ESW}.

Eisenbud--Schreyer--Weyman \cite{ESW} asked whether or not every smooth projective variety carry an Ulrich bundle with respect to some polarization.
For example, they showed that projective curve and Veronese variety have an Ulrich sheaf.
Beauville \cite{Bea2} showed that any abelian surface admits a rank $2$ Ulrich bundle.
Various other studies have also been conducted on some varieties (see \cite{ACCMT}, \cite{ACMR}, \cite{AFO}, \cite{BN}, \cite{Cas1}, \cite{Cas2}, \cite{Fae}, \cite{Lop1}, \cite{Lop}, \cite{MRPM}).

We consider the case that $X$ is a homogeneous variety $G/P$, where $G$ is a complex semi-simple linear algebraic group and $P$ is a parabolic subgroup of $G$.
We first review the case that $X$ has Picard number one, i.e. $P$ is a maximal parabolic subgroup.
In type $A$, such varieties are Grassmann varieties.
Costa--Mir\'{o}-Roig \cite{CM2} classified all irreducible homogeneous Ulrich bundles on Grassmann variety.
Fonarev \cite{Fon} gave a criterion for an irreducible homogeneous vector bundle to be Ulrich applicable to all types of $G/P$ with Picard number one.
By using this result, he classified irreducible homogeneous Ulrich bundles on isotropic Grassmann varieties of types $B,C$ and $D$.
Lee--Park \cite{LP} studied the case that $G$ is an exceptional type.
In particular, they showed that the only homogeneous varieties $G/P$ with Picard number one admitting an irreducible homogeneous Ulrich bundle are the Cayley plane $E_{6}/P_{1}$ and the $E_{7}$-adjoint variety $E_{7}/P_{1}$.

In this paper, we study the existence problem of Ulrich bundles on a homogeneous variety $G/P$ with Picard number $r\geq2$.
We consider the case that polarization is the minimal ample class, i.e. the minimal element of ample cone of $G/P$(see section $2$).
In type $A$, such varieties are the $r$-step flag variety $Fl(k_{1},\cdots,k_{r};\mathbb{C}^{n})$;
the variety of the nested sequence $V_{1}\subset\cdots\subset V_{r}\subset\mathbb{C}^{n}$ of linear subspaces with dim$(V_{i})=k_{i}$, with $1\leq k_{1}<\cdots<k_{r}\leq n$.
The polarization, in this case, is given by the Pl\"{u}cker embedding followed by the Segre embedding
$$Fl(k_{1},\cdots,k_{r};\mathbb{C}^{n})\hookrightarrow \prod_{i=1}^{r}Gr(k_{i},\mathbb{C}^{n})\hookrightarrow \prod_{i=1}^{r}\mathbb{P}^{\binom{n}{k_{i}}-1}\hookrightarrow\mathbb{P}^{\prod_{i=1}^{r}\binom{n}{k_{i}}-1}.$$ 
Coskun--Costa--Huizenga--Mir\'{o}-Roig--Woolf \cite{CCHMW} showed that no flag variety $Fl(k_{1},\cdots,k_{r};\mathbb{C}^{n})$ admits an initialized irreducible homogeneous Ulrich bundles with respect to minimal ample class if $r\geq3$.
Furthermore, they classified initialized irreducible homogeneous Ulrich bundles with respect to minimal ample class on certain two step flag varieties and conjectured that the two step flag variety $Fl(k_{1},k_{2};\mathbb{C}^{n})$ does not admit an initialized irreducible homogeneous Ulrich bundle with respect to minimal ample class if $k_{1}\geq3$ and $k_{2}-k_{1}\geq3$.
Coskun--Jaskowiak \cite{CJ} proved that this conjecture is true.
To the best of our knowledge no results are known for other types.

The main result in this paper is as follows.
\begin{thm}[Theorem \ref{thm}]\label{thm1}If $G$ is a complex semi-simple linear algebraic group of types $E_{6},\ F_{4}$ or $G_{2}$ and $P$ is a parabolic subgroup that is not maximal, then homogeneous varieties $G/P$ do not carry initialized irreducible homogeneous Ulrich bundles with respect to the minimal ample class.
\end{thm}
In order to prove this theorem, we provide a criterion for an initialized irreducible homogeneous vector bundle on a homogeneous variety $G/P$ of all types to be Ulrich with respect to any ample equivariant line bundles (see Lemma \ref{lem}).
This theorem is proved by specific calculations by using this criterion.

Here, it would be appropriate to mention that there is an another interesting class called arithmetically Cohen--Macaulay (ACM) bundles.
In particular, Ulrich bundles are ACM bundles. Such bundles on a homogeneous variety with Picard number one were studied by \cite{CM1},\cite{DFR} and \cite{FNR}. 

The paper is organized as follows.
In section $2$, we collect some definitions and notations and state precisely form of Theorem \ref{thm1}.
Moreover, we provide a criterion for an initialized irreducible homogeneous vector bundle on a homogeneous variety to be Ulrich with respect to any polarizations and prove it.
In section $3$, we treat homogeneous varieties of type $E_{6}$. Similarly, in section $4$ and section $5$ we deal with homogeneous varieties of types $F_{4}$ and $G_{2}$, respectively.

\begin{Ack} I am grateful to Takeshi Ikeda for their helpful advice and comments, to Hajime Kaji for useful discussions and beneficial comments, to Chikashi Miyazaki for comments on the previous draft, Takafumi Kouno for beneficial discussions. 
\end{Ack}
\section{Statement of results}
All algebraic varieties in this study are defined over the field of complex numbers, $\mathbb{C}$.
\subsection{Main theorem}
In order to state a precise form of Theorem \ref{thm}, we prepare notations. Let $G$ be a semi-simple linear algebraic group over the complex field $\mathbb{C}$ and $T\subset G$ be a maximal torus. We set $\mathfrak{g}:=$Lie$(G)$ and $\mathfrak{h}:=$Lie$(T)$. Let $\Phi$ be the set of roots associated with pair $(\mathfrak{g},\mathfrak{h})$.
We fix set $\Delta:=\{\alpha_{1},\cdots,\alpha_{n}\}$ of simple roots of $\Phi$. Let $\Phi^{+}$ be the set of positive roots.
The {\it weight lattice} $\Lambda$ is the set of all linear functions $\lambda:\mathfrak{h}\to\mathbb{C}$ for which $\frac{2(\lambda,\alpha)}{(\alpha,\alpha)}\in \mathbb{Z}$ for any $\alpha \in \Phi$, where $(,)$ denotes the Killing form. We define $\lambda \in \Lambda$ to be {\it dominant\/} if all integers $\frac{2(\lambda,\alpha)}{(\alpha,\alpha)}\ (\alpha \in \Phi)$ are nonnegative and {\it strongly\ dominant\/} if they are positive. 
Let $\Lambda^{+}$ be the set consisting of dominant weights.
Let $\varpi_{1},\cdots,\varpi_{n}\in \Lambda$ be the {\it fundamental weights\/}, i.e., $\frac{2(\varpi_{i},\alpha_{j})}{(\alpha_{j},\alpha_{j})}=\delta_{i,j}$. From this definition, $\lambda=\sum_{i=1}^{n}a_{i}\varpi_{i}$ is a dominant weight if and only if $a_{i}\geq 0$ and a strongly dominant weight if and only if $a_{i}>0$ for $1\leq i \leq n$. For each dominant $\lambda\in \Lambda$, the irreducible highest weight representation of  $\mathfrak{g}$ with highest weight $\lambda$ is denoted by $V_\lambda$.

Let $I$ be the index set of Dynkin nodes of $G$ and $J$ be a subset of $I$.
Let $\Phi^{-}(J)$ be a set consisting of negative roots $\alpha$ such that $\alpha=\sum_{j\notin J}a_{j}\alpha_{j}$. Let
$$\mathfrak{p}_{J}:=\mathfrak{h}{\huge \oplus}(\oplus_{\alpha\in\Phi^{+}}\mathfrak{g}_{\alpha})\oplus(\oplus_{\alpha\in\Phi^{-}(J)}\mathfrak{g}_{\alpha}),$$
and $P_{J}$ be a subgroup of $G$ such that the Lie algebra of $P_{J}$ is $\mathfrak{p}_{J}$, where each $\mathfrak{g}_{\alpha}$ is a one-dimensional eigenspace with respect to the adjoint action of $\mathfrak{h}$. We define the set $\Phi_{J}^{+}$ as follows:
$$\Phi_{J}^{+}:=\{\alpha\in\Phi^{+}\ |\ (\varpi_{j},\alpha)\neq 0\ {\rm for\ some}\ j\in J\}.$$ 
A homogeneous variety $G/P_{J}$ has $|J|$ projections $\pi_{j}$ from $G/P_{J}$ to $G/P_{j}$. The Picard  group of $G/P_{J}$ is generated by $L_{j}:=\pi^{*}_{j}\mathcal{O}_{G/P_{j}}(1)$. We define $\mathcal{O}(1):=\mathcal{O}_{G/P_{J}}(1)$ as $\otimes_{j\in J}L_{j}^{\otimes b_{j}}$ with positive integers $b_{j}$ for every $j\in J$.
If $b_{j}=1$ for every $j\in J$, we call $\mathcal{O}(1)=\otimes_{j\in J}L_{j}$ {\it minimal\ ample\ class}.
Note that the cardinality of $\Phi_{J}^{+}$ is equal to the dimension of the homogeneous variety $G/P_{J}$. 

Next, we consider vector bundles on $G/P_{J}$. In particular, we introduce an important class of vector bundles. A vector bundle, $E$, with fiber $V$ over $G/P_{J}$ is {\it homogeneous\/} if there exists a representation $\rho:P_{J}\to GL(V)$ such that $E\cong E_{\rho}$, where $E_{\rho}$ is a vector bundle over $G/P_{J}$ with fiber $V$ originating from the principal bundle, $G\to G/P_{J}$, via $\rho$. Let $E$ be a homogeneous vector bundle on $G/P_{J}$. If representation $\rho:P_{J}\to GL(V)$ is irreducible, we call $E$ an {\it irreducible\ homogeneous\ vector\ bundle}.

Finally, we describe all irreducible representations of a parabolic subgroup, $P_{J}$.
Let $\varpi_{j}$ be the corresponding fundamental weight and $S_{P_{J}}$ be the semi-simple part of $P_{J}$.
Then all irreducible representations of $P_{J}$ are
$$V\otimes \left(\otimes_{j\in J} L_{\varpi_{j}}^{t_{j}}\right),$$
where $V$ is an irreducible representation of $S_{P_{J}}$, $t_{j}\in \mathbb{Z}$, and $L_{\varpi_{j}}$ is a one-dimensional representation with weight $\varpi_{j}$.
If $\lambda$ is the highest weight of irreducible representation $V$ of $S_{P_{J}}$, we say that $\lambda+\sum_{j\in J}t_{j}\varpi_{j}$ is the highest weight of irreducible representation $V\otimes \left(\otimes_{j\in J} L_{\varpi_{j}}^{t_{j}}\right)$ of $P_{I}$. 
In this study, $E_{\lambda}$ denotes the irreducible homogeneous vector bundle arising from the irreducible representation of $P_{J}$ with highest weight $\lambda$.

 As any irreducible representation of a semi-simple linear algebraic group is determined by its highest weight, if $E_{\lambda}$ is an irreducible homogeneous vector bundle on $G/P_{J}$ with $\lambda=\sum_{i\in I} a_{i}\varpi_{i}$, then $a_{j}\geq 0$ for $j\notin J$.
 In general, $a_{j}$ is not necessarily non-negative integer if $j$ belongs to $J$.
 However, if we impose the initialized on $E_{\lambda}$, then $a_{i}$ is non-negative integer for every $i$ (see Lemma \ref{initialized}).
 Recall that a vector bundle $E$ on a projective variety $X$ is called an {\it initilized} if $H^{0}(X,E(-1))=0\ {\rm and}\ H^{0}(X,E)\neq 0$, where $E(t):=E\otimes_{\mathcal{O}_{X}}\mathcal{O}_{X}(t)$. 
 
Now that the notations are ready, the main theorem is stated again.

\begin{thm}\label{thm}For any subset $I\subset\Delta$ with $|J|\geq2$, a homogeneous variety $G/P_{J}$ does not admit an initialized irreducible homogeneous Ulrich bundle with respect to $\mathcal{O}(1):=\otimes_{j\in J}L_{j}$ when $G$ is one of types $E_{6}$, $F_{4}$ or $G_{2}$.
\end{thm}

The proof of this theorem is based on specific calculations with respect to each types of $G$ and $J$ by using Lemma \ref{lem}.

By combining the result of Lee--Park with main theorem, all initialized irreducible homogeneous Ulrich bundles with respect to $\mathcal{O}(1)$ on a homogeneous variety $G/P_{J}$ are classified when $G$ is one of types $E_{6}$, $F_{4}$ or $G_{2}$.
\subsection{Criterion}
In this subsection, we state a criterion for an initialized irreducible homogeneous vector bundle on a homogeneous variety to be Ulrich with respect to any polarizations.

Let $E_{\lambda}$ be an initialized irreducible homogeneous vector bundle over $G/P_{J}$ with highest weight $\lambda=\sum_{i\in I}a_{i}\varpi_{i}$. We define the map $\varphi_{\lambda}^{J}$ as follows:
$$\varphi_{\lambda}^{J}:\Phi_{J}^{+} \to \mathbb{Q},\ \  \alpha\mapsto\frac{1}{(\alpha,\sum_{j\in J}b_{j}\varpi_{j})}(\lambda+\rho,\alpha),$$
where $\rho$ is the half sum of all positive roots.
We now state the criterion for an initialized irreducible homogeneous vector bundle to be Ulrich with respect to $\mathcal{O}(1):=\otimes_{j\in J}L_{j}^{\otimes b_{j}}$.

\begin{lem}\label{lem}With the notations as above, the following are equivalent.\\
(1) $E_{\lambda}$ is an Ulrich vector bundle with respect to $\mathcal{O}(1)$.\\
(2) $\varphi_{\lambda}^{J}$ is injective
with image $[1,\dim(G/P_{J})]:=\{1,2,\cdots,\dim(G/P_{J})\}\subset\mathbb{Z}$.\\
\end{lem}

Note that this criterion does not depend on types of $G$ and polarizations. When the cardinality of $J$ is equal to one (in other words, $P_{J}$ is a maximal parabolic subgroup) and $b_{i}$ is equal to one, then this coincides with the Fonarev's criterion.

In order to prove the Lemma \ref{lem}, recall the following definitions. Let $\lambda$ be a weight in $\Lambda$.
$\lambda$ is called {\it singular\/} if there exists $\alpha \in \Phi^{+}$ such that $(\lambda,\alpha)=0$.
$\lambda$ is called {\it regular\ of\ index\/} $p$ if it is not singular and if there are exactly $p$ roots $\alpha_{1},\cdots,\alpha_{p} \in \Phi^{+}$ such that $(\lambda,\alpha_{i})<0$.
The following theorem is crucial to the proof of the Lemma \ref{lem}.
\begin{thm}[Borel-Weil-Bott]\label{BWB} Let $E_{\lambda}$ be an irreducible homogeneous vector bundle over $G/P_{J}$.\\
(i) If $\lambda+\rho$ is singular, then
$$H^{i}(G/P_{J},E_{\lambda})=0,\ \forall i.$$
(ii) If $\lambda+\rho$ is regular of index $p$, then
$$H^{i}(G/P_{J},E_{\lambda})=0,\ \forall i\neq p$$
and
$$H^{p}(G/P_{J},E_{\lambda})=V_{w(\lambda+\rho)-\rho},$$
where we write the sum of all fundamental weights for $\rho$;  $w(\lambda+\rho)$ is the unique element of the fundamental Weyl Chamber of $G$, which is congruent to $\lambda+\rho$ under the action of the Weyl group; and $V_{w(\lambda+\rho)-\rho}$ is an irreducible representation of $G$.
\end{thm}
\begin{rmk}\ In \cite{Dem}, a parabolic subgroup is stated in the case of Borel subgroup. However, it is known that it is sufficient to consider such cases (See \cite{Ott}).
\end{rmk}

Before proving Lemma \ref{lem}, we show the following lemma.

\begin{lem}\label{initialized}Let $E_{\lambda}$ be an initialized irreducible homogeneous vector bundle on $G/P_{J}$ with $\lambda=\sum_{i=1}^{n}a_{i}\varpi_{i}$. Then we have $a_{i}\geq0$ for any $1\leq i\leq n$.
\end{lem}
\begin{proof}The condition $H^{0}(G/P_{J},E_{\lambda})\neq0$ implies that $\lambda+\rho$ is a regular of index $0$ by the Borel--Weil-Bott Theorem. In other words, $\lambda+\rho=\sum_{i=1}^{n}(a_{i}+1)\varpi_{i}$ is a strongly dominant. Therefore, we have $a_{i}+1>0$ for any $i\in I$. Since $a_{i}$ is an integer, this is grater than or equal to $0$. 
\end{proof}

\begin{proof}[Proof of Lemma \ref{lem}] 
Note that the highest weight of $E_{\lambda}(-t)$ is $\lambda-t\cdot\sum_{j\in J}b_{j}\varpi_{j}$.
By Borel--Weil--Bott theorem, $E_{\lambda}$ is an Ulrich vector bundle with respect to $\mathcal{O}(1)$ if and only if $\lambda+\rho-t\cdot\sum_{j\in J}b_{j}\varpi_{j}$ is singular for every $t \in [1,\dim(G/P_{I})]$.
If $\alpha$ is in $\Phi^{+}\setminus\Phi_{J}^{+}$, then we have $(\lambda+\rho-t\cdot\sum_{j\in J}b_{j}\varpi_{j},\alpha)=(\lambda+\rho,\alpha)$. As $a_{i}\geq0$ for $i\in I$ by Lemma \ref{initialized} and $(\rho,\alpha)>0$, this is positive. Hence, we consider only the case in which $\alpha$ is in $\Phi_{J}^{+}$.

If $\lambda+\rho-t\cdot\sum_{j\in J}b_{j}\varpi_{j}$ is singular for any $t \in [1,dim(G/P_{J})]$, there exists a positive root $\alpha$ in $\Phi_{J}^{+}$ such that
$$0=(\lambda+\rho-t\cdot\sum_{j\in J}b_{j}\varpi_{j},\alpha)=(\lambda+\rho,\alpha)-t\cdot(\sum_{j\in J}b_{j}\varpi_{j},\alpha).$$
As $\alpha$ is in $\Phi_{J}^{+}$, $(\sum_{j\in J}b_{j}\varpi_{j},\alpha)$ is not equal to zero. Since the dimension of $G/P_{J}$ is equal to the cardinality of $\Phi^{+}_{J}$, this $\alpha$ must be unique, that is, $\varphi_{\lambda}^{J}$ is injective with image $[1,\dim(G/P_{J})]$.

Conversely, suppose that $\varphi_{\lambda}^{J}$ is injective
with image $[1,\dim(G/P_{J})]$.
Then for every $t\in[1,\dim(G/P_{J})]$,
there exists a unique positive root $\alpha$ in $\Phi_{J}^{+}$ such that $t=\varphi_{J}^{+}(\alpha)$.
Therefore, we have
$$0=(\lambda+\rho,\alpha)-t\cdot(\alpha,\sum_{j\in J}b_{j}\varpi_{j})=(\lambda+\rho-t\cdot\sum_{j\in J}b_{j}\varpi_{j},\alpha).$$
Hence, $\lambda+\rho-t\cdot\sum_{j\in J}b_{j}\varpi_{j}$ is singular for any $t \in [1,dim(G/P_{J})]$.
\end{proof}
\subsection{Bad pair}
Let $S$ be a subset of $I$, and  
$\Lambda^{+}_S$ be $\bigoplus_{i\in S}\mathbb{Z}_{\ge 0}\varpi_i.$
\begin{df}\label{badpair}
A pair 
$\{\alpha,\beta\}\subset \Phi^+_J$ of positive roots is a {\it bad pair} if 
there is a subset $S\subset I$ and an element $\mu\in \Lambda_S^+$ such that the following two conditions hold:
\begin{eqnarray*}
&(1)&\varphi_{\varpi_{i}-\rho}^{J}(\alpha)=\varphi_{\varpi_{i}-\rho}^{J}(\beta)\ (\forall i \in I\setminus S),\ {\rm and}\\
&(2)&\varphi_{\mu}^{J}(\alpha)-\varphi_{\mu}^{J}(\beta)\in\mathbb{Q}\setminus\mathbb{Z}.
\end{eqnarray*}
In this case, we say that $\{\alpha,\beta\}$ is bad pair with respect to $(S,\mu)$.
\end{df}
\begin{ex}We consider the case that $G$ is type $F_{4}$
(For notations, see Section $4$).
Let $J=\{1,4\}\subset I=\{1,2,3,4\}$ and $b_{j}=1$ for all $j\in J$.
Then $\{\ep_{2},\ep_{1}-\ep_{4}\}\subset \Phi_{J}^{+}$ is a bad pair.
Indeed, we can take $S=J$ and $\mu=0\cdot\varpi_{1}+1\cdot\varpi_{4} \in \Lambda_{S}^{+}$.
Then one can check that
$\varphi_{\varpi_{2}-\rho}^{J}(\ep_{2})-\varphi_{\varpi_{2}-\rho}^{J}(\ep_{1}-\ep_{4})=1-1=0$, and
$\varphi_{\varpi_{3}-\rho}^{J}(\ep_{2})-\varphi_{\varpi_{3}-\rho}^{J}(\ep_{1}-\ep_{4})=\frac{1}{2}-\frac{1}{2}=0$, and
$\varphi_{\mu}^{J}(\ep_{2})-\varphi_{\mu}^{J}(\ep_{1}-\ep_{4})=\frac{5}{2}-3=-\frac{1}{2}\in\mathbb{Q}\setminus\mathbb{Z}$.
\end{ex}


For $\lambda=\sum_{i\in I}a_i \varpi_i\in \Lambda^+$,
and $S\subset I$, 
define 
\[
\lambda_S=\sum_{i\in S}a_i \varpi_i
\in \Lambda^+_S.
\]
Note that for every $\gamma\in\Phi_{J}^{+}$, we have
\begin{eqnarray}\label{eq}
   \varphi_{\lambda}^{J}(\gamma)=\varphi_{\lambda_{S}}^{J}(\gamma)+\sum_{i\in I\setminus S}\varphi_{\varpi_{i}-\rho}^{J}(\gamma)a_{i}. 
\end{eqnarray}
\begin{lem}\label{lem2}
Let $\lambda\in \Lambda^+$.
If there is a subset $S\subset I$ and $\alpha,\beta\in \Phi^+_J$
such that $\{\alpha,\beta\}$ is bad pair with respect to $(S,\lambda_S)$,
then $E_{\lambda}$ is not Ulrich bundle on $G/P_{J}$.
\end{lem}
\begin{proof}Suppose for contradiction that $E_{\lambda}$ is Ulrich bundle with respect to $\mathcal{O}(1)=\otimes_{j\in J}L_{j}^{\otimes b_{j}}$.
By the above equation (\ref{eq}) and the condition $(1)$ of Definition \ref{badpair}, we have
\begin{eqnarray*}
\varphi_{\lambda}^{J}(\alpha)-\varphi_{\lambda}^{J}(\beta)&=&\left(\varphi_{\lambda_{S}}^{J}(\alpha)+\sum_{i\in I\setminus S}\varphi_{\varpi_{i}-\rho}^{J}(\alpha)a_{i}\right)-\left(\varphi_{\lambda_{S}}^{J}(\beta)+\sum_{i\in I\setminus S}\varphi_{\varpi_{i}-\rho}^{J}(\beta)a_{i}\right)\\
&=&\varphi_{\lambda_{S}}^{J}(\alpha)-\varphi_{\lambda_{S}}^{J}(\beta).
\end{eqnarray*}
This is not an integer by the condition $(2)$ of Definition \ref{badpair}.
Therefore, we see that at least one of the value of $\varphi_{\lambda}^{J}(\alpha)$ or $\varphi_{\lambda}^{J}(\beta)$ is not an integer.
This contradicts $\im(\varphi_{\lambda}^{J})=[1,\dim(G/P_{J})]$.
\end{proof}

This lemma also does not depend on types of $G$ and polarizations.

From the next sections, we investigate initialized irreducible homogeneous Ulrich bundles on $G/P_{J}$ when is one of types $E_{6}$, $F_{4}$ or $G_{2}$.
In the rest of the paper, we only consider the polarization of $G/P_{I}$ given by the minimal ample line bundle
$\mathcal{O}(1)=\otimes_{j\in J} L_j$.
\section{Type $E_{6}$}
The goal of this section is to prove the following theorem.
\begin{thm}\label{thmE6}For every $J\subset I$ with $|J|\geq2$, there are no initialized irreducible homogeneous Ulrich bundles with respect to $\mathcal{O}(1)$ on $E_{6}/P_{J}$.
\end{thm}
The proof is based on specific calculations with respect to each $I$ by using Lemma \ref{lem}.
The following subsections 

Let $G$ be a simply connected simple linear algebraic group with a Dynkin diagram of type $E_{6},\ E_{7}$, or $E_{8}$ as follows.
\begin{center}
$E_{6}$:\dynkin[label]{E}{6}\ \ \ \ \ \ $E_{7}$:\dynkin[label]{E}{7}\ \ \ \ \ \ $E_{8}$:\dynkin[label]{E}{8}
\end{center}
$\epsilon_{1},\cdots,\epsilon_{n}$ denotes the usual standard orthonoramal basis for $\mathbb{R}^{n}$. The $\mathbb{Z}$-span of this basis is a lattice denoted by $I_{n}$. Let $I^{\prime}_{8}=I_{8}+\mathbb{Z}((\epsilon_{1}+\cdots+\epsilon_{8})/2)$ and $L$ be a subgroup of $I^{\prime}_{8}$ consisting of all elements $\sum c_{i}\epsilon_{i}+\frac{c}{2}(\epsilon_{1}+\cdots+\epsilon_{8})$ for which $\sum c_{i}$ is an even integer. We define $\Phi_{E_{8}}=\{\alpha \in L\ |\ (\alpha,\alpha)=2\}$. $\Phi_{E_{8}}$ consists of the vectors $\pm(\epsilon_{i}\pm\epsilon_{j}),\ i\neq j,$ along with the lesser ones $\frac{1}{2}\sum_{i=1}^{8}(-1)^{k(i)}\epsilon_{i}$ (where $k(i)=0,1$ sum to an even integer). As a base, we take
$$\Delta_{E_{8}}=\{\alpha_{1}:=\frac{1}{2}(\epsilon_{1}-\epsilon_{2}-\epsilon_{3}-\epsilon_{4}-\epsilon_{5}-\epsilon_{6}-\epsilon_{7}+\epsilon_{8}),\ \alpha_{2}:=\epsilon_{1}+\epsilon_{2},\ \alpha_{i}:=\epsilon_{i-1}-\epsilon_{i-2}\ (3\leq i\leq8)\}.$$

Let $V_{6}$ denote the span $\{\alpha_{1},\cdots,\alpha_{6}\}_{\mathbb{R}{\rm -lin}} \subset \mathbb{R}^{8}$. We define $\Phi_{E_{6}}=\Phi_{E_{8}}\cap V_{6}$. $\Phi_{E_{6}}$ consists of the vectors $\pm(\epsilon_{i}\pm\epsilon_{j}),\ (1\leq i<j\leq 5)$ along with the lesser ones $\frac{1}{2}(\epsilon_{8}-\epsilon_{7}-\epsilon_{6}+\sum_{i=1}^{5}(-1)^{k(i)}\epsilon_{i})$ (where $k(i)=0,1$ sum to an odd integer).
As a base, we take
$$\Delta_{E_{6}}:=\{\alpha_{1},\cdots,\alpha_{6}\}.$$
Let $\gamma$, $\gamma_{i,j}$ and $\beta_{l}$ for each $1\leq i<j\leq 5$ and $1\leq l\leq5$ be written as follows:
\begin{eqnarray*}
\gamma&=&\alpha_{1}+ 2\alpha_{2}+ 2\alpha_{3}+ 3\alpha_{4}+ 2\alpha_{5}+ \alpha_{6}\\
\gamma_{i,j}&=&\alpha_{1}+ \alpha_{2}+ 2\alpha_{3}+ 3\alpha_{4}+ 2\alpha_{5}+ \alpha_{6}-(\sum_{s=1}^{i-1}\alpha_{s+2}+\sum_{t=1}^{j-2}\alpha_{3+t})\\
\beta_{l}&=&\alpha_{1}+\sum_{s=3}^{l+1}\alpha_{s},
\end{eqnarray*} 
which are elements in $\Phi_{E_{6}}^{+}$.
The set of positive roots, $\Phi_{E_{6}}^{+}$, is
$$\{\pm\epsilon_{i}+\epsilon_{j}\}_{1\leq i<j\leq 5}\cup\{\gamma_{i,j}\}_{1\leq i<j\leq 5}\cup\{\gamma\}\cup\{\beta_{i}\}_{1\leq i\leq5}.$$
The total number of positive roots is $36$. 
The fundamental weights are as follows.
\begin{eqnarray*}
\varpi_{1}&=&-\frac{2}{3}\epsilon_{6}-\frac{2}{3}\epsilon_{7}+\frac{2}{3}\epsilon_{8}\\
\varpi_{2}&=&\frac{1}{2}(\epsilon_{1}+\epsilon_{2}+\epsilon_{3}+\epsilon_{4}+\epsilon_{5}-\epsilon_{6}-\epsilon_{7}+\epsilon_{8})\\
\varpi_{3}&=&\frac{1}{2}(-\epsilon_{1}+\epsilon_{2}+\epsilon_{3}+\epsilon_{4}+\epsilon_{5})+\frac{5}{6}(-\epsilon_{6}-\epsilon_{7}+\epsilon_{8})\\
\varpi_{4}&=&\epsilon_{3}+\epsilon_{4}+\epsilon_{5}-\epsilon_{6}-\epsilon_{7}+\epsilon_{8}\\
\varpi_{5}&=&\epsilon_{4}+\epsilon_{5}+\frac{2}{3}(-\epsilon_{6}-\epsilon_{7}+\epsilon_{8})\\
\varpi_{6}&=&\epsilon_{5}+\frac{1}{3}(-\epsilon_{6}-\epsilon_{7}+\epsilon_{8})
\end{eqnarray*} 

\subsection{Proof of theorem \ref{thmE6} when $|J|=2$}
In this subsection, we show that $E_{6}/P_{J}$ does not admit an initialized Ulrich bundle for any subset $J\subset\Delta$ with $|J|=2$.
\begin{prop}\label{prop1}There are no initialized irreducible homogeneous Ulrich bundles on $E_{6}/P_{1,2}\cong E_{6}/P_{2,6}$.
\end{prop}
Before proving this proposition, we prepare some lemmas.
\begin{lem}\label{n1lem}
If there is an initialized irreducible homogeneous Ulrich bundle $E_{\lambda}$ on $E_{6}/P_{1,2}$ with $\lambda=\sum_{i\in I}a_{i}\varpi_{i}$, then we have that exactly one of $a_{1}$ or $a_{2}$ is equal to zero.
\end{lem}
\begin{proof}
One can check by direct computation that for each $\alpha \in \Phi^+_I\setminus \{\alpha_1,\alpha_2\}$ we have
$\varphi_\lambda^J(\alpha)>1$, and
$\varphi_\lambda^J(\alpha_1)=a_{1}+1$, $\varphi_\lambda^J(\alpha_2)=a_{2}+1$.
So in order that
$1\in \im(\varphi_{\lambda}^J)$, 
we must have either $a_1=0$ or $a_2=0.$
Since $\varphi_{\lambda}^{I}$ is injective, we have $a_{1}\neq a_{2}$.
\end{proof}
\begin{lem}\label{n2lem}If there is an initialized irreducible homogeneous Ulrich bundle $E_{\lambda}$ on $E_{6}/P_{1,2}$ with $\lambda=\sum_{i\in I}a_{i}\varpi_{i}$ and $a_{1}=0$, we obtain that exactly one of $a_{2}=1$ or $a_{3}=0$ holds.
\end{lem}
\begin{proof}
Let $\Psi:=\{\alpha_1,\alpha_{2},\beta_{2}\}$.
We already know $a_{2}\geq 1$.
Under those conditions of $a_{1}$ and $a_{2}$, we see that for each $\alpha \in \Phi^+_J\setminus\Psi$ we have
$\varphi_\lambda^J(\alpha)>2$,
and
$\varphi_\lambda^J(\beta_2)=a_{3}+2$.
So, in view of $\varphi_\lambda^J(\alpha_2)=a_{2}+1$, the claim 
holds by the same argument  
above
applied for $a_2$ and $a_3$.
\end{proof}
\begin{lem}\label{n3lem}Suppose that there is an initialized irreducible homogeneous Ulrich bundle $E_{\lambda}$ on $E_{6}/P_{1,2}$ with $\lambda=\sum_{i\in I}a_{i}\varpi_{i}$.
If $a_{1}=0$ and $a_{2}=1$, we get that exactly one of $a_{3}=1$ or $a_{4}=0$ holds.
\end{lem}
\begin{proof}
Let $\Psi:=\{\alpha_1,\alpha_{2},\beta_{2},\ep_{1}+\ep_{3},\gamma_{4,5}\}$.
We already know $a_{3}\geq1$.
Under those conditions of $a_{1}$, $a_{2}$ and $a_{3}$, one can check that for each $\alpha \in \Phi^+_I\setminus \Psi$ we have
$\varphi_\lambda^J(\alpha)>3$,
and
$\varphi_\lambda^J(\ep_{1}+\ep_{3})=a_{4}+3$,
$\varphi_\lambda^J(\gamma_{4,5})=\frac{5+a_{3}+a_{4}}{2}$.
Thus the only elements of $\Phi^{+}_{I}$ that can take the value $3$ under the map $\varphi_{\lambda}^{J}$ are $\beta_{2},\ep_{1}+\ep_{3},\gamma_{4,5}$.
However $\varphi_\lambda^J(\gamma_{4,5})=3$ does not occur.
Indeed, if it occurs, we have $(a_{3},a_{4})=(1,0)$.
Then, we have $\varphi_\lambda^I(\beta_2)=3$, which contradicts the injectivity of $\varphi_{\lambda}^{I}$.
So the claim holds.
\end{proof}
\begin{lem}\label{maxlem}Suppose that there is an initialized irreducible homogeneous Ulrich bundle $E_{\lambda}$ on $E_{6}/P_{1,2}$ with $\lambda=\sum_{i\in I}a_{i}\varpi_{i}$.
If $a_{1}=0$, $a_{2}=1$ and $a_{3}=1$, we have $a_{5}=0$
\end{lem}
\begin{proof}
The function $\varphi_{\lambda}^{J}$ takes the maximal value $\dim(E_{6}/P_{1,2})=26$ at $\ep_{4}+\ep_{5}$.
In fact, one can check by direct computation that the coefficients of $a_{i}$ for $4\leq i\leq 6$ in $\varphi_{\lambda}^{I}(\ep_{4}+\ep_{5})$ are grater than or equal to those coefficients in $\varphi_{\lambda}^{J}(\alpha)$ for every $\alpha\in\Phi_{I}^{+}$,
and the constant term of $\varphi_{\lambda}^{J}(\ep_{4}+\ep_{5})$ is grater than the one of $\varphi_{\lambda}^{J}(\alpha)$ for every $\alpha\in\Phi_{J}^{+}\setminus \{\ep_{4}+\ep_{5}\}$.
Similarly, we can show that the function $\varphi_{\lambda}^{J}$ takes the second maximal value $25$ at $\ep_{3}+\ep_{5}$.
Hence, $\varphi_{\lambda}^{J}(\ep_{4}+\ep_{5})=9+2a_{4}+2a_{5}+a_{6}=26$,
and $\varphi_{\lambda}^{J}(\ep_{3}+\ep_{5})=8+2a_{4}+a_{5}+a_{6}=25$,
so we have $a_{5}=0$.
\end{proof}
\begin{lem}\label{list1}
Let $J=\{1,2\}$.
The following pairs $\{\alpha,\beta\}$ are bad with respect to $(S,\mu)$.
\begin{center}
    \begin{tabular}{cccc} \hline
   &$\{\alpha,\beta\}$&$S$&$\mu$ \\ \hline
   (I)&$\{\gamma_{4,5},\gamma_{3,5}\}$
   &\{1,2,3,5\}
&$0\cdot\varpi_{1}+1\cdot\varpi_{2}+1\cdot\varpi_{3}+0\cdot\varpi_{5}$ \\
   (I\hspace{-1.2pt}I)&$\{\gamma_{2,5},\gamma_{3,5}\}$&\{1,2,4,5\} &$0\cdot\varpi_{1}+1\cdot\varpi_{2}+0\cdot\varpi_{4}+0\cdot\varpi_{5}$ \\
   (I\hspace{-1.2pt}I\hspace{-1.2pt}I)&$\{\gamma_{4,5},\gamma_{3,5}\}$&\{1,2,3,5\}&
$0\cdot\varpi_{1}+2\cdot\varpi_{2}+0\cdot\varpi_{3}+0\cdot\varpi_{5}$ \\
   (I\hspace{-1.2pt}V)&$\{\gamma_{4,5},\gamma_{3,5}\}$&\{1,3,4,5\} &$0\cdot\varpi_{1}+0\cdot\varpi_{3}+0\cdot\varpi_{4}+0\cdot\varpi_{5}$ \\
   
   (V)&$\{\gamma_{4,5},\gamma_{3,5}\}$&\{1,2,4,5\} &$1\cdot\varpi_{1}+0\cdot\varpi_{2}+1\cdot\varpi_{4}+0\cdot\varpi_{5}$ \\
   (V\hspace{-1.2pt}I)&$\{\gamma_{4,5},\gamma_{3,5}\}$&\{1,2,3,5\}  &$1\cdot\varpi_{1}+0\cdot\varpi_{2}+0\cdot\varpi_{3}+0\cdot\varpi_{5}$ \\
   
   (V\hspace{-1.2pt}I\hspace{-1.2pt}I)&$\{\gamma_{2,5},\gamma_{3,5}\}$&\{1,2,4\} &
   $2\cdot\varpi_{1}+0\cdot\varpi_{2}+0\cdot\varpi_{4}$ \\ 
   (V\hspace{-1.2pt}I\hspace{-1.2pt}I\hspace{-1.2pt}I)&$\{\gamma_{2,4},\gamma_{3,4}\}$&\{2,3,4\} &
   $0\cdot\varpi_{2}+0\cdot\varpi_{3}+0\cdot\varpi_{4}$ \\
   (I\hspace{-1.2pt}X)&$\{\gamma_{4,5},\gamma_{3,5}\}$&\{2,4,5\} &
   $0\cdot\varpi_{2}+0\cdot\varpi_{4}+0\cdot\varpi_{5}$ \\ \hline
 \end{tabular}
\end{center}
\end{lem}
\begin{proof}
Let $\{\gamma_{4,5},\gamma_{3,5}\}$, $S=\{1,2,3,5\}$ and $\mu=0\cdot\varpi_{1}+1\cdot\varpi_{2}+1\cdot\varpi_{3}+0\cdot\varpi_{5}$.
Then, we have $\varphi_{\varpi_{3}-\rho}^{J}(\gamma_{4,5})=\varphi_{\varpi_{3}-\rho}^{J}(\gamma_{3,5})=\frac{1}{2}$, and
$\varphi_{\varpi_{6}-\rho}^{J}(\gamma_{4,5})=\varphi_{\varpi_{6}-\rho}^{J}(\gamma_{3,5})=0$, and
$\varphi_{\mu}^{J}(\gamma_{4,5})-\varphi_{\mu}^{J}(\gamma_{3,5})=\frac{5}{2}-3=-\frac{1}{2}\in\mathbb{Q}\setminus\mathbb{Z}$.
So $\{\gamma_{4,5},\gamma_{3,5}\}$ is a bad pair with respect to $(S,\mu)$.
For the rest part, we can also check by direct computation this lemma.
\end{proof}
\begin{proof}[Proof of Proposition \ref{prop1}]
Suppose for contradiction that there is an initialized irreducible homogeneous Ulrich bundle $E_{\lambda}$ on $E_{6}/P_{1,2}$ with $\lambda=\sum_{i=1}^{6}a_{i}\varpi_{i}$.
We investigate the condition that $1\in \im(\varphi_\lambda^I)$.
By Lemma \ref{n1lem}, we have that exactly one of $a_{1}$ or $a_{2}$ is equal to zero.

Let us consider the case $a_{1}=0$ first. 
Now we turn to the condition that $2\in \im(\varphi_\lambda^I)$.
By Lemma \ref{n2lem}, we obtain that exactly one of $a_{2}=1$ or $a_{3}=0$ holds.
Now suppose that $a_{2}=1$.
We examine the condition that $3\in \im(\varphi_\lambda^I)$.
By Lemma \ref{n3lem}, we have that exactly one of $a_{3}=1$ or $a_{4}=0$ holds.
We assume that $a_{3}=1$.
By Lemma \ref{maxlem}, we have $a_{5}=0$.
Now we are considering the case $(a_1,a_2,a_3,a_5)=(0,1,1,0)$ holds.
In this case, we can find a bad pair of positive roots;
if we set $S=\{1,2,3,5\}$ and $\mu=\sum_{i\in S}a_i\varpi_i\in \Lambda^+_S$, then
$\{\gamma_{4,5},\gamma_{3,5}\}$ is a bad pair with respect to $(S,\mu)$.
In fact the pair corresponds to (I) in Lemma \ref{list1}.
Suppose that $a_{1}=0$, $a_{2}=1$ and $a_{4}=0$.
By the same argument as the proof of Lemma \ref{maxlem}, we have $a_{5}=0$.
In this case, we can find a bad pair corresponding to (I\hspace{-1.2pt}I) in the list of Lemma \ref{list1}.
Assume that $a_{1}=0$ and $a_{3}=0$.
We turn to the condition that $3\in \im(\varphi_\lambda^J)$.
By the same argument as the proof of Lemma \ref{n3lem}, we obtain that exactly one of $a_{2}=2$ or $a_{4}=0$ holds.
If $a_{2}=2$, we have $a_{5}=0$ by the same argument as the proof of Lemma \ref{maxlem}.
Then, we can find a bad pair corresponding to (I\hspace{-1.2pt}I\hspace{-1.2pt}I) in the list of Lemma \ref{list1}.
If $a_{4}=0$, we also have $a_{5}=0$ by the same argument as the proof of Lemma \ref{maxlem}.
In this case, we can find a bad pair corresponding to (I\hspace{-1.2pt}V) in the list of Lemma \ref{list1}.

Let us consider the case $a_{2}=0$.
We turn to the condition that $2\in \im(\varphi_\lambda^J)$.
By the same argument as the proof of Lemma \ref{n2lem}, we see that exactly one of $a_{4}=0$ or $a_{1}=1$ holds.
Suppose that $a_{4}=0$.
We examine the condition that $3\in \im(\varphi_\lambda^J)$.
By the same argument as the proof of Lemma \ref{n3lem}, we can check that exactly one of $a_{1}=2$ or $a_{3}=0$ or $a_{5}=0$ holds. 
If $a_{1}=2$ (resp. $a_{3}=0$, $a_{5}=0$), this is (V\hspace{-1.2pt}I\hspace{-1.2pt}I) (resp. (V\hspace{-1.2pt}I\hspace{-1.2pt}I\hspace{-1.2pt}I), (I\hspace{-1.2pt}X)) in the list of Lemma \ref{list1}.
Assume that $a_{2}=0$ and $a_{1}=1$.
We turn to the condition that $3\in \im(\varphi_\lambda^J)$.
By the same argument as the proof of Lemma \ref{n3lem}, we see that exactly one of $a_{4}=1$ or $a_{3}=0$ holds.
If $a_{4}=1$, we have $a_{5}=0$ by the same argument as the proof of Lemma \ref{maxlem}.
In this case, we can find a bad pair corresponding to (V) in the list of Lemma \ref{list1}.
So this contradicts our hypothesis.
If $a_{3}=0$, we also get $a_{5}=0$ by the same argument as the proof of Lemma \ref{maxlem}.
This case is (V\hspace{-1.2pt}I) in the list of Lemma \ref{list1}.
So this case also contradicts our hypothesis by Lemma \ref{lem2}.
\end{proof}
The proof of Theorem \ref{thmE6} is proved by a similar argument of the proof of Proposition \ref{prop1} for each $J$.
From now on, if $(\lambda_{S},)$ can be determined from the context, it may not be written.
\begin{prop}There are no initialized irreducible homogeneous Ulrich bundles on $E_{6}/P_{1,3}\cong E_{6}/P_{5,6}$.
\end{prop}
\begin{proof}Suppose for contradiction that there is an initialized irreducible homogeneous Ulrich bundle $E_{\lambda}$ on $E_{6}/P_{1,3}$ with $\lambda=\sum_{i=1}^{6}a_{i}\varpi_{i}$.
By the same argument as the proof of Lemma \ref{n1lem}, we see that exactly one of $a_{1}=0$ or $a_{3}=0$ holds.

Suppose that $a_{1}=0$.
We turn to the condition that $2\in \im(\varphi_\lambda^J)$.
In this case, we have $a_{3}=2$ by taking $\Psi=\{\alpha_{1},\alpha_{3},\beta_{2}\}$ and the same argument as the proof of Lemma \ref{n2lem}.
If $a_{3}=2$, we have $a_{5}=0$ by the same argument as the proof of Lemma \ref{maxlem}.
In this case, the function $\varphi_{\lambda}^{J}$ takes the maximal value $\dim(E_{6}/P_{1,3})=26$ and second maximal value $25$ at $\ep_{4}+\ep_{5}$ and $\ep_{3}+\ep_{5}$ respectively. 
Then, we can check that $\{\gamma_{1,4},\gamma_{1,3}\}$ is a bad pair.

Assume that $a_{3}=0$.
We examine the condition that $2\in \im(\varphi_\lambda^J)$.
By $\Psi=\{\alpha_{1},\alpha_{3},\beta_{2}\}$ and the same argument as the proof of Lemma \ref{n2lem}, we see that exactly one of $a_{1}=2$ or $a_{4}=0$ holds.
If $a_{1}=2$, we obtain $a_{5}=0$ by similar argument as the proof of Lemma \ref{maxlem}.
In this case, the function $\varphi_{\lambda}^{J}$ takes the maximal value $\dim(E_{6}/P_{1,3})=26$ and second maximal value $25$ at $\ep_{4}+\ep_{5}$ and $\ep_{3}+\ep_{5}$ respectively. 
Then, we see that $\{\beta_{3},\beta_{4}\}$ is a bad pair.
If $a_{4}=0$, we can check that $\{\beta_{2},\beta_{3}\}$ is a bad pair.
\end{proof}
\begin{prop}There are no initialized irreducible homogeneous Ulrich bundles on $E_{6}/P_{1,4}\cong E_{6}/P_{4,6}$.
\end{prop}
\begin{proof}Suppose for contradiction that there is an initialized irreducible homogeneous Ulrich bundle $E_{\lambda}$ on $E_{6}/P_{4,6}$ with $\lambda=\sum_{i=1}^{6}a_{i}\varpi_{i}$.
By the same argument as the proof of Lemma \ref{n1lem}, we check that exactly one of $a_{4}=0$ or $a_{6}=0$ holds.

Suppose that $a_{4}=0$.
We turn to the condition that $2\in \im(\varphi_\lambda^J)$.
By taking $\Psi=\{\alpha_{4},\alpha_{6},\ep_{1}+\ep_{3},-\ep_{1}+\ep_{3},-\ep_{2}+\ep_{4}\}$ and the same argument as the proof of Lemma \ref{n2lem}, we obtain that exactly one of $a_{6}=1$ or $a_{2}=0$ or $a_{3}=0$ or $a_{5}=0$ holds.
If $a_{6}=1$ (resp. $a_{2}=0,\ a_{3}=0,\ a_{5}=0)$, we can check that $\{\ep_{2}+\ep_{5},\ep_{3}+\ep_{4}\}$ (resp. $\{\ep_{2}+\ep_{5},-\ep_{1}+\ep_{5}\}$, $\{\ep_{1}+\ep_{5},\ep_{2}+\ep_{5}\}$, $\{\ep_{3}+\ep_{5},\ep_{4}+\ep_{5}\}$) is a bad pair.

Assume that $a_{6}=0$.
We examine the condition that $2\in \im(\varphi_\lambda^J)$.
We see that exactly one of $a_{4}=1$ or $a_{5}=0$ holds by taking $\Psi=\{\alpha_{4},\alpha_{6},-\ep_{3}+\ep_{5}\}$ and the similar argument as proof of Lemma \ref{n2lem}.
If $a_{4}=1$ (resp. $a_{5}=0$), we can check that $\{\ep_{2}+\ep_{5},\ep_{3}+\ep_{4}\}$ (resp. $\{\ep_{3}+\ep_{5},\ep_{4}+\ep_{5}\}$) is a bad pair.
\end{proof}
\begin{prop}There are no initialized irreducible homogeneous Ulrich bundles on $E_{6}/P_{1,5}\cong E_{6}/P_{3,6}$.
\end{prop}
\begin{proof}Suppose for contradiction that there is an initialized irreducible homogeneous Ulrich bundle $E_{\lambda}$ on $E_{6}/P_{3,6}$ with $\lambda=\sum_{i=1}^{6}a_{i}\varpi_{i}$.
By the same argument as the proof of Lemma \ref{n1lem}, we deduce that exactly one of $a_{3}=0$ or $a_{6}=0$ holds.

Assume that $a_{3}=0$.
We consider the condition that $2\in \im(\varphi_\lambda^J)$.
By taking $\Psi=\{\alpha_{3},\alpha_{6},-\ep_{1}+\ep_{3},\beta_{2}\}$ and the same argument as proof of Lemma \ref{n2lem}, we see that exactly one of $a_{6}=1$ or $a_{4}=0$ or $a_{1}=0$ holds.
If $a_{6}=1$ (resp. $a_{4}=0,\ a_{1}=0$), we can check that $\{\gamma_{2,4},\gamma_{1,5}\}$ (resp. $\{\ep_{2}+\ep_{5},\ep_{3}+\ep_{5}\}$, $\{-\ep_{1}+\ep_{5},\beta_{5}\}$) is a bad pair.

Suppose that $a_{6}=0$.
We examine the condition that $2\in \im(\varphi_\lambda^J)$.
By taking $\Psi=\{\alpha_{3},\alpha_{6},-\ep_{3}+\ep_{5}\}$ and the same argument as proof of Lemma \ref{n2lem}, we see that exactly one of $a_{3}=1$ or $a_{5}=0$ holds.
When $a_{3}=1$ (resp. $a_{5}=0$), we can check that $\{\gamma_{2,4},\gamma_{1,5}\}$ (resp. $\{\gamma_{2,4},\gamma_{2,3}\}$) is a bad pair.
\end{proof}
\begin{prop}There are no initialized irreducible homogeneous Ulrich bundles on $E_{6}/P_{1,6}$.
\end{prop}
\begin{proof}Suppose for contradiction that there is an initialized irreducible homogeneous Ulrich bundle $E_{\lambda}$ on $E_{6}/P_{1,6}$ with $\lambda=\sum_{i=1}^{6}a_{i}\varpi_{i}$.
By taking $\Psi=\{\alpha_{},\alpha_{}\}$ and the same argument as the proof of Lemma \ref{n1lem}, we deduce that exactly one of $a_{1}=0$ or $a_{6}=0$ holds.

Suppose that $a_{1}=0$.
We examine the condition that $2\in \im(\varphi_\lambda^J)$.
By taking $\Psi=\{\alpha_{1},\alpha_{6},\beta_{2}\}$ and the same argument as proof of Lemma \ref{n2lem}, we see that exactly one of $a_{6}=1$ or $a_{3}=0$ holds.
If $a_{3}=0$, we can check that $\{\gamma_{1,4},\gamma_{2,4}\}$ is a bad pair.
Assume that $a_{6}=1$.
We turn to the condition that $3\in \im(\varphi_\lambda^J)$.
By taking $\Psi=\{\alpha_{1},\alpha_{6},\beta_{2},-\ep_{3}+\ep_{5}\}$ and the same argument as proof of Lemma \ref{n3lem}, we deduce that exactly one of $a_{3}=1$ or $a_{5}=0$ holds.
If $a_{5}=0$, we can check that $\{\gamma_{2,3},\gamma_{2,4}\}$ is a bad pair.
We assume that $a_{3}=1$.
We examine the condition that $4\in \im(\varphi_\lambda^J)$.
By taking $\Psi=\{\alpha_{1},\alpha_{6},\beta_{2},-\ep_{3}+\ep_{5}, \beta_{3}\}$ and the same argument as proof of Lemma \ref{n3lem}, we deduce that exactly one of $a_{5}=1$ or $a_{4}=0$ holds.
If $a_{4}=0$, we can check that $\{\gamma_{1,3},\gamma_{1,2}\}$ is a bad pair.
We also assume that $a_{5}=1$.
In this case also, we deduce that $a_{4}=2$ by taking $\Psi=\{\alpha_{1},\alpha_{6},\beta_{2},-\ep_{3}+\ep_{5}, \beta_{3},\beta_{5}\}$ and using the same argument as the proof of Lemma \ref{n3lem}.
Then we see that $\{\gamma_{1,3},\gamma_{1,2}\}$ is a bad pair.

Assume that $a_{6}=0$.
We examine the condition that $2\in \im(\varphi_\lambda^J)$.
By taking $\Psi=\{\alpha_{1},\alpha_{6},-\ep_{3}+\ep_{5}\}$ and the same argument as proof of Lemma \ref{n2lem}, we see that exactly one of $a_{1}=1$ or $a_{5}=0$ holds.
If $a_{5}=0$, we can see that $\{\gamma_{2,4},\gamma_{2,3}\}$ is a bad pair.
Suppose that $a_{1}=1$.
We turn to the condition that $3\in \im(\varphi_\lambda^J)$.
By taking $\Psi=\{\alpha_{1},\alpha_{6},-\ep_{3}+\ep_{5},\beta_{2}\}$ and the same argument as proof of Lemma \ref{n3lem}, we deduce that exactly one of $a_{5}=1$ or $a_{3}=0$ holds.
If $a_{3}=0$, we can check that $\{\gamma_{1,4},\gamma_{2,4}\}$ is a bad pair.
We assume that $a_{5}=1$.
We consider the condition that $4\in \im(\varphi_\lambda^J)$.
By taking $\Psi=\{\alpha_{1},\alpha_{6},-\ep_{3}+\ep_{5},\beta_{2},-\ep_{2}+\ep_{5}\}$ and the same argument as proof of Lemma \ref{n3lem}, we deduce that exactly one of $a_{3}=1$ or $a_{4}=0$ holds.
If $a_{4}=0$, we can see that $\{\gamma_{1,3},\gamma_{1,2}\}$ is a bad pair.
We suppose that $a_{3}=1$.
We investigate the condition that $5\in \im(\varphi_\lambda^J)$.
By taking $\Psi=\{\alpha_{1},\alpha_{6},-\ep_{3}+\ep_{5},\beta_{2},-\ep_{2}+\ep_{5},\beta_{5}\}$ and the same argument as proof of Lemma \ref{n3lem}, we see that $a_{4}=2$.
Then, we can check that $\{\gamma_{1,3},\gamma_{1,2}\}$ is a bad pair.
\end{proof}
\begin{prop}There are no initialized irreducible homogeneous Ulrich bundles on $E_{6}/P_{2,3}\cong E_{6}/P_{2,5}$.
\end{prop}
\begin{proof}Suppose for contradiction that there is an initialized irreducible homogeneous Ulrich bundle $E_{\lambda}$ on $E_{6}/P_{2,3}$ with $\lambda=\sum_{i=1}^{6}a_{i}\varpi_{i}$.
By the same argument as the proof of Lemma \ref{n1lem}, we deduce that exactly one of $a_{2}=0$ or $a_{3}=0$ holds.

Suppose that $a_{2}=0$.
We examine the condition that $2\in \im(\varphi_\lambda^J)$.
By taking $\Psi=\{\alpha_{2},\alpha_{3},\ep_{1}+\ep_{3}\}$ and the same argument as proof of Lemma \ref{n2lem}, we see that exactly one of $a_{4}=0$ or $a_{3}=1$ holds.
If $a_{4}=0$, we see that $\{\gamma_{3,4},\gamma_{2,4}\}$ is a bad pair.
Assume that $a_{3}=1$.
We turn to the condition that $3\in \im(\varphi_\lambda^J)$.
By the following Lemma \ref{2lem}, we have that exactly one of $a_{4}=2$ or $a_{1}=0$ holds.
If $a_{4}=2$ (resp. $a_{1}=0$), we can check that $\{\gamma_{3,4},\gamma_{2,4}\}$ (resp. $\{\ep_{4}+\ep_{5},\gamma_{2,3}\}$) is a bad pair.

Assume that $a_{3}=0$.
We examine the condition that $2\in \im(\varphi_\lambda^J)$.
By taking $\Psi=\{\alpha_{2},\alpha_{3},-\ep_{1}+\ep_{3},\beta_{2}\}$ and the same argument as proof of Lemma \ref{n2lem}, we see that exactly one of $a_{1}=0$ or $a_{4}=0$ or $a_{2}=1$ holds.
When $a_{1}=0$ (resp. $a_{4}=0$), we can check that $\{\gamma_{3,4},\gamma_{2,4}\}$ (resp. $\{\ep_{3}+\ep_{5},\gamma_{2,4}\}$) is a bad pair, respectively.
Let $a_{2}=1$.
We turn to the condition that $3\in \im(\varphi_\lambda^J)$.
By taking $\Psi=\{\alpha_{2},\alpha_{3},\ep_{2}+\ep_{3},\beta_{2}\}$ and the same argument as proof of Lemma \ref{2lem}, we see that exactly one of $a_{4}=2$ or $a_{1}=1$ holds.
When $a_{4}=2$, we can check that $\{\gamma_{3,4},\gamma_{2,4}\}$ is a bad pair.
We also assume that $a_{1}=1$.
We examine the condition that $4\in \im(\varphi_\lambda^J)$.
By taking $\Psi=\{\alpha_{2},\alpha_{3},\ep_{2}+\ep_{3},\beta_{2}\}$ and the same argument as proof of Lemma \ref{2lem}, we see that exactly one of $a_{4}=4$ holds.
Then we can check that $\{\gamma_{3,5},\gamma_{2,5}\}$ is a bad pair.
\end{proof}
\begin{lem}\label{2lem}Suppose that there is an initialized irreducible homogeneous Ulrich bundle $E_{\lambda}$ on $E_{6}/P_{2,3}$ with $\lambda=\sum_{i\in I}a_{i}\varpi_{i}$.
If $a_{2}=0$ and $a_{3}=1$, we obtain that exactly one of $a_{4}=2$ or $a_{1}=0$ holds.
\end{lem}
\begin{proof}Note that $a_{4}\geq2$.
Indeed, under this assumption, we have $\varphi_{\lambda}^{J}(\ep_{2}+\ep_{3})=\frac{4+a_{4}}{2}$.
Therefore, if $a_{4}=0$, we have $\varphi_{\lambda}^{J}(\ep_{2}+\ep_{3})=2$, so this contradicts injectivity of $\varphi_{\lambda}^{J}$. 
If $a_{4}=1$, we have $\varphi_{\lambda}^{J}(\ep_{2}+\ep_{3})=\frac{5}{2}$, so this contradicts $\im(\varphi_{\lambda}^{j})=[1,\dim(E_{6}/P_{2,3})]$.
Let $\Psi=\{\alpha_{2},\alpha_{3},\ep_{2}+\ep_{3},\beta_{2}\}$.
Then one can check by direct computation that for every $\alpha\in\Phi_{E_{6}^{+}}\setminus\Psi$ we have $\varphi_{\lambda}^{J}(\alpha)>3$.
Since $\varphi_{\lambda}^{J}$ is injective, we obtain that exactly one of $a_{4}=2$ or $a_{1}=0$ holds.
\end{proof}
\begin{prop}\label{prop24}There are no initialized irreducible homogeneous Ulrich bundles on $E_{6}/P_{2,4}$
\end{prop}
In order to prove this statement, we need the following lemma.
\begin{lem}\label{22lem}Suppose that there is an initialized irreducible homogeneous Ulrich bundle $E_{\lambda}$ on $E_{6}/P_{2,4}$ with $\lambda=\sum_{i\in I}a_{i}\varpi_{i}$.
If $a_{2}=0$ and $a_{4}=2$, we have that exactly one of $a_{3}=3$ or $a_{5}=3$ holds.
\end{lem}
\begin{proof}
Note that $\varphi_{\lambda}^{J}(-\ep_{2}+\ep_{3})=3$ in this case.
Thus, we investigate the condition that $4\in \im(\varphi_\lambda^J)$.
By $\varphi_{\lambda}^{J}(\ep_{1}+\ep_{4})=\frac{a_{5}+5}{2}$ and $\varphi_{\lambda}^{J}(\ep_{2}+\ep_{3})=\frac{a_{3}+5}{2}$, we have $a_{3}\geq 3$ and $a_{5} \geq 3$.
Indeed, if $a_{3}=0$, then we have $\varphi_{\lambda}^{J}(\ep_{2}+\ep_{3})=\frac{5}{2}$.
This contradicts $\im(\varphi_{\lambda}^{J})=[1,\dim(E_{6}/P_{2,4})]$.
If $a_{3}=1$, we have $\varphi_{\lambda}^{J}(\ep_{2}+\ep_{3})=3=\varphi_{\lambda}^{J}(-\ep_{2}+\ep_{3})$.
This contradicts injectivity of $\varphi_{\lambda}^{J}$.
If $a_{3}=2$, then we have $\varphi_{\lambda}^{J}(\ep_{2}+\ep_{3})=\frac{7}{2}$.
This contradicts $\im(\varphi_{\lambda}^{J})=[1,\dim(E_{6}/P_{2,4})]$.
By the same argument as above, we can check the case $a_{5}\geq3$.
Let $\Psi=\{\ep_{1}+\ep_{4},\ep_{2}+\ep_{3},-\ep_{2}+\ep_{3},\ep_{1}+\ep_{2},\ep_{1}+\ep_{3}\}$.
One can check by direct computation that for every $\alpha\in\Phi_{E_{6}}^{+}\setminus\Psi$ we have $\varphi_{\lambda}^{J}(\alpha)>4$.
So we get $\varphi_{\lambda}^{J}(\ep_{1}+\ep_{4})=4$ or $\varphi_{\lambda}^{J}(\ep_{2}+\ep_{3})=4$.
Since $\varphi_{\lambda}^{J}$ is injective, we obtain that exactly one of $a_{3}=3$ or $a_{5}=3$ holds.
\end{proof}
\begin{proof}[Proof of proposition \ref{prop24}]Suppose for contradiction that there is an initialized irreducible homogeneous Ulrich bundle $E_{\lambda}$ on $E_{6}/P_{2,4}$ with $\lambda=\sum_{i=1}^{6}a_{i}\varpi_{i}$.
By the same argument as the proof of Lemma \ref{n1lem}, we deduce that exactly one of $a_{2}=0$ or $a_{4}=0$ holds.

Assume that $a_{2}=0$.
We examine the condition that $2\in \im(\varphi_\lambda^J)$.
By taking $\Psi=\{\alpha_{2},\ep_{1}+\ep_{3}\}$ the same argument of the proof of Lemma \ref{2lem}, we
have $a_{4}=2$.
By Lemma \ref{22lem}, we obtain that exactly one of $a_{3}=3$ or $a_{5}=3$ holds.
If $a_{3}=3$ (resp. $a_{5}=3$), we see that $\{\gamma_{2,3},\gamma_{1,3}\}$ (resp. $\{\gamma_{1,4},\gamma_{1,3}\}$) is a bad pair.

Suppose that $a_{4}=0$.
We examine the condition that $2\in \im(\varphi_\lambda^J)$.
By taking $\Psi=\{\alpha_{4},\ep_{1}+\ep_{3},-\ep_{1}+\ep_{3},-\ep_{2}+\ep_{4}\}$ and Lemma \ref{2lem}, we have that exactly one of $a_{3}=0$ or $a_{5}=0$ or $a_{2}=2$ holds.
When $a_{3}=0$ (resp. $a_{5}=0$), we can check that $\{\gamma_{2,3},\gamma_{1,3}\}$ (resp. $\{\ep_{1}+\ep_{3},\ep_{1}+\ep_{4}\}$) is a bad pair.
Let $a_{2}=2$.
By $\varphi_{\lambda}^{J}(\alpha_{2})=3$ in this case, we turn to the condition that $4\in \im(\varphi_\lambda^J)$.
By taking $\Psi=\{\alpha_{4},\alpha_{2},\ep_{1}+\ep_{3},\ep_{1}+\ep_{4},\ep_{2}+\ep_{3}\}$ and the same argument as proof of Lemma \ref{2lem}, we see that exactly one of $a_{3}=3$ or $a_{5}=3$ holds.
Then, we can see that $\{\gamma_{2,4},\gamma_{2,3}\}$ ($\{\gamma_{2,3},\gamma_{1,3}\}$) is a bad pair, respectively.
\end{proof}
\begin{prop}There are no initialized irreducible homogeneous Ulrich bundles on $E_{6}/P_{3,4}\cong E_{6}/P_{4,5}$.
\end{prop}
\begin{proof}Suppose for contradiction that there is an initialized irreducible homogeneous Ulrich bundle $E_{\lambda}$ on $E_{6}/P_{3,4}$ with $\lambda=\sum_{i=1}^{6}a_{i}\varpi_{i}$.
By the same argument as the proof of Lemma \ref{n1lem}, we deduce that exactly one of $a_{3}=0$ or $a_{4}=0$ holds.

Suppose that $a_{3}=0$.
We examine the condition that $2\in \im(\varphi_\lambda^J)$.
By taking $\Psi=\{\alpha_{3},-\ep_{1}+\ep_{3},\beta_{2}\}$ and the same argument as proof of Lemma \ref{2lem}, we see that exactly one of $a_{4}=2$ or $a_{1}=0$ holds.
When $a_{1}=0$, we see that $\{\ep_{2}+\ep_{5},\gamma_{3,4}\}$ is a bad pair.
Let $a_{4}=2$.
We turn to the condition that $4\in \im(\varphi_\lambda^J)$.
By taking $\Psi=\{\alpha_{3},-\ep_{1}+\ep_{3},\beta_{3},\alpha_{4},-\ep_{1}+\ep_{4},\ep_{2}+\ep_{3}\}$ and the same argument as proof of Lemma \ref{22lem}, we see that exactly one of $a_{1}=3$ or $a_{2}=3$ or $a_{5}=3$ holds.
If $a_{1}=3$ (resp. $a_{2}=3$, $a_{5}=3$), we can check that $\{\ep_{3}+\ep_{5},\gamma_{2,4}\}$ (resp. $\{\gamma_{1,2},\gamma\}$, $\{\gamma_{2,4},\gamma_{2,3}\}$) is a bad pair.

Assume that $a_{4}=0$.
We examine the condition that $2\in \im(\varphi_\lambda^J)$.
By $\Psi=\{\alpha_{4},\ep_{1}+\ep_{3},-\ep_{1}+\ep_{3},-\ep_{2}+\ep_{4}\}$ and the same argument as proof of Lemma \ref{2lem}, we see that exactly one of $a_{3}=2$ or $a_{2}=0$ or $a_{5}=0$ holds.
When $a_{2}=0$ (resp. $a_{5}=0$), we see that $\{\ep_{2}+\ep_{3},-\ep_{1}+\ep_{3}\}$ (resp. $\{\ep_{2}+\ep_{3},\ep_{2}+\ep_{4}\}$) is a bad pair, respectively.
Let $a_{3}=2$.
We turn to the condition that $4\in \im(\varphi_\lambda^J)$.
By taking $\Psi=\{\alpha_{4},-\ep_{1}+\ep_{3},\alpha_{3},\ep_{2}+\ep_{3},-\ep_{1}+\ep_{4},\beta_{3}\}$ the same argument as proof of Lemma \ref{22lem}, we see that exactly one of $a_{1}=3$ or $a_{2}=3$ or $a_{5}=3$ holds.
If $a_{1}=3$ (resp. $a_{2}=3$, $a_{5}=3$), we can check that $\{\ep_{3}+\ep_{5},\gamma_{2,4}\}$ (resp. $\{\gamma_{1,2},\gamma\}$, $\{\ep_{3}+\ep_{5},\ep_{4}+\ep_{5}\}$) is a bad pair, respectively. 
\end{proof}
\begin{prop}There are no initialized irreducible homogeneous Ulrich bundles on $E_{6}/P_{3,5}$.
\end{prop}
\begin{proof}Suppose for contradiction that there is an initialized irreducible homogeneous Ulrich bundle $E_{\lambda}$ on $E_{6}/P_{3,5}$ with $\lambda=\sum_{i=1}^{6}a_{i}\varpi_{i}$.
By the same argument as the proof of Lemma \ref{n1lem}, we deduce that exactly one of $a_{3}=0$ or $a_{5}=0$ holds.

Suppose that $a_{3}=0$.
We examine the condition that $2\in \im(\varphi_\lambda^J)$.
By taking $\Psi=\{\alpha_{3},\alpha_{5},-\ep_{1}+\ep_{3},\beta_{2}\}$ and the same argument as proof of Lemma \ref{n2lem}, we see that exactly one of $a_{5}=1$ or $a_{4}=0$ or $a_{1}=0$ holds.
When $a_{4}=0$ (resp. $a_{1}=0$), we can check that $\{\ep_{2}+\ep_{5},\ep_{3}+\ep_{5}\}$ (resp. $\{-\ep_{1}+\ep_{5},\beta_{5}\}$) is a bad pair.
Let $a_{5}=1$.
We turn to the condition that $3\in \im(\varphi_\lambda^J)$.
By taking $\Psi=\{\alpha_{3},\alpha_{5},-\ep_{1}+\ep_{4},-\ep_{3}+\ep_{5},\beta_{2}\}$ and the same argument as proof of Lemma \ref{2lem}, we deduce that exactly one of $a_{4}=2$ or $a_{1}=1$ or $a_{6}=0$ holds.
If $a_{4}=2$ (resp. $a_{1}=1$, $a_{6}=0$), we see that $\{\ep_{2}+\ep_{5},\ep_{3}+\ep_{5}\}$ (resp. $\{\gamma_{2,3},\gamma_{1,4}\}$, $\{\ep_{2}+\ep_{4},\ep_{2}+\ep_{5}\}$) is a bad pair.

Assume that $a_{5}=0$.
We examine the condition that $2\in \im(\varphi_\lambda^J)$.
By taking $\Psi=\{\alpha_{5},\alpha_{3},-\ep_{2}+\ep_{4},-\ep_{3}+\ep_{5}\}$ and the same argument as proof of Lemma \ref{n2lem}, we see that exactly one of $a_{3}=1$ or $a_{4}=0$ or $a_{6}=0$ holds.
When $a_{4}=0$ (resp. $a_{6}=0$), we check that $\{\ep_{2}+\ep_{5},\ep_{3}+\ep_{5}\}$ (resp. $\{-\ep_{1}+\ep_{4},-\ep_{1}+\ep_{5}\}$) is a bad pair.
Let $a_{3}=1$.
We turn to the condition that $3\in \im(\varphi_\lambda^J)$.
By taking $\Psi=\{\alpha_{5},\alpha_{3},-\ep_{1}+\ep_{4},-\ep_{3}+\ep_{5},\beta_{2}\}$ and the same argument as proof of Lemma \ref{2lem}, we deduce that exactly one of $a_{4}=2$ or $a_{6}=1$ or $a_{1}=0$ holds.
If $a_{4}=2$ (resp. $a_{6}=1$, $a_{1}=0$), we see that $\{\ep_{2}+\ep_{5},\ep_{3}+\ep_{5}\}$ (resp. $\{\gamma_{2,3},\gamma_{1,4}\}$, $\{-\ep_{1}+\ep_{5},\beta_{5}\}$) is a bad pair.
\end{proof}
\subsection{Proof of theorem \ref{thmE6} when $|J|=3$}
\begin{prop}There are no initialized irreducible homogeneous Ulrich bundles on $E_{6}/P_{1,2,3}\cong E_{6}/P_{2,5,6}$.
\end{prop}
\begin{proof}Suppose for contradiction that there is an initialized irreducible homogeneous Ulrich bundle $E_{\lambda}$ on $E_{6}/P_{1,2,3}$ with $\lambda=\sum_{i=1}^{6}a_{i}\varpi_{i}$.
By the same argument as the proof of Lemma \ref{n1lem}, we deduce that exactly one of $a_{1}=0$ or $a_{2}=0$ or $a_{3}=0$ holds.

Assume that $a_{1}=0$.
We examine the condition that $2\in \im(\varphi_\lambda^J)$.
By taking $\Psi=\{\alpha_{1},\beta_{2},\alpha_{2}\}$ and the same argument as proof of Lemma \ref{2lem}, we see that exactly one of $a_{3}=2$ or $a_{2}=1$ holds.
When $a_{2}=1$, we can check that $\{\ep_{2}+\ep_{5},\beta_{5}\}$ is a bad pair.
Let $a_{3}=2$.
We turn to the condition that $4\in \im(\varphi_\lambda^J)$.
By taking $\Psi=\{\alpha_{1},\beta_{3},\alpha_{2},\alpha_{3},\beta_{2}\}$ and the same argument as proof of Lemma \ref{22lem}, we see that exactly one of $a_{4}=3$ or $a_{2}=3$ holds.
When $a_{4}=3$ (resp. $a_{2}=3$), we see that $\gamma_{3,4},\gamma_{2,4}\}$ (resp. $\{\ep_{2}+\ep_{5},\beta_{5}\}$) is a bad pair. 

Assume that $a_{2}=0$.
We examine the condition that $2\in \im(\varphi_\lambda^J)$.
By taking $\Psi=\{\alpha_{2},\ep_{1}+\ep_{3},\alpha_{3},\alpha_{1}\}$ and the same argument as proof of Lemma \ref{n2lem}, we see that exactly one of $a_{3}=1$ or $a_{1}=1$ or $a_{4}=0$ holds.
When $a_{1}=1$ (resp. $a_{4}=0$), we see that $\{\ep_{2}+\ep_{5},\beta_{5}\}$ (resp. $\{\beta_{2},\beta_{3}\}$) is a bad pair.
Let $a_{3}=1$.
We turn to the condition that $3\in \im(\varphi_\lambda^J)$.
By taking $\Psi=\{\alpha_{2},\ep_{1}+\ep_{3},\alpha_{3},\beta_{2}\}$ and the same argument as proof of Lemma \ref{2lem}, we see that exactly one of $a_{4}=2$ or $a_{1}=3$ holds.
If $a_{4}=2$ (resp. $a_{1}=3$), we can check that $\{\beta_{2},\beta_{3}\}$ (resp. $\{\ep_{2}+\ep_{5},\beta_{5}\}$) is a bad pair. 

Assume that $a_{3}=0$.
We examine the condition that $2\in \im(\varphi_\lambda^J)$.
By taking $\Psi=\{\alpha_{2},-\ep_{1}+\ep_{3},\alpha_{3},\beta_{2}\}$ and the same argument as proof of Lemma \ref{2lem}, we see that exactly one of $a_{1}=2$ or $a_{2}=1$ or $a_{4}=0$ holds.
If $a_{4}=0$, we can check that $\{\beta_{2},\beta_{3}\}$ is a bad pair.
Let $a_{1}=2$.
We turn to the condition that $4\in \im(\varphi_\lambda^J)$.
By taking $\Psi=\{\alpha_{2},\alpha_{3},\alpha_{1},\beta_{3},\beta_{2}\}$ and the same argument as proof of Lemma \ref{2lem}, we see that exactly one of $a_{2}=3$ or $a_{3}=3$ holds.
When $a_{2}=3$ (resp. $a_{4}=3$), we check that $\{\ep_{2}+\ep_{3},\beta_{3}\}$ (resp. $\{\gamma_{3,5},\gamma_{2,5}\}$) is a bad pair.
Let $a_{2}=1$.
We investigate the condition that $3\in \im(\varphi_\lambda^J)$.
By taking $\Psi=\{\alpha_{2},\alpha_{3},\ep_{2}+\ep_{3},\beta_{2}\}$ and the same argument as proof of Lemma \ref{2lem}, we see that exactly one of $a_{1}=4$ or $a_{4}=2$ holds.
When $a_{1}=4$ (resp. $a_{4}=2$), we see that $\{\ep_{2}+\ep_{3},\beta_{3}\}$ (resp. $\{\beta_{2},\beta_{3}\}$) is a bad pair.
\end{proof}

\begin{prop}There are no initialized irreducible homogeneous Ulrich bundles on $E_{6}/P_{1,2,4}\cong E_{6}/P_{2,4,6}$.
\end{prop}
\begin{proof}Suppose for contradiction that there is an initialized irreducible homogeneous Ulrich bundle $E_{\lambda}$ on $E_{6}/P_{1,2,4}$ with $\lambda=\sum_{i=1}^{6}a_{i}\varpi_{i}$.
By the same argument as the proof of Lemma \ref{n1lem}, we deduce that exactly one of $a_{1}=0$ or $a_{2}=0$ or $a_{4}=0$ holds.

Assume that $a_{1}=0$.
We examine the condition that $2\in \im(\varphi_\lambda^J)$.
By taking $\Psi=\{\alpha_{1},\alpha_{2},\alpha_{4},\beta_{2}\}$ and the same argument as proof of Lemma \ref{n2lem}, we see that exactly one of $a_{2}=1$ or $a_{4}=1$ or $a_{3}=0$ holds.
When $a_{2}=1$ (resp. $a_{4}=1,\ a_{3}=0$), we can check that $\{\ep_{2}+\ep_{5},\beta_{5}\}$ (resp. $\{\ep_{3}+\ep_{4},\gamma_{3,5}\}$, $\{\gamma_{2,4},\gamma_{1,4}\}$) is a bad pair.

Assume that $a_{2}=0$.
We turn to the condition that $2\in \im(\varphi_\lambda^J)$.
By taking $\Psi=\{\alpha_{2},\ep_{1}+\ep_{3},\alpha_{1}\}$ and the same argument as proof of Lemma \ref{n2lem}, we see that exactly one of $a_{4}=2$ or $a_{1}=1$ holds.
If $a_{1}=1$, we can check that $\{\ep_{2}+\ep_{3},\beta_{3}\}$ is a bad pair.
Let $a_{4}=2.$
We examine the condition that $4\in \im(\varphi_\lambda^J)$.
By taking $\Psi=\{\alpha_{2},\ep_{1}+\ep_{3},\alpha_{4},\ep_{1}+\ep_{4},\ep_{2}+\ep_{3},\alpha_{1}\}$ and the same argument as proof of Lemma \ref{n2lem}, we see that exactly one of $a_{1}=3$ or $a_{3}=3$ or $a_{5}=3$ holds.
If $a_{1}=3$ (resp. $a_{5}=3$), we can check that $\{\ep_{2}+\ep_{5},\beta_{5}\}$ (resp. $\{\ep_{3}+\ep_{5},\ep_{4}+\ep_{5}\}$) is a bad pair.
By the following Lemma \ref{lem4}, we have $a_{1}=6$.
However, $\varphi_{\lambda_{S}}^{J}(\beta_{3})=7=\varphi_{\lambda_{S}}^{J}(-\ep_{1}+\ep_{3})$ in this case (i.e. $S=\{1,2,3,4\}$ $n_{2}=0$, $n_{4}=2$, $n_{3}=3$ and $n_{1}=6$).
This contradicts injectivity of $\varphi_{\lambda}^{J}$.

Suppose that $a_{4}=0$.
We turn to the condition that $2\in \im(\varphi_\lambda^J)$.
By taking $\Psi=\{\alpha_{4},\ep_{1}+\ep_{3},-\ep_{2}+\ep_{4},\alpha_{1}\}$ and the same argument as proof of Lemma \ref{2lem}, we obtain that exactly one of $a_{2}=2$ or $a_{1}=1$ or $a_{3}=0$ or $a_{5}=0$ holds.
If $a_{3}=0$ (resp. $a_{5}=0$), we see that $\{\gamma_{2,5},\gamma_{1,5}\}$ (resp. $\{\ep_{2}+\ep_{3},\ep_{2}+\ep_{4}\}$) is a bad pair.
Now suppose that $a_{1}=1$.
We examine the condition that $3\in \im(\varphi_\lambda^J)$.
By taking $\Psi=\{\alpha_{4},\alpha_{1},\ep_{1}+\ep_{3},-\ep_{2}+\ep_{4},-\ep_{1}+\ep_{3}\}$ and the same argument as proof of Lemma \ref{n3lem}, we have that exactly one of $a_{2}=4$ or $a_{3}=2$ or $a_{5}=1$ holds.
When $a_{2}=4$ (resp. $a_{3}=2,\ a_{5}=1$), we obtain that $\{\ep_{2}+\ep_{3},\beta_{3}\}$ (resp. $\{\ep_{1}+\ep_{4},\ep_{2}+\ep_{4}\}$, $\{\gamma_{4,5},\gamma_{3,5}\}$) is a bad pair.
We assume that $a_{2}=2$.
We examine the condition that $4\in \im(\varphi_\lambda^J)$.
By taking $\Psi=\{\alpha_{4},\ep_{1}+\ep_{3},\alpha_{2},\ep_{1}+\ep_{4},\ep_{2}+\ep_{3},\alpha_{1}\}$ and the same argument as proof of Lemma \ref{n3lem}, we have that exactly one of $a_{1}=3$ or $a_{3}=3$ or $a_{5}=3$ holds.
If $a_{1}=3$ (resp. $a_{5}=3$), we see that $\{\ep_{2}+\ep_{5},\beta_{5}\}$ (resp. $\{\ep_{3}+\ep_{5},\ep_{4}+\ep_{5}\}$) is a bad pair.
Finally, assume that $a_{3}=3$.
We investigate the condition that $6\in \im(\varphi_\lambda^J)$.
By $\varphi_{\lambda}^{J}(\ep_{1}+\ep_{4})=\frac{5+a_{5}}{2}$ and $\varphi_{\lambda}^{J}(\ep_{3}+\ep_{4})=\frac{10+a_{5}}{3}$, we note that $a_{5}$ is congruent to $2$ modulo $3$ and $a_{5}\geq 8$ in this case.
Therefore, $a_{1}$ must be equal to $6$ by the same argument as proof of Lemma \ref{lem4}.
However, $\varphi_{\lambda}^{J}(-\ep_{1}+\ep_{3})=\varphi_{\lambda}^{J}(\gamma_{4,5})=5$.
Hence, this contradicts injectivity of $\varphi_{\lambda}^{J}$.
\end{proof}
\begin{lem}\label{lem4}Suppose that there is an initialized irreducible homogeneous Ulrich bundle $E_{\lambda}$ on $E_{6}/P_{1,2,4}$ with $\lambda=\sum_{i\in I}a_{i}\varpi_{i}$.
If $a_{2}=0$, $a_{4}=2$ and $a_{3}=3$, we have $a_{1}=6$.
\end{lem}
\begin{proof}By $\varphi_{\lambda}^{J}(\ep_{3}+\ep_{4})=\frac{12+a_{5}}{3}$ and $\varphi_{\lambda}^{J}(\ep_{2}+\ep_{4})=\frac{5+a_{5}}{2}$, we deduce that  $a_{5}$ is congruent to $0$ modulo $3$ and $a_{5}\geq 9$.
Furthermore, by $\varphi_{\lambda}^{J}(\gamma_{4,5})=\frac{9+a_{1}}{3}$ and $\varphi_{\lambda}^{J}(\beta_{3})=\frac{8+a_{1}}{2}$, we deduce that $a_{1}$ is congruent to $0$ modulo $6$ and $a_{1}\geq 6$.
Under those conditions, let us investigate the condition that $5\in \im(\varphi_\lambda^J)$.
Let $\Psi=\{\alpha_{2},\ep_{1}+\ep_{3},\ep_{2}+\ep_{3},-\ep_{2}+\ep_{3},\gamma_{4,5}\}$.
Then, one can check by direct computation that for every $\alpha\in\Phi_{E_{6}^{+}}\setminus\Psi$ we have $\varphi_{\lambda}^{J}(\alpha)>5$.
Therefore, we deduce that $\varphi_{\lambda}^{J}(\gamma_{4,5})=5$, so we have $a_{1}=6$.
\end{proof}

\begin{prop}There are no initialized irreducible homogeneous Ulrich bundles on $E_{6}/P_{1,2,5}\cong E_{6}/P_{2,3,6}$.
\end{prop}
\begin{proof}Suppose for contradiction that there is an initialized irreducible homogeneous Ulrich bundle $E_{\lambda}$ on $E_{6}/P_{1,2,5}$ with $\lambda=\sum_{i=1}^{6}a_{i}\varpi_{i}$.
By the same argument as the proof of Lemma \ref{n1lem}, we deduce that exactly one of $a_{1}=0$ or $a_{2}=0$ or $a_{5}=0$ holds.

Let us consider the case $a_{1}=0$.
We turn to the condition that $2\in \im(\varphi_\lambda^J)$.
By taking $\Psi=\{\alpha_{1},\alpha_{2},\alpha_{5},\beta_{2}\}$ and the same argument as proof of Lemma \ref{n2lem}, we see that exactly one of $a_{2}=1$ or $a_{5}=1$ or $a_{3}=0$ holds.
When $a_{2}=1$ (resp. $a_{5}=1,\ a_{3}=0$), we can check that $\{\ep_{2}+\ep_{4},\beta_{4}\}$ (resp. $\{\ep_{2}+\ep_{4},\gamma_{4,5}\}$, $\{\ep_{1}+\ep_{4},\ep_{2}+\ep_{4}\}$) is a bad pair.

Let us consider the case $a_{2}=0$.
We examine the condition that $2\in \im(\varphi_\lambda^J)$.
By taking $\Psi=\{\alpha_{2},\alpha_{1},\alpha_{5},\ep_{1}+\ep_{3}\}$ and the same argument as proof of Lemma \ref{n2lem}, we see that exactly one of $a_{1}=1$ or $a_{5}=1$ or $a_{4}=0$ holds.
When $a_{1}=1$ (resp. $a_{5}=1,\ a_{4}=0$), we see that $\{\ep_{2}+\ep_{4},\beta_{4}\}$ (resp. $\{\beta_{4},\gamma_{4,5}\}$, $\{\ep_{2}+\ep_{4},\ep_{3}+\ep_{4}\}$) is a bad pair.

Finally, assume that $a_{5}=0$.
We turn to the condition that $2\in \im(\varphi_\lambda^J)$.
By taking $\Psi=\{\alpha_{5},\alpha_{1},\alpha_{2},-\ep_{2}+\ep_{5},-\ep_{2}+\ep_{4}\}$ and the same argument as proof of Lemma \ref{n2lem}, we have that exactly one of $a_{1}=1$ or $a_{2}=1$ or $a_{4}=0$ or $a_{6}=0$ holds.
When $a_{1}=1$ (resp. $a_{2}=1,\ a_{4}=0,\ a_{6}=0$), we see that $\{\ep_{2}+\ep_{4},\gamma_{4,5}\}$ (resp. $\{\beta_{4},\gamma_{4,5}\}$, $\{\ep_{2}+\ep_{4},\ep_{3}+\ep_{4}\}$, $\{\ep_{1}+\ep_{4},\ep_{1}+\ep_{5}\}$) is a bad pair.
\end{proof}
\begin{prop}There are no initialized irreducible homogeneous Ulrich bundles on $E_{6}/P_{1,2,6}$.
\end{prop}
\begin{proof}Suppose for contradiction that there is an initialized irreducible homogeneous Ulrich bundle $E_{\lambda}$ on $E_{6}/P_{1,2,6}$ with $\lambda=\sum_{i=1}^{6}a_{i}\varpi_{i}$.
By the same argument as the proof of Lemma \ref{n1lem}, we deduce that exactly one of $a_{1}=0$ or $a_{2}=0$ or $a_{6}=0$ holds.

Let us consider the case $a_{1}=0$.
We turn to the condition that $2\in \im(\varphi_\lambda^J)$.
By taking $\Psi=\{\alpha_{1},\beta_{2},\alpha_{6},\alpha_{2}\}$ and the same argument as proof of Lemma \ref{n2lem}, we obtain that exactly one of $a_{2}=1$ or $a_{6}=1$ or $a_{3}=0$ holds.
When $a_{2}=1$ (resp. $a_{6}=1,\ a_{3}=0$), we see that $\{\ep_{2}+\ep_{5},\beta_{5}\}$ (resp. $\{\ep_{2}+\ep_{5},\gamma_{3,5}\}$, $\{\ep_{1}+\ep_{5},\ep_{2}+\ep_{5}\}$) is a bad pair.

Suppose that $a_{2}=0$.
We examine the condition that $2\in \im(\varphi_\lambda^J)$.
By taking $\Psi=\{\alpha_{2},\alpha_{1},\alpha_{6},\ep_{1}+\ep_{3}\}$ and the same argument as proof of Lemma \ref{n2lem}, we have that exactly one of $a_{1}=1$ or $a_{6}=1$ or $a_{4}=0$ holds.
If $a_{1}=1$ (resp. $a_{6}=1,\ a_{4}=0$), we can check that $\{\ep_{2}+\ep_{5},\beta_{5}\}$ (resp. $\{\beta_{5},\gamma_{3,5}\}$, $\{\ep_{2}+\ep_{5},\ep_{3}+\ep_{5}\}$) is a bad pair.

Assume that $a_{6}=0$.
We investigate the condition that $2\in \im(\varphi_\lambda^J)$.
By taking $\Psi=\{\alpha_{6},\alpha_{1},\alpha_{2},-\ep_{3}+\ep_{5}\}$ and the same argument as proof of Lemma \ref{n2lem}, we have that exactly one of $a_{1}=1$ or $a_{2}=1$ or $a_{5}=0$ holds.
When $a_{1}=1$ (resp. $a_{2}=1,\ a_{5}=0$), we see that $\{\ep_{2}+\ep_{5},\gamma_{3,5}\}$ (resp. $\{\beta_{5},\gamma_{3,5}\}$, $\{\ep_{3}+\ep_{5},\ep_{4}+\ep_{5}\}$) is a bad pair.
\end{proof}

\begin{prop}There are no initialized irreducible homogeneous Ulrich bundles on $E_{6}/P_{1,3,4}\cong E_{6}/P_{4,5,6}$.
\end{prop}
\begin{proof}Suppose for contradiction that there is an initialized irreducible homogeneous Ulrich bundle $E_{\lambda}$ on $E_{6}/P_{1,3,4}$ with $\lambda=\sum_{i=1}^{6}a_{i}\varpi_{i}$.
By the same argument as the proof of Lemma \ref{n1lem}, we deduce that exactly one of $a_{1}=0$ or $a_{3}=0$ or $a_{4}=0$ holds.

Let us consider the case $a_{1}=0$. 
We turn to the condition that $2\in \im(\varphi_\lambda^J)$.
By taking $\Psi=\{\alpha_{1},\beta_{2},\alpha_{3}\}$ and the same argument as proof of Lemma \ref{2lem}, we obtain that exactly one of $a_{3}=2$ or $a_{4}=1$ holds.
If $a_{4}=1$, we see that $\{\ep_{3}+\ep_{4},\gamma_{3,5}\}$ is a bad pair.
If $a_{3}=2$, this case contradicts our hypothesis by the following Lemma \ref{lem5}.

Assume that $a_{3}=0$.
We examine the condition that $2\in \im(\varphi_\lambda^J)$.
By taking $\Psi=\{\alpha_{3},\beta_{2},-\ep_{1}+\ep_{3}\}$ and the same argument as proof of Lemma \ref{2lem}, we have that exactly one of $a_{1}=2$ or $a_{4}=2$ holds.
Suppose that $a_{1}=2$.
Let $\Psi=\{\alpha_{3},\alpha_{1},\beta_{2}\}$ and $\mu=2\cdot\varpi_{1}+0\cdot\varpi_{3}$.
In this case for every $\alpha\in\Phi_{1,3,4}^{+}\setminus\Psi$, we can check that $\varphi_{\mu}^{J}(\alpha)$ contains $a_{4}$.
By taking above $\Psi$, $\mu$ and the same argument as proof of Lemma \ref{lem5}, we deduce that $E_{\mu}$ is not Ulrich bundle, so this contradicts our hypothesis.
On the other hand, assume that $a_{4}=2$.
We turn to the condition that $4\in \im(\varphi_\lambda^J)$.
By $\varphi_{\lambda}^{J}(\beta_{2})=\frac{a_{1}+2}{2}$, and $\varphi_{\lambda}^{J}(\beta_{3})=\frac{a_{1}+5}{3}$ and the same argument as the proof of Lemma \ref{22lem}, we deduce that $a_{1}$ is grater than or equal to $10$.
By taking $\Psi=\{\alpha_{3},-\ep_{1}+\ep_{3},\alpha_{4},\ep_{2}+\ep_{3},-\ep_{1}+\ep_{4}\}$ and the same argument as proof of Lemma \ref{2lem}, we have that exactly one of $a_{2}=3$ or $a_{5}=3$.
If $a_{2}=3$ (resp. $a_{5}=3$), we see that $\{\gamma_{4,5},\beta_{3}\}$ (resp. $\{\ep_{3}+\ep_{5},\ep_{4}+\ep_{5}\}$) is a bad pair. 

Suppose that $a_{4}=0$.
We examine the condition that $2\in \im(\varphi_\lambda^J)$.
By taking $\Psi=\{\alpha_{4},\alpha_{1},-\ep_{1}+\ep_{3},-\ep_{2}+\ep_{4},\ep_{1}+\ep_{3}\}$ and the same argument as proof of Lemma \ref{2lem}, we have that exactly one of $a_{3}=2$ or $a_{1}=1$ or $a_{2}=0$ or $a_{5}=0$ holds.
Now assume that $a_{3}=2$.
We turn to the condition that $4\in \im(\varphi_\lambda^J)$.
By $\varphi_{\lambda}^{J}(\beta_{2})=\frac{a_{1}+4}{2}$, and $\varphi_{\lambda}^{J}(\beta_{3})=\frac{a_{1}+5}{3}$ and the same argument as the proof of Lemma \ref{22lem}, we deduce that $a_{1}$ is grater than or equal to $10$.
By taking $\Psi=\{\alpha_{4},\ep_{2}+\ep_{3},-\ep_{1}+\ep_{3},\alpha_{3},-\ep_{1}+\ep_{4}\}$ and the same argument as proof of Lemma \ref{22lem}, we have that exactly one of $a_{2}=3$ or $a_{5}=3$.
If $a_{2}=3$ (resp. $a_{5}=3$), we see that $\{\gamma_{4,5},\beta_{3}\}$ (resp. $\{\beta_{3},\beta_{4}\}$) is a bad pair.
If $a_{1}=1$ (resp. $a_{2}=0,\ a_{5}=0$), we can check that $\{\ep_{3}+\ep_{4},\gamma_{3,5}\}$ (resp. $\{\ep_{2}+\ep_{3},-\ep_{1}+\ep_{3}\}$, $\{\ep_{2}+\ep_{3},\ep_{2}+\ep_{4}\}$) is a bad pair.
\end{proof}
\begin{lem}\label{lem5}Let $E_{\mu}$ be an initialized irreducible homogeneous vector bundle on $E_{6}/P_{1,3,5}$ with $\mu=a_{3}\varpi_{3}+a_{4}\varpi_{4}+a_{5}\varpi_{5}+a_{6}\varpi_{6}+2\cdot\varpi_{2}$.
Then $E_{\mu}$ is not an Ulrich bundle.
\end{lem}
\begin{proof}
By $\varphi_{\mu}^{J}(\beta_{3})=\frac{5+a_{4}}{3}$, and $\varphi_{\mu}^{J}(-\ep_{1}+\ep_{3})=\frac{4+a_{4}}{2}$, and the same argument as the proof of Lemma \ref{22lem}, we deduce that $a_{4}$ is congruent to $1$ modulo $3$ and $a_{4}\geq10.$
Let $\Psi=\{\alpha_{1},\beta_{2},-\ep_{1}+\ep_{2}\}$.
Moreover, for every $\alpha\in\Phi_{1,3,5}^{+}\setminus\Psi$, we see that the coefficients of $a_{4}$ in $\varphi_{\mu}$ are not equal to $0$.
Therefore, there is not a positive root $\beta$ such that $4=\varphi_{\mu}^{J}(\beta)$.
This contradicts $\im\varphi_{\mu}^{J}=[1,\dim(E_{6}/P_{1,3,5})]$.
\end{proof}
\begin{prop}There are no initialized irreducible homogeneous Ulrich bundles on $E_{6}/P_{1,3,5}\cong E_{6}/P_{3,5,6}$.
\end{prop}
\begin{proof}Suppose for contradiction that there is an initialized irreducible homogeneous Ulrich bundle $E_{\lambda}$ on $E_{6}/P_{1,3,5}$ with $\lambda=\sum_{i=1}^{6}a_{i}\varpi_{i}$.
By the same argument as the proof of Lemma \ref{n1lem}, we deduce that exactly one of $a_{1}=0$ or $a_{3}=0$ or $a_{5}=0$ holds.

Assume that $a_{1}=0$.
We turn to the condition that $2\in \im(\varphi_\lambda^J)$.
By taking $\Psi=\{\alpha_{1},\beta_{2},\alpha_{5}\}$ and the same argument as proof of Lemma \ref{2lem}, we obtain that exactly one of $a_{3}=2$ or $a_{5}=1$ holds.
If $a_{5}=1$, we can check that $\{-\ep_{1}+\ep_{4},\beta_{3}\}$ is a bad pair.
We also assume that $a_{3}=2$.
We examine the condition that $4\in \im(\varphi_\lambda^J)$.
By taking $\Psi=\{\alpha_{1},\beta_{2},\alpha_{3},\beta_{3},\alpha_{5}\}$ and the same argument as proof of Lemma \ref{22lem}, we have that exactly one of $a_{4}=3$ or $a_{5}=3$ holds.
If $a_{4}=3$ (resp. $a_{5}=3$), we see that $\{\gamma_{3,5},\gamma_{2,5}\}$ (resp. $\{-\ep_{1}+\ep_{4},\beta_{3}\}$) is a bad pair.

Suppose that $a_{3}=0$.
We turn to the condition that $2\in \im(\varphi_\lambda^J)$.
By taking $\Psi=\{\alpha_{3},\beta_{2},\alpha_{5},-\ep_{1}+\ep_{3}\}$ and the same argument as proof of Lemma \ref{2lem}, we obtain that exactly one of $a_{1}=2$ or $a_{5}=1$ or $a_{4}=0$ holds.
If $a_{5}=1$ (resp. $a_{4}=0$), we see that $\{\gamma_{2,3},\gamma_{1,4}\}$ (resp. $\{\ep_{2}+\ep_{5},\ep_{3}+\ep_{5}\}$) is a bad pair.
When $a_{1}=2$, we examine the condition that $4\in \im(\varphi_\lambda^J)$.
By taking $\Psi=\{\alpha_{3},\beta_{2},\alpha_{1},\beta_{3},\alpha_{5}\}$ and the same argument as proof of Lemma \ref{22lem}, we have that exactly one of $a_{4}=3$ or $a_{5}=3$ holds.
If $a_{4}=3$ (resp. $a_{5}=3$), we can check that $\{\gamma_{1,3},\gamma_{1,2}\}$ (resp. $\{\ep_{4}+\ep_{5},\gamma_{2,4}\}$) is a bad pair.

Assume that $a_{5}=0$.
We turn to the condition that $2\in \im(\varphi_\lambda^J)$.
By taking $\Psi=\{\alpha_{5},\alpha_{1},-\ep_{3}+\ep_{5},-\ep_{2}+\ep_{4},\alpha_{3}\}$ and the same argument as proof of Lemma \ref{n2lem}, we get that exactly one of $a_{1}=1$ or $a_{3}=1$ or $a_{4}=0$ or $a_{6}=0$ holds.
If $a_{1}=1$ (resp. $a_{3}=1,\ a_{4}=0,\ a_{6}=0$), we see that $\{\ep_{4}+\ep_{5},\gamma_{2,4}\}$ (resp. $\{\gamma_{2,3},\gamma_{1,4}\}$, $\{\ep_{2}+\ep_{5},\ep_{3}+\ep_{5}\}$, $\{\ep_{2}+\ep_{4},\ep_{2}+\ep_{5}\}$) is a bad pair.
\end{proof}
\begin{prop}There are no initialized irreducible homogeneous Ulrich bundles on $E_{6}/P_{1,3,6}\cong E_{6}/P_{1,5,6}$.
\end{prop}
\begin{proof}Suppose for contradiction that there is an initialized irreducible homogeneous Ulrich bundle $E_{\lambda}$ on $E_{6}/P_{1,3,6}$ with $\lambda=\sum_{i=1}^{6}a_{i}\varpi_{i}$.
By the same argument as the proof of Lemma \ref{n1lem}, we deduce that exactly one of $a_{1}=0$ or $a_{3}=0$ or $a_{6}=0$ holds.

Let us consider the case $a_{1}=0$.
We turn to the condition that $2\in \im(\varphi_\lambda^J)$.
By taking $\Psi=\{\alpha_{1},\beta_{2},\alpha_{6}\}$ and the same argument as proof of Lemma \ref{2lem}, we obtain that exactly one of $a_{3}=2$ or $a_{6}=1$ holds.
If $a_{6}=1$, we see that $\{-\ep_{2}+\ep_{5},\gamma_{3,5}\}$ is a bad pair.
When $a_{3}=2$, we examine the condition that $4\in \im(\varphi_\lambda^J)$.
By taking $\Psi=\{\alpha_{1},\beta_{2},\alpha_{3},\beta_{3},\alpha_{6}\}$ and the same argument as proof of Lemma \ref{22lem}, we have that exactly one of $a_{4}=3$ or $a_{6}=3$ holds.
If $a_{4}=3$ (resp. $a_{6}=3$), we can check that $\{\gamma_{2,4},\gamma_{3,4}\}$ (resp. $\{-\ep_{1}+\ep_{5},\beta_{4}\}$) is a bad pair. 

Suppose that $a_{3}=0$.
We turn to the condition that $2\in \im(\varphi_\lambda^J)$.
By taking $\Psi=\{\alpha_{3},-\ep_{1}+\ep_{3},\alpha_{6},\beta_{2}\}$ and the same argument as proof of Lemma \ref{2lem}, we get that exactly one of $a_{1}=2$ or $a_{6}=1$ or $a_{4}=0$ holds.
If $a_{6}=1$ (resp. $a_{4}=0$), we see that $\{\gamma_{2,4},\gamma_{1,5}\}$ (resp. $\{\ep_{2}+\ep_{5},\ep_{3}+\ep_{5}\}$) is a bad pair.
When $a_{1}=2$, we examine the condition that $4\in \im(\varphi_\lambda^J)$.
By taking $\Psi=\{\alpha_{3},\alpha_{1},\beta_{2},\beta_{3},\alpha_{6}\}$ and the same argument as proof of Lemma \ref{22lem}, we have that exactly one of $a_{4}=3$ or $a_{6}=3$ holds.
If $a_{4}=3$ (resp. $a_{6}=3$), we can check that $\{\gamma_{3,4},\gamma_{2,4}\}$ (resp. $\{\beta_{4},-\ep_{1}+\ep_{5}\}$) is a bad pair.

Assume that $a_{6}=0$.
We turn to the condition that $2\in \im(\varphi_\lambda^J)$.
By taking $\Psi=\{\alpha_{6},\alpha_{1},-\ep_{3}+\ep_{5},\alpha_{3},\}$ and the same argument as proof of Lemma \ref{n2lem}, we have that exactly one of $a_{1}=1$ or $a_{3}=1$ or $a_{5}=0$ holds.
If $a_{1}=1$ (resp. $a_{3}=1,\ a_{5}=0$), we see that $\{-\ep_{1}+\ep_{5},\beta_{4}\}$ (resp. $\{\gamma_{2,4},\gamma_{1,5}\}$, $\{\ep_{3}+\ep_{5},\ep_{4}+\ep_{5}\}$) is a bad pair.
\end{proof}
\begin{prop}There are no initialized irreducible homogeneous Ulrich bundles on $E_{6}/P_{1,4,5}\cong E_{6}/P_{3,4,6}$.
\end{prop}
\begin{proof}Suppose for contradiction that there is an initialized irreducible homogeneous Ulrich bundle $E_{\lambda}$ on $E_{6}/P_{1,4,5}$ with $\lambda=\sum_{i=1}^{6}a_{i}\varpi_{i}$.
By the same argument as the proof of Lemma \ref{n1lem}, we deduce that exactly one of $a_{1}=0$ or $a_{4}=0$ or $a_{5}=0$ holds.

Suppose that $a_{1}=0$.
We turn to the condition that $2\in \im(\varphi_\lambda^J)$.
By taking $\Psi=\{\alpha_{1},\beta_{2},\alpha_{4},\alpha_{5}\}$ and the same argument as proof of Lemma \ref{n2lem}, we have that exactly one of $a_{4}=1$ or $a_{5}=1$ or $a_{3}=0$ holds.
If $a_{4}=1$ (resp. $a_{5}=1,\ a_{3}=0$), we see that $\{\ep_{3}+\ep_{5},\gamma_{3,4}\}$ (resp. $\{\ep_{2}+\ep_{4},\gamma_{4,5}\}$, $\{\ep_{1}+\ep_{5},\ep_{2}+\ep_{5}\}$) is a bad pair.

Assume that $a_{4}=0$.
We turn to the condition that $2\in \im(\varphi_\lambda^J)$.
By taking $\Psi=\{\alpha_{4},\alpha_{1},-\ep_{2}+\ep_{4},-\ep_{1}+\ep_{3},\ep_{1}+\ep_{3}\}$ and the same argument as proof of Lemma \ref{2lem}, we have that exactly one of $a_{5}=2$ or $a_{1}=1$ or $a_{2}=0$ or $a_{3}=0$ holds.
If $a_{1}=1$ (resp. $a_{2}=0,\ a_{3}=0$), we check that $\{\ep_{3}+\ep_{4},\gamma_{3,5}\}$ (resp. $\{\beta_{4},\gamma_{3,5}\}$, $\{\ep_{1}+\ep_{5},\ep_{2}+\ep_{5}\}$) is a bad pair.
When $a_{5}=2$, we examine the condition that $4\in \im(\varphi_\lambda^J)$.
By taking $\Psi=\{\alpha_{4},-\ep_{2}+\ep_{4},\alpha_{5},\ep_{1}+\ep_{4},-\ep_{1}+\ep_{4},-\ep_{2}+\ep_{5},\alpha_{1}\}$ and the same argument as proof of Lemma \ref{22lem}, we have that exactly one of $a_{1}=3$ or $a_{2}=3$ or $a_{3}=3$ or $a_{6}=3$ holds.
If $a_{1}=3$ (resp. $a_{2}=3,\ a_{3}=3,\ a_{6}=3$), we can check that $\{\ep_{2}+\ep_{4},\gamma_{4,5}\}$ (resp. $\{\gamma_{3,5},\beta_{4}\}$, $\{\gamma_{2,3},\gamma_{1,3}\}$, $\{\ep_{3}+\ep_{4},\ep_{3}+\ep_{5}\}$) is a bad pair.

Suppose that $a_{5}=0$.
We turn to the condition that $2\in \im(\varphi_\lambda^J)$.
By taking $\Psi=\{\alpha_{5},-\ep_{2}+\ep_{4},-\ep_{3}+\ep_{5},\alpha_{1}\}$ and the same argument as proof of Lemma \ref{2lem}, we have that exactly one of $a_{4}=2$ or $a_{1}=1$ or $a_{6}=0$ holds.
If $a_{1}=1$ (resp. $a_{6}=0$), we see that $\{\beta_{3},-\ep_{1}+\ep_{4}\}$ (resp. $\{\ep_{3}+\ep_{4},\ep_{3}+\ep_{5}\}$) is a bad pair.
When $a_{4}=2$,  we examine the condition that $4\in \im(\varphi_\lambda^J)$.
By taking $\Psi=\{\alpha_{5},-\ep_{2}+\ep_{4},\alpha_{4},-\ep_{2}+\ep_{5},\alpha_{1},\ep_{1}+\ep_{4},-\ep_{1}+\ep_{4}\}$ and the same argument as proof of Lemma \ref{22lem}, we have that exactly one of $a_{1}=3$ or $a_{2}=3$ or $a_{3}=3$ or $a_{6}=3$. holds.
If $a_{1}=3$ (resp. $a_{2}=3,\ a_{3}=3,\ a_{6}=3$), we can check that $\{\ep_{3}+\ep_{5},\gamma_{3,4}\}$ (resp. $\{\gamma,\gamma_{1,2}\}$, $\{\gamma_{2,3},\gamma_{1,3}\}$, $\{\ep_{3}+\ep_{4},\ep_{3}+\ep_{5}\}$) is a bad pair.
\end{proof}
\begin{prop}There are no initialized irreducible homogeneous Ulrich bundles on $E_{6}/P_{1,4,6}$.
\end{prop}
\begin{proof}Suppose for contradiction that there is an initialized irreducible homogeneous Ulrich bundle $E_{\lambda}$ on $E_{6}/P_{1,4,6}$ with $\lambda=\sum_{i=1}^{6}a_{i}\varpi_{i}$.
By the same argument as the proof of Lemma \ref{n1lem}, we deduce that exactly one of $a_{1}=0$ or $a_{4}=0$ or $a_{6}=0$ holds.

Suppose that $a_{1}=0$.
We turn to the condition that $2\in \im(\varphi_\lambda^J)$.
By taking $\Psi=\{\alpha_{1},\beta_{2},\alpha_{4},\alpha_{6}\}$ and the same argument as proof of Lemma \ref{n2lem}, we have that exactly one of $a_{4}=1$ or $a_{6}=1$ or $a_{3}=0$ holds.
If $a_{4}=1$ (resp. $a_{6}=1,\ a_{3}=0$), we can check that $\{\ep_{3}+\ep_{4},\gamma_{3,5}\}$ (resp. $\{-\ep_{1}+\ep_{5},\beta_{4}\}$, $\{\gamma_{2,5},\gamma_{1,5}\}$) is a bad pair.

Assume that $a_{4}=0$.
We examine the condition that $2\in \im(\varphi_\lambda^J)$.
By taking $\Psi=\{\alpha_{4},\alpha_{1},\alpha_{6},\ep_{1}+\ep_{3},-\ep_{1}+\ep_{3},-\ep_{2}+\ep_{4}\}$ and the same argument as proof of Lemma \ref{n2lem}, we obtain that exactly one of $a_{1}=1$ or $a_{6}=1$ or $a_{2}=0$ or $a_{3}=0$ or $a_{5}=0$ holds.
If $a_{1}=1$ (resp. $a_{6}=1,\ a_{2}=0,\ a_{3}=0,\ a_{5}=0$), we see that $\{\ep_{3}+\ep_{4},\gamma_{3,5}\}$ (resp. $\{\gamma_{3,4},\gamma_{2,5}\}$, $\{-\ep_{1}+\ep_{5},\ep_{2}+\ep_{5}\}$, $\{\ep_{1}+\ep_{5},\ep_{2}+\ep_{5}\}$, $\{\ep_{3}+\ep_{5},\ep_{4}+\ep_{5}\}$) is a bad pair.

Suppose that $a_{6}=0$.
We turn to the condition that $2\in \im(\varphi_\lambda^J)$.
By taking $\Psi=\{\alpha_{6},\alpha_{1},\alpha_{4},-\ep_{3}+\ep_{5}\}$ and the same argument as proof of Lemma \ref{n2lem}, we obtain that exactly one of $a_{1}=1$ or $a_{4}=1$ or $a_{5}=0$ holds.
If $a_{1}=1$ (resp. $a_{4}=1,\ a_{5}=0$), we can check that $\{-\ep_{1}+\ep_{5},\beta_{4}\}$ (resp. $\{\ep_{2}+\ep_{5},\ep_{3}+\ep_{4}\}$, $\{\ep_{3}+\ep_{5},\ep_{4}+\ep_{5}\}$) is a bad pair.
\end{proof}
\begin{prop}There are no initialized irreducible homogeneous Ulrich bundles on $E_{6}/P_{2,3,4}\cong E_{6}/P_{2,4,5}$.
\end{prop}
\begin{proof}Suppose for contradiction that there is an initialized irreducible homogeneous Ulrich bundle $E_{\lambda}$ on $E_{6}/P_{2,3,4}$ with $\lambda=\sum_{i=1}^{6}a_{i}\varpi_{i}$.
By the same argument as the proof of Lemma \ref{n1lem}, we deduce that exactly one of $a_{2}=0$ or $a_{3}=0$ or $a_{4}=0$ holds.

Suppose that $a_{2}=0$.
We turn to the condition that $2\in \im(\varphi_\lambda^J)$.
By taking $\Psi=\{\alpha_{2},\ep_{1}+\ep_{3},\alpha_{3},\alpha_{4}\}$ and the same argument as proof of Lemma \ref{2lem}, we obtain that exactly one of $a_{4}=2$ or $a_{3}=1$ holds.
If $a_{3}=1$, we see that $\{\ep_{1}+\ep_{3},-\ep_{1}+\ep_{3}\}$ is a bad pair.
We also assume that $a_{4}=2$.
We examine the condition that $4\in \im(\varphi_\lambda^J)$.
By $\varphi_{\lambda}^{J}(-\ep_{1}+\ep_{3})=\frac{4+a_{3}}{2}$, and $\varphi_{\lambda}^{J}(\ep_{2}+\ep_{3})=\frac{5+a_{3}}{3}$ and the same argument as the proof of Lemma \ref{22lem}, we deduce that $a_{3}$ is grater than or equal to $10$ in this case.
Therefore, we have $a_{5}=3$ by taking $\Psi=\{\alpha_{2},\ep_{1}+\ep_{3},\alpha_{4},\ep_{1}+\ep_{4}\}$ and the same argument as proof of Lemma \ref{lem5}.
If $a_{5}=3$, we see that $\{\ep_{2}+\ep_{3},\ep_{2}+\ep_{4}\}$ is a bad pair.

Assume that $a_{3}=0$.
We turn to the condition that $2\in \im(\varphi_\lambda^J)$.
By taking $\Psi=\{\alpha_{3},\alpha_{2},-\ep_{1}+\ep_{3},\beta_{2}\}$ and the same argument as proof of Lemma \ref{2lem}, we obtain that exactly one of $a_{4}=2$ or $a_{2}=1$ or $a_{1}=0$ holds.
If $a_{2}=1$ (resp. $a_{1}=0$), we check that $\{\ep_{1}+\ep_{5},-\ep_{1}+\ep_{5}\}$ (resp. $\{-\ep_{1}+\ep_{5},\beta_{5}\}$) is a bad pair.
We also suppose that $a_{4}=2$.
We investigate the condition that $4\in \im(\varphi_\lambda^J)$.
By $\varphi_{\lambda}^{J}(\ep_{1}+\ep_{3})=\frac{a_{2}+4}{2}$, and $\varphi_{\lambda}^{J}(\ep_{2}+\ep_{3})=\frac{a_{2}+5}{3}$ and the same argument as the proof of Lemma \ref{22lem}, we deduce that $a_{2}$ is grater than or equal to $10$.
In this case we have that exactly one $a_{1}=3$ or $a_{5}=3$ holds by taking $\Psi=\{\alpha_{3},-\ep_{1}+\ep_{3},\alpha_{4},-\ep_{1}+\ep_{4},\beta_{2}\}$ and the same argument as the proof of Lemma \ref{n3lem}.
If $a_{1}=3$ (resp. $a_{5}=3$), we can check that $\{\ep_{2}+\ep_{4},\gamma_{3,5}\}$ (resp. $\{\ep_{2}+\ep_{3},\ep_{2}+\ep_{4}\}$) is a bad pair.

Let us consider the case $a_{4}=0$.
We turn to the condition that $2\in \im(\varphi_\lambda^J)$.
By taking $\Psi=\{\alpha_{4},\ep_{1}+\ep_{3},-\ep_{1}+\ep_{3},-\ep_{2}+\ep_{4}\}$ and the same argument as proof of Lemma \ref{2lem}, we obtain that exactly one of $a_{2}=2$ or $a_{3}=2$ or $a_{5}=0$ holds.
If $a_{5}=0$, we see that $\{\ep_{2}+\ep_{3},\ep_{2}+\ep_{4}\}$ is a bad pair.
Suppose that $a_{2}=2$.
We investigate the condition that $4\in \im(\varphi_\lambda^J)$.
By $\varphi_{\lambda}^{J}(-\ep_{1}+\ep_{3})=\frac{a_{3}+2}{2}$, and $\varphi_{\lambda}^{J}(\ep_{2}+\ep_{3})=\frac{a_{3}+5}{3}$ and the same argument as the proof of Lemma \ref{22lem}, we deduce that $a_{3}$ is grater than or equal to $10$.
In this case, we have $a_{5}=3$ by taking $\Psi=\{\alpha_{4},\ep_{1}+\ep_{3},\alpha_{2},\ep_{1}+\ep_{4}\}$ and the same argument as the proof of Lemma \ref{n3lem}.
If $a_{5}=3$, we see that $\{\ep_{2}+\ep_{3},\ep_{2}+\ep_{4}\}$ is a bad pair.
Assume that $a_{3}=2$.
We examine the condition that $4\in \im(\varphi_\lambda^J)$.
By $\varphi_{\lambda}^{J}(\ep_{1}+\ep_{3})=\frac{a_{2}+2}{2}$, and $\varphi_{\lambda}^{J}(\ep_{2}+\ep_{3})=\frac{a_{2}+5}{3}$ and the same argument as the proof of Lemma \ref{22lem}, we deduce that $a_{2}$ is grater than or equal to $10$.
In this case, we have that exactly one of $a_{1}=3$ or $a_{5}=3$ holds by taking $\Psi=\{\alpha_{4},-\ep_{1}+\ep_{3},\alpha_{3},-\ep_{1}+\ep_{4},\beta_{3}\}$ and the same argument as the proof of Lemma \ref{n3lem}.
If $a_{1}=3$ (resp. $a_{5}=3$), we can check that $\{\ep_{2}+\ep_{4},\gamma_{3,5}\}$ (resp. $\{\ep_{2}+\ep_{3},\ep_{2}+\ep_{4}\}$) is a bad pair.
\end{proof}
\begin{prop}There are no initialized irreducible homogeneous Ulrich bundles on $E_{6}/P_{2,3,5}$.
\end{prop}
\begin{proof}Suppose for contradiction that there is an initialized irreducible homogeneous Ulrich bundle $E_{\lambda}$ on $E_{6}/P_{2,3,5}$ with $\lambda=\sum_{i=1}^{6}a_{i}\varpi_{i}$.
By the same argument as the proof of Lemma \ref{n1lem}, we deduce that exactly one of $a_{2}=0$ or $a_{3}=0$ or $a_{5}=0$ holds.

Let us consider the case $a_{2}=0$.
We turn to the condition that $2\in \im(\varphi_\lambda^J)$.
By taking $\Psi=\{\alpha_{2},\ep_{1}+\ep_{3},\alpha_{3},\alpha_{5}\}$ and the same argument as proof of Lemma \ref{n2lem}, we have that exactly one of $a_{3}=1$ or $a_{5}=1$ or $a_{4}=0$ holds.
If $a_{3}=1$ (resp. $a_{5}=1,\ a_{4}=0$), we see that $\{\ep_{1}+\ep_{4},-\ep_{1}+\ep_{4}\}$ (resp. $\{-\ep_{1}+\ep_{4},\ep_{2}+\ep_{3}\}$, $\{\gamma_{1,3},\gamma_{1,2}\}$) is a bad pair.

Assume that $a_{3}=0$.
We examine to the condition that $2\in \im(\varphi_\lambda^J)$.
By taking $\Psi=\{\alpha_{3},\alpha_{2},-\ep_{1}+\ep_{3},\alpha_{5},\beta_{2}\}$ and the same argument as proof of Lemma \ref{n2lem}, we obtain that exactly one of $a_{2}=1$ or $a_{5}=1$ or $a_{1}=0$ or $a_{4}=0$ holds.
If $a_{2}=1$ (resp. $a_{5}=1,\ a_{1}=0,\ a_{4}=0$), we can check that $\{\ep_{1}+\ep_{5},-\ep_{1}+\ep_{5}\}$ (resp. $\{\ep_{1}+\ep_{4},\ep_{2}+\ep_{3}\}$, $\{\ep_{2}+\ep_{4},\gamma_{3,5}\}$, $\{\ep_{2}+\ep_{5},\ep_{3}+\ep_{5}\}$) is a bad pair.

Suppose that $a_{5}=0$.
We turn to the condition that $2\in \im(\varphi_\lambda^J)$.
By taking $\Psi=\{\alpha_{5},\alpha_{2},\alpha_{3},-\ep_{2}+\ep_{4},-\ep_{2}+\ep_{5}\}$ and the same argument as proof of Lemma \ref{n2lem}, we have that exactly one of $a_{2}=1$ or $a_{3}=1$ or $a_{4}=0$ or $a_{6}=0$ holds.
If $a_{2}=1$ (resp. $a_{3}=1,\ a_{4}=0, a_{6}=0$), we see that $\{\ep_{2}+\ep_{3},-\ep_{1}+\ep_{4}\}$ (resp. $\{\ep_{1}+\ep_{4},\ep_{2}+\ep_{3}\}$, $\{\ep_{2}+\ep_{5},\ep_{3}+\ep_{5}\}$, $\{\ep_{1}+\ep_{4},\ep_{1}+\ep_{5}\}$) is a bad pair.
\end{proof}
\begin{prop}There are no initialized irreducible homogeneous Ulrich bundles on $E_{6}/P_{3,4,5}$.
\end{prop}
\begin{proof}Suppose for contradiction that there is an initialized irreducible homogeneous Ulrich bundle $E_{\lambda}$ on $E_{6}/P_{3,4,5}$ with $\lambda=\sum_{i=1}^{6}a_{i}\varpi_{i}$.
By the same argument as the proof of Lemma \ref{n1lem}, we deduce that exactly one of $a_{3}=0$ or $a_{4}=0$ or $a_{5}=0$ holds.

Let us consider the case $a_{3}=0$.
We turn to the condition that $2\in \im(\varphi_\lambda^J)$.
By taking $\Psi=\{\alpha_{3},-\ep_{1}+\ep_{3},\alpha_{5},\beta_{2}\}$ and the same argument as proof of Lemma \ref{2lem}, we have that exactly one of $a_{4}=2$ or $a_{5}=1$ or $a_{1}=0$ holds.
If $a_{5}=1$ (resp. $a_{1}=0$), we can check that $\{\gamma_{2,3},\gamma_{1,4}\}$ (resp. $\{-\ep_{1}+\ep_{5},\beta_{5}\}$) is a bad pair.
Suppose $a_{4}=2$.
We investigate the condition that $4\in \im(\varphi_\lambda^J)$.
By $\varphi_{\lambda}^{J}(-\ep_{1}+\ep_{4})=\frac{5+a_{5}}{3}$, and $\varphi_{\lambda}^{J}(-\ep_{2}+\ep_{4})=\frac{4+a_{5}}{2}$ and the same argument as the proof of Lemma \ref{22lem}, we deduce that $a_{5}$ is grater than or equal to $10$.
In this case, we have that exactly one of $a_{1}=3$ or $a_{2}=3$ holds by taking $\Psi=\{\alpha_{3},-\ep_{1}+\ep_{3},\alpha_{4},\beta_{3},\ep_{2}+\ep_{3}\}$ and the same argument as the proof of Lemma \ref{22lem}.
If $a_{1}=3$ (resp. $a_{2}=3$), we see that $\{-\ep_{1}+\ep_{5},\beta_{5}\}$ (resp. $\{\ep_{2}+\ep_{5},-\ep_{1}+\ep_{5}\}$) is a bad pair.

Suppose that $a_{4}=0$.
We turn to the condition that $2\in \im(\varphi_\lambda^J)$.
By taking $\Psi=\{\alpha_{4},\ep_{1}+\ep_{3},-\ep_{1}+\ep_{3},-\ep_{2}+\ep_{4}\}$ and the same argument as proof of Lemma \ref{2lem}, we get that exactly one of $a_{3}=2$ or $a_{5}=2$ or $a_{2}=0$ holds.
If $a_{2}=0$, we see that $\{\ep_{1}+\ep_{5},-\ep_{2}+\ep_{5}\}$ is a bad pair.
Assume $a_{3}=2$.
We investigate the condition that $4\in \im(\varphi_\lambda^J)$.
By $\varphi_{\lambda}^{J}(-\ep_{1}+\ep_{4})=\frac{5+a_{5}}{3}$, and $\varphi_{\lambda}^{J}(-\ep_{2}+\ep_{4})=\frac{2+a_{5}}{2}$ and the same argument as the proof of Lemma \ref{22lem}, we deduce that $a_{5}$ is grater than or equal to $10$.
In this case, we have that exactly one of $a_{1}=3$ or $a_{2}=3$ holds by taking $\Psi=\{\alpha_{4},-\ep_{1}+\ep_{3},\alpha_{3},\ep_{1}+\ep_{3},\beta_{3}\}$ and the same argument as the proof of Lemma \ref{22lem}.
If $a_{1}=3$ (resp. $a_{2}=3$), we can check that $\{-\ep_{1}+\ep_{5},\beta_{5}\}$ (resp. $\{\ep_{2}+\ep_{5},-\ep_{1}+\ep_{5}\}$) is a bad pair.
Assume $a_{5}=2$.
We also examine the condition that $4\in \im(\varphi_\lambda^J)$.
By $\varphi_{\lambda}^{J}(-\ep_{1}+\ep_{3})=\frac{a_{3}+2}{2}$, and $\varphi_{\lambda}^{J}(-\ep_{1}+\ep_{4})=\frac{5+a_{3}}{3}$ and the same argument as the proof of Lemma \ref{22lem}, we deduce that $a_{3}$ is grater than or equal to $10$.
In this case, we have that exactly one of $a_{2}=3$ or $a_{6}=3$ by taking $\Psi=\{\alpha_{4},\alpha_{5},-\ep_{2}+\ep_{4},\ep_{1}+\ep_{3},-\ep_{2}+\ep_{5}\}$ and the same argument as the proof of Lemma \ref{22lem}.
If $a_{2}=3$ (resp. $a_{6}=3$), we see that $\{\gamma_{3,5},\beta_{4}\}$ (resp. $\{\ep_{2}+\ep_{4},\ep_{2}+\ep_{5}\}$) is a bad pair.

Suppose that $a_{5}=0$.
We turn to the condition that $2\in \im(\varphi_\lambda^J)$.
By taking $\Psi=\{\alpha_{5},\alpha_{3},-\ep_{2}+\ep_{4},-\ep_{3}+\ep_{5}\}$ and the same argument as proof of Lemma \ref{2lem}, we get that exactly one of $a_{4}=2$ or $a_{3}=1$ or $a_{6}=0$ holds.
If $a_{3}=1$ (resp. $a_{1}=0$), we see that $\{-\ep_{1}+\ep_{3},-\ep_{2}+\ep_{4}\}$ (resp. $\{\ep_{1}+\ep_{4},\ep_{1}+\ep_{5}\}$) is a bad pair.
Suppose $a_{4}=2$.
We investigate the condition that $4\in \im(\varphi_\lambda^J)$.
By $\varphi_{\lambda}^{J}(-\ep_{1}+\ep_{3})=\frac{a_{3}+4}{2}$, and $\varphi_{\lambda}^{J}(-\ep_{1}+\ep_{4})=\frac{a_{3}+5}{3}$ and the same argument as the proof of Lemma \ref{22lem}, we deduce that $a_{3}$ is grater than or equal to $10$.
In this case, we have that exactly one of $a_{2}=3$ or $a_{6}=3$ holds by taking $\Psi=\{\alpha_{5},\alpha_{4},-\ep_{2}+\ep_{4}\}$ and the same argument as the proof of Lemma \ref{22lem}.
If $a_{2}=3$ (resp. $a_{6}=3$), we can check that $\{\ep_{2}+\ep_{4},-\ep_{1}+\ep_{4}\}$ (resp. $\{\ep_{2}+\ep_{4},\ep_{2}+\ep_{5}\}$) is a bad pair.
\end{proof}
\subsection{Proof of theorem \ref{thmE6} when $|J|=4$}
\begin{prop}There are no initialized irreducible homogeneous Ulrich bundles on $E_{6}/P_{1,2,3,4}\cong E_{6}/P_{2,4,5,6}$.
\end{prop}
\begin{proof}Suppose for contradiction that there is an initialized irreducible homogeneous Ulrich bundle $E_{\lambda}$ on $E_{6}/P_{1,2,3,4}$ with $\lambda=\sum_{i=1}^{6}a_{i}\varpi_{i}$.
By the same argument as the proof of Lemma \ref{n1lem}, we deduce that exactly one of $a_{1}=0$ or $a_{2}=0$ or $a_{3}=0$ or $a_{4}=0$ holds.

Let us consider the case $a_{1}=0$.
We turn to the condition that $2\in \im(\varphi_\lambda^J)$.
By taking $\Psi=\{\alpha_{1},\beta_{2},\alpha_{4},\alpha_{2}\}$ and the same argument as proof of Lemma \ref{2lem}, we get that exactly one of $a_{3}=2$ or $a_{2}=1$ or $a_{4}=1$ holds.
If $a_{2}=1$ (resp. $a_{4}=1$), we can check that $\{\ep_{2}+\ep_{4},\beta_{4}\}$ (resp. $\{\ep_{3}+\ep_{4},\gamma_{3,5}\}$) is a bad pair.
Suppose that $a_{3}=2$.
We examine the condition that $4\in \im(\varphi_\lambda^J)$.
By $\varphi_{\lambda}^{J}(-\ep_{1}+\ep_{4})=\frac{4+a_{4}}{2}$, and $\varphi_{\lambda}^{J}(\beta_{3})=\frac{a_{4}+5}{3}$ and the same argument as the proof of Lemma \ref{22lem}, we deduce that $a_{4}$ is grater than or equal to $10$.
Therefore, we have $a_{2}=3$ by taking $\Psi=\{\alpha_{1},\beta_{2},\alpha_{3},\alpha_{2}\}$ and the same argument as the proof of Lemma \ref{22lem}.
If $a_{2}=3$, we see that $\{\ep_{1}+\ep_{3},-\ep_{1}+\ep_{3}\}$ is a bad pair.

Assume that $a_{2}=0$.
We examine the condition that $2\in \im(\varphi_\lambda^J)$.
By taking $\Psi=\{\alpha_{2},\alpha_{1},\alpha_{3},\ep_{1}+\ep_{3}\}$ and the same argument as proof of Lemma \ref{2lem}, we get that exactly one of $a_{4}=2$ or $a_{1}=1$ or $a_{3}=1$ holds.
If $a_{1}=1$ (resp. $a_{4}=1$), we can check that $\{\ep_{2}+\ep_{3},\beta_{3}\}$ (resp. $\{\ep_{1}+\ep_{4},-\ep_{1}+\ep_{4}\}$) is a bad pair.
Suppose that $a_{4}=2$.
We investigate the condition that $4\in \im(\varphi_\lambda^J)$.
By $\varphi_{\lambda}^{J}(-\ep_{1}+\ep_{3})=\frac{4+a_{3}}{2}$, and $\varphi_{\lambda}^{J}(\ep_{2}+\ep_{3})=\frac{a_{3}+5}{3}$ and the same argument as the proof of Lemma \ref{22lem}, we deduce that $a_{3}$ is grater than or equal to $10$.
In this case, we have that exactly one of $a_{1}=3$ or $a_{5}=3$ holds by taking $\Psi=\{\alpha_{2},\ep_{1}+\ep_{3},-\ep_{2}+\ep_{3},\alpha_{1},\ep_{1}+\ep_{4}\}$ and the same argument as the proof of Lemma \ref{22lem}.
If $a_{1}=3$ (resp. $a_{5}=3$), we see that $\{\ep_{3}+\ep_{4},\gamma_{3,5}\}$ (resp. $\{\ep_{2}+\ep_{3},\ep_{2}+\ep_{4}\}$) is a bad pair.

Suppose that $a_{3}=0$.
We turn to the condition that $2\in \im(\varphi_\lambda^J)$.
By taking $\Psi=\{\alpha_{3},\beta_{2},-\ep_{1}+\ep_{3},\alpha_{2}\}$ and the same argument as proof of Lemma \ref{2lem}, we obtain that exactly one of $a_{1}=2$ or $a_{4}=2$ or $a_{2}=1$ holds.
If $a_{2}=1$, we see that $\{\ep_{1}+\ep_{4},-\ep_{1}+\ep_{4}\}$ is a bad pair.
Assume that $a_{1}=2$.
We investigate the condition that $4\in \im(\varphi_\lambda^J)$.
By $\varphi_{\lambda}^{J}(-\ep_{1}+\ep_{3})=\frac{2+a_{4}}{2}$, and $\varphi_{\lambda}^{J}(\beta_{3})=\frac{a_{4}+5}{3}$ and the same argument as the proof of Lemma \ref{22lem}, we deduce that $a_{4}$ is grater than or equal to $10$.
In this case, we obtain $a_{2}=3$ holds by taking $\Psi=\{\alpha_{3},\beta_{2},\alpha_{1},\alpha_{2}\}$ and the same argument as the proof of Lemma \ref{22lem}.
If $a_{2}=3$, we can check that $\{\ep_{1}+\ep_{4},-\ep_{1}+\ep_{4}\}$ is a bad pair.
Assume that $a_{4}=2$.
We also examine the condition that $4\in \im(\varphi_\lambda^J)$.
By $\varphi_{\lambda}^{J}(\beta_{2})=\frac{2+a_{1}}{2}$, and $\varphi_{\lambda}^{J}(\beta_{3})=\frac{a_{1}+5}{3}$ and
$\varphi_{\lambda}^{J}(\ep_{1}+\ep_{3})=\frac{4+a_{2}}{2}$, and $\varphi_{\lambda}^{J}(\ep_{2}+\ep_{3})=\frac{a_{2}+5}{3}$ and
the same argument as the proof of Lemma \ref{22lem}, we deduce that $a_{1}$ and $a_{2}$ are grater than or equal to $10$.
In this case, we obtain $a_{5}=3$ holds by taking $\Psi=\{\alpha_{3},-\ep_{1}+\ep_{3},\alpha_{4},-\ep_{1}+\ep_{4}\}$ and the same argument as the proof of Lemma \ref{22lem}.
If $a_{5}=3$, we see that $\{\ep_{2}+\ep_{3},\ep_{2}+\ep_{4}\}$ is a bad pair.

Assume that $a_{4}=0$.
We turn to the condition that $2\in \im(\varphi_\lambda^J)$.
By taking $\Psi=\{\alpha_{4},\alpha_{1},\ep_{1}+\ep_{3},-\ep_{1}+\ep_{3}\}$ and the same argument as proof of Lemma \ref{2lem}, we obtain that exactly one of $a_{2}=2$ or $a_{3}=2$ or $a_{1}=1$ holds.
If $a_{1}=1$, we see that $\{\ep_{3}+\ep_{4},\gamma_{3,5}\}$ is a bad pair.
Suppose that $a_{2}=2$.
We investigate the condition that $4\in \im(\varphi_\lambda^J)$.
By $\varphi_{\lambda}^{J}(-\ep_{1}+\ep_{3})=\frac{2+a_{4}}{2}$, and $\varphi_{\lambda}^{J}(\ep_{2}+\ep_{3})=\frac{a_{3}+5}{3}$ and the same argument as the proof of Lemma \ref{22lem}, we deduce that $a_{3}$ is grater than or equal to $10$.
In this case, we obtain that exactly one of $a_{1}=3$ or $a_{5}=3$ holds by taking $\Psi=\{\alpha_{4},\ep_{2}+\ep_{3},\alpha_{2},\ep_{1}+\ep_{4},\alpha_{1}\}$ and the same argument as the proof of Lemma \ref{22lem}.
If $a_{1}=3$ (resp. $a_{5}=3$), we check that $\{\ep_{2}+\ep_{4},\beta_{4}\}$ (resp. $\{\ep_{2}+\ep_{3},\ep_{2}+\ep_{4}\}$) is a bad pair.
Assume that $a_{3}=2$. 
We examine the condition that $4\in \im(\varphi_\lambda^J)$.
By $\varphi_{\lambda}^{J}(\beta_{2})=\frac{4+a_{1}}{2}$, and $\varphi_{\lambda}^{J}(\beta_{3})=\frac{a_{1}+5}{3}$, and
$\varphi_{\lambda}^{J}(\ep_{1}+\ep_{3})=\frac{2+a_{2}}{2}$, and $\varphi_{\lambda}^{J}(\ep_{2}+\ep_{3})=\frac{a_{2}+5}{3}$, and 
the same argument as the proof of Lemma \ref{22lem}, we deduce that $a_{1}$ and $a_{2}$ are grater than or equal to $10$.
In this case, we obtain $a_{5}=3$ holds by taking $\Psi=\{\alpha_{4},-\ep_{1}+\ep_{3},\alpha_{3},-\ep_{1}+\ep_{4}\}$ and the same argument as the proof of Lemma \ref{22lem}.
If $a_{5}=3$, we see that $\{\ep_{2}+\ep_{3},\ep_{2}+\ep_{4}\}$ is a bad pair.
\end{proof}
\begin{prop}There are no initialized irreducible homogeneous Ulrich bundles on $E_{6}/P_{1,2,3,5}\cong E_{6}/P_{2,3,5,6}$.
\end{prop}
\begin{proof}Suppose for contradiction that there is an initialized irreducible homogeneous Ulrich bundle $E_{\lambda}$ on $E_{6}/P_{1,2,3,5}$ with $\lambda=\sum_{i=1}^{6}a_{i}\varpi_{i}$.
By the same argument as the proof of Lemma \ref{n1lem}, we deduce that exactly one of $a_{1}=0$ or $a_{2}=0$ or $a_{3}=0$ or $a_{5}=0$ holds.

Suppose that $a_{1}=0$.
We turn to the condition that $2\in \im(\varphi_\lambda^J)$.
By taking $\Psi=\{\alpha_{1},\beta_{2},\alpha_{5},\alpha_{2}\}$ and the same argument as proof of Lemma \ref{2lem}, we obtain that exactly one of $a_{3}=2$ or $a_{2}=1$ or $a_{5}=1$ holds.
If $a_{2}=1$ (resp. $a_{5}=1$), we check that $\{\ep_{2}+\ep_{4},\beta_{4}\}$ (resp. $\{-\ep_{1}+\ep_{4},\beta_{3}\}$) is a bad pair.
Assume that $a_{3}=2$.
By taking $\Psi=\{\alpha_{1},\beta_{2},\alpha_{3},\beta_{3},\alpha_{2},\alpha_{5}\}$ and the same argument as the proof of Lemma \ref{22lem}, we have that exactly one of $a_{2}=3$ or $a_{4}=3$ or $a_{5}=3$ holds.
If $a_{2}=3$ (resp. $a_{4}=3,\ a_{5}=3$), we see that $\{\ep_{1}+\ep_{4},-\ep_{1}+\ep_{4}\}$ (resp. $\{\gamma_{1,3},\gamma_{1,2}\}$, $\{\ep_{1}+\ep_{4},\ep_{2}+\ep_{3}\}$) is a bad pair.

Assume that $a_{2}=0$.
We examine the condition that $2\in \im(\varphi_\lambda^J)$.
By taking $\Psi=\{\alpha_{2},\alpha_{1},\alpha_{3},\alpha_{5},\ep_{1}+\ep_{3}\}$ and the same argument as proof of Lemma \ref{n2lem}, we have that exactly one of $a_{1}=1$ or $a_{3}=1$ or $a_{5}=1$ or $a_{4}=0$ holds.
If $a_{1}=1$ (resp. $a_{3}=1$, $a_{6}=1$, $a_{4}=0$), we can check that $\{\ep_{2}+\ep_{4},\beta_{4}\}$ (resp. $\{\ep_{1}+\ep_{5},-\ep_{1}+\ep_{5}\}$, $\{\ep_{2}+\ep_{3},-\ep_{1}+\ep_{4}\}$, $\{\ep_{2}+\ep_{4},\ep_{3}+\ep_{4}\}$) is a bad pair.

Assume that $a_{3}=0$.
We turn to the condition that $2\in \im(\varphi_\lambda^J)$.
By taking $\Psi=\{\alpha_{3},\alpha_{2},\alpha_{5},\beta_{2},-\ep_{1}+\ep_{3}\}$ and the same argument as proof of Lemma \ref{2lem}, we obtain that exactly one of $a_{1}=2$ or $a_{2}=1$, $a_{5}=1$ or $a_{4}=0$ holds.
If $a_{2}=1$ (resp. $a_{5}=1$, $a_{4}=0$), we see that $\{\ep_{1}+\ep_{4},-\ep_{1}+\ep_{4}\}$ (resp. $\{\ep_{1}+\ep_{4},\ep_{2}+\ep_{3}\}$, $\{\ep_{2}+\ep_{5},\ep_{3}+\ep_{5}\}$) is a bad pair.
Suppose that $a_{1}=2$.
By taking $\Psi=\{\alpha_{3},\beta_{2},\alpha_{1},\beta_{3},\alpha_{2},\alpha_{5}\}$ and the same argument as the proof of Lemma \ref{22lem}, we have that exactly one of $a_{2}=3$ or $a_{4}=3$ or $a_{5}=3$ holds.
If $a_{2}=3$ (resp. $a_{4}=3,\ a_{5}=3$), we can check that $\{\ep_{1}+\ep_{4},-\ep_{1}+\ep_{4}\}$ (resp. $\{\ep_{2}+\ep_{4},\ep_{3}+\ep_{4}\}$, $\{\ep_{2}+\ep_{4},\gamma_{4,5}\}$) is a bad pair.

Suppose that $a_{5}=0$.
We examine the condition that $2\in \im(\varphi_\lambda^J)$.
By taking $\Psi=\{\alpha_{5},\alpha_{1},\alpha_{2},\alpha_{3},-\ep_{2}+\ep_{4},-\ep_{3}+\ep_{5}\}$ and the same argument as proof of Lemma \ref{n2lem}, we have that exactly one of $a_{1}=1$ or $a_{2}=1$ or $a_{3}=1$ or $a_{4}=0$ or $a_{5}=0$ holds.
If $a_{1}=1$ (resp. $a_{2}=1$, $a_{3}=1$, $a_{4}=0$, $a_{5}=0$), we can check that $\{-\ep_{1}+\ep_{4},\beta_{3}\}$ (resp. $\{\ep_{2}+\ep_{3},-\ep_{1}+\ep_{4}\}$, $\{\ep_{1}+\ep_{4},\ep_{2}+\ep_{3}\}$, $\{\ep_{2}+\ep_{4},\ep_{3}+\ep_{4}\}$, $\{\ep_{1}+\ep_{4},\ep_{1}+\ep_{5}\}$) is a bad pair.
\end{proof}
\begin{prop}There are no initialized irreducible homogeneous Ulrich bundles on $E_{6}/P_{1,2,3,6}\cong E_{6}/P_{1,2,5,6}$.
\end{prop}
\begin{proof}Suppose for contradiction that there is an initialized irreducible homogeneous Ulrich bundle $E_{\lambda}$ on $E_{6}/P_{1,2,3,6}$ with $\lambda=\sum_{i=1}^{6}a_{i}\varpi_{i}$.
By the same argument as the proof of Lemma \ref{n1lem}, we deduce that exactly one of $a_{1}=0$ or $a_{2}=0$ or $a_{3}=0$ or $a_{6}=0$ holds.

Let us consider the case $a_{1}=0$ first.
We turn to the condition that $2\in \im(\varphi_\lambda^J)$.
By taking $\Psi=\{\alpha_{1},\beta_{2},\alpha_{2},\alpha_{6}\}$ and the same argument as proof of Lemma \ref{2lem}, we obtain that exactly one of $a_{3}=2$ or $a_{2}=1$ or $a_{6}=1$ holds.
If $a_{2}=1$ (resp. $a_{6}=1$), we see that $\{\ep_{2}+\ep_{4},\beta_{4}\}$ (resp. \{$-\ep_{1}+\ep_{5},\beta_{4}\}$) is a bad pair.
Suppose that $a_{3}=2$.
We examine the condition that $4\in \im(\varphi_\lambda^J)$.
By taking $\Psi=\{\alpha_{1},\beta_{2},\alpha_{3},\alpha_{2},\alpha_{6},\beta_{3}\}$ and the same argument as proof of Lemma \ref{22lem}, we have that exactly one of $a_{2}=3$ or $a_{4}=3$ or $a_{6}=3$ holds.
If $a_{2}=3$ (resp. $a_{4}=3,\ a_{6}=3$), we can check that $\{\ep_{1}+\ep_{5},-\ep_{1}+\ep_{5}\}$ (resp. $\{\ep_{2}+\ep_{5},\ep_{3}+\ep_{5}\}$, $\{\ep_{2}+\ep_{4},\ep_{1}+\ep_{5}\}$) is a bad pair.

Next, let us consider the case $a_{2}=0$.
We turn to the condition that $2\in \im(\varphi_\lambda^J)$.
By taking $\Psi=\{\alpha_{2},\alpha_{1},\alpha_{3},\alpha_{6},\ep_{1}+\ep_{3}\}$ and the same argument as proof of Lemma \ref{2lem}, we obtain that exactly one of $a_{1}=1$ or $a_{3}=1$ or $a_{6}=1$ or $a_{4}=0$ holds.
If $a_{1}=1$ (resp. $a_{3}=1$, $a_{6}=1$, $a_{4}=0$), we see that $\{\ep_{2}+\ep_{3},\beta_{3}\}$ (resp. $\{\ep_{1}+\ep_{5},-\ep_{1}+\ep_{5}\}$, $\{\ep_{2}+\ep_{4},-\ep_{1}+\ep_{5}\}$, $\{\ep_{2}+\ep_{4},\ep_{3}+\ep_{4}\}$) is a bad pair.

Next, let us consider the case $a_{3}=0$.
We examine the condition that $2\in \im(\varphi_\lambda^J)$.
By taking $\Psi=\{\alpha_{3},\beta_{2},\alpha_{2},\alpha_{6},-\ep_{1}+\ep_{3}\}$ and the same argument as proof of Lemma \ref{2lem}, we have that exactly one of $a_{1}=2$ or $a_{2}=1$ or $a_{6}=1$ or $a_{4}=0$ or $a_{5}=0$ holds.
If $a_{2}=1$ (resp. $a_{6}=1$, $a_{4}=0$), we see that $\{\ep_{1}+\ep_{5},-\ep_{1}+\ep_{5}\}$ (resp. $\{\ep_{1}+\ep_{5},\ep_{2}+\ep_{4}\}$, $\{\ep_{2}+\ep_{4},\ep_{3}+\ep_{4}\}$) is a bad pair.
Suppose that $a_{1}=2$.
We examine the condition that $4\in \im(\varphi_\lambda^J)$.
By taking $\Psi=\{\alpha_{3},\beta_{2},\alpha_{1},\alpha_{2},\alpha_{6},\beta_{3}\}$ and the same argument as proof of Lemma \ref{22lem}, we have that exactly one of $a_{2}=3$ or $a_{4}=3$ or $a_{6}=3$ holds.
If $a_{2}=3$ (resp. $a_{4}=3,\ a_{6}=3$), we can check that $\{\ep_{1}+\ep_{5},-\ep_{1}+\ep_{5}\}$ (resp. $\{\gamma_{3,5},\gamma_{2,5}\}$, $\{-\ep_{1}+\ep_{5},\beta_{4}\}$) is a bad pair.

Finally, let us consider the case $a_{6}=0$.
We turn to the condition that $2\in \im(\varphi_\lambda^J)$.
By taking $\Psi=\{\alpha_{6},\alpha_{1},\alpha_{2},\alpha_{3},-\ep_{3}+\ep_{5}\}$ and the same argument as proof of Lemma \ref{n2lem}, we obtain that exactly one of $a_{1}=1$ or $a_{2}=1$ or $a_{3}=1$ or $a_{5}=0$ holds.
If $a_{1}=1$ (resp. $a_{2}=1$, $a_{3}=1$, $a_{5}=0$), we see that $\{-\ep_{1}+\ep_{5},\beta_{4}\}$ (resp. $\{\ep_{2}+\ep_{4},-\ep_{1}+\ep_{5}\}$, $\{\gamma_{2,4},\gamma_{1,5}\}$, $\{\ep_{3}+\ep_{5},\ep_{4}+\ep_{5}\}$) is a bad pair.
\end{proof}
\begin{prop}There are no initialized irreducible homogeneous Ulrich bundles on $E_{6}/P_{1,2,4,5}\cong E_{6}/P_{2,3,4,6}$.
\end{prop}
\begin{proof}Suppose for contradiction that there is an initialized irreducible homogeneous Ulrich bundle $E_{\lambda}$ on $E_{6}/P_{1,2,4,5}$ with $\lambda=\sum_{i=1}^{6}a_{i}\varpi_{i}$.
By the same argument as the proof of Lemma \ref{n1lem}, we deduce that exactly one of $a_{1}=0$ or $a_{2}=0$ or $a_{4}=0$ or $a_{5}=0$ holds.

Assume that $a_{1}=0$.
We examine the condition that $2\in \im(\varphi_\lambda^J)$.
By taking $\Psi=\{\alpha_{1},\alpha_{2},\alpha_{4},\alpha_{5},\beta_{2}\}$ and the same argument as proof of Lemma \ref{n2lem}, we obtain that exactly one of $a_{2}=1$ or $a_{4}=1$ or $a_{5}=1$ or $a_{3}=0$ holds.
If $a_{2}=1$ (resp. $a_{4}=1$, $a_{5}=1$, $a_{3}=0$), we see that $\{\ep_{2}+\ep_{3},\beta_{3}\}$ (resp. $\{\ep_{3}+\ep_{5},\gamma_{3,4}\}$, $\{-\ep_{1}+\ep_{4},\beta_{3}\}$, $\{\ep_{1}+\ep_{3},\ep_{2}+\ep_{3}\}$) is a bad pair.

Suppose that $a_{2}=0$.
We turn to the condition that $2\in \im(\varphi_\lambda^J)$.
By taking $\Psi=\{\alpha_{2},\alpha_{1},\alpha_{5},\ep_{1}+\ep_{3}\}$ and the same argument as proof of Lemma \ref{2lem}, we have that exactly one of $a_{4}=2$ or $a_{1}=1$ or $a_{5}=1$ holds.
If $a_{1}=1$ (resp. $a_{5}=1$), we see that $\{\ep_{2}+\ep_{3},\beta_{3}\}$ (resp. $\{\ep_{2}+\ep_{3},-\ep_{1}+\ep_{4}\}$) is a bad pair.
Assume that $a_{4}=2$.
We investigate the condition that $4\in \im(\varphi_\lambda^J)$.
By $\varphi_{\lambda}^{J}(-\ep_{2}+\ep_{4})=\frac{4+a_{5}}{2}$, and $\varphi_{\lambda}^{J}(\ep_{1}+\ep_{4})=\frac{a_{5}+5}{3}$, and
the same argument as the proof of Lemma \ref{22lem}, we deduce that $a_{5}$ is grater than or equal to $10$.
In this case, we obtain that exactly one of $a_{1}=3$ or $a_{3}=3$ holds by taking $\Psi=\{\alpha_{2},\ep_{1}+\ep_{3},\alpha_{4},\ep_{2}+\ep_{3},\alpha_{1}\}$ and the same argument as the proof of Lemma \ref{22lem}.
If $a_{1}=3$ (resp. $a_{3}=3$), we see that $\{\ep_{3}+\ep_{4},\gamma_{3,5}\}$ (resp. $\{\ep_{1}+\ep_{5},\ep_{2}+\ep_{5}\}$) is a bad pair.

Assume that $a_{4}=0$.
We turn to the condition that $2\in \im(\varphi_\lambda^J)$.
By taking $\Psi=\{\alpha_{4},\alpha_{1},-\ep_{2}+\ep_{4},\ep_{1}+\ep_{3},-\ep_{1}+\ep_{3}\}$ and the same argument as proof of Lemma \ref{2lem}, we obtain that exactly one of $a_{2}=2$ or $a_{5}=2$ or $a_{1}=1$ or $a_{3}=0$ holds.
If $a_{1}=1$ (resp. $a_{3}=0$), we see that $\{\ep_{3}+\ep_{4},\gamma_{3,5}\}$ (resp. $\{\ep_{1}+\ep_{4},\ep_{2}+\ep_{4}\}$) is a bad pair.
Suppose that $a_{2}=2$.
We investigate the condition that $4\in \im(\varphi_\lambda^J)$.
By $\varphi_{\lambda}^{J}(-\ep_{2}+\ep_{4})=\frac{2+a_{5}}{2}$, and $\varphi_{\lambda}^{J}(\ep_{1}+\ep_{4})=\frac{a_{5}+5}{3}$, and
the same argument as the proof of Lemma \ref{22lem}, we deduce that $a_{5}$ is grater than or equal to $10$.
In this case, we obtain that exactly one of $a_{1}=3$ or $a_{3}=3$ holds by taking $\Psi=\{\alpha_{4},\ep_{1}+\ep_{3},\alpha_{2},\ep_{2}+\ep_{3},\alpha_{1}\}$ and the same argument as the proof of Lemma \ref{22lem}.
If $a_{1}=3$ (resp. $a_{3}=3$), we can check that $\{\ep_{2}+\ep_{3},\beta_{3}\}$ (resp. $\{\ep_{1}+\ep_{5},\ep_{2}+\ep_{5}\}$) is a bad pair.
Suppose that $a_{5}=2$.
We investigate the condition that $4\in \im(\varphi_\lambda^J)$.
By $\varphi_{\lambda}^{J}(\ep_{1}+\ep_{3})=\frac{2+a_{2}}{2}$, and $\varphi_{\lambda}^{J}(\ep_{1}+\ep_{4})=\frac{a_{2}+5}{3}$, and
the same argument as the proof of Lemma \ref{22lem}, we deduce that $a_{2}$ is grater than or equal to $10$.
In this case, we obtain that exactly one of $a_{1}=3$ or $a_{3}=3$ or $a_{6}=3$ holds by taking $\Psi=\{\alpha_{4},-\ep_{2}+\ep_{4},\alpha_{5},-\ep_{2}+\ep_{5},\alpha_{1},-\ep_{1}+\ep_{4}\}$ and the same argument as the proof of Lemma \ref{22lem}.
If $a_{1}=3$ (resp. $a_{3}=3,\ a_{6}=3$), we can check that $\{-\ep_{1}+\ep_{4},\beta_{3}\}$ (resp. $\{\gamma_{2,5},\gamma_{1,5}\}$, $\{\ep_{1}+\ep_{4},\ep_{1}+\ep_{5}\}$) is a bad pair.

Assume that $a_{5}=0$.
We turn to the condition that $2\in \im(\varphi_\lambda^J)$.
By taking $\Psi=\{\alpha_{5},\alpha_{1},\alpha_{2},-\ep_{2}+\ep_{4},-\ep_{3}+\ep_{5}\}$ and the same argument as proof of Lemma \ref{2lem}, we obtain that exactly one of $a_{4}=2$ or $a_{1}=1$ or $a_{2}=1$ or $a_{6}=0$ holds.
If $a_{1}=1$ (resp. $a_{2}=1$, $a_{6}=0$), we see that $\{-\ep_{1}+\ep_{4},\beta_{4}\}$ (resp. $\{\ep_{2}+\ep_{3},-\ep_{1}+\ep_{4}\}$, $\{\ep_{1}+\ep_{4},\ep_{1}+\ep_{5}\}$) is a bad pair, respectively.
Suppose that $a_{4}=2$.
We investigate the condition that $4\in \im(\varphi_\lambda^J)$.
By $\varphi_{\lambda}^{J}(\ep_{1}+\ep_{3})=\frac{4+a_{2}}{2}$, and $\varphi_{\lambda}^{J}(\ep_{1}+\ep_{4})=\frac{a_{2}+5}{3}$, and
the same argument as the proof of Lemma \ref{22lem}, we deduce that $a_{2}$ is grater than or equal to $10$.
In this case, we obtain that exactly one of $a_{1}=3$ or $a_{3}=3$ or $a_{6}=3$ holds by taking $\Psi=\{\alpha_{5},-\ep_{2}+\ep_{4},\alpha_{4},\alpha_{1},-\ep_{1}+\ep_{4},-\ep_{2}+\ep_{5}\}$ and the same argument as the proof of Lemma \ref{22lem}.
If $a_{1}=3$ (resp. $a_{3}=3,\ a_{6}=3$), we can check that $\{\ep_{3}+\ep_{4},\gamma_{3,5}\}$ (resp. $\{\ep_{1}+\ep_{4},\ep_{2}+\ep_{4}\}$, $\{\ep_{2}+\ep_{4},\ep_{2}+\ep_{5}\}$) is a bad pair, respectively.
\end{proof}
\begin{prop}There are no initialized irreducible homogeneous Ulrich bundles on $E_{6}/P_{1,2,4,6}$.
\end{prop}
\begin{proof}Suppose for contradiction that there is an initialized irreducible homogeneous Ulrich bundle $E_{\lambda}$ on $E_{6}/P_{1,2,4,6}$ with $\lambda=\sum_{i=1}^{6}a_{i}\varpi_{i}$.
By the same argument as the proof of Lemma \ref{n1lem}, we deduce that exactly one of $a_{1}=0$ or $a_{2}=0$ or $a_{4}=0$ or $a_{6}=0$ holds.

Assume that $a_{1}=0$.
We examine the condition that $2\in \im(\varphi_\lambda^J)$.
By taking $\Psi=\{\alpha_{1},\beta_{2},\alpha_{2},\alpha_{4},\alpha_{6}\}$ and the same argument as proof of Lemma \ref{2lem}, we obtain that exactly one of $a_{2}=1$ or $a_{4}=1$ or $a_{6}=1$ or $a_{3}=0$ holds.
If $a_{2}=1$ (resp. $a_{4}=1$, $a_{3}=0$, $a_{6}=1$), we see that $\{\ep_{2}+\ep_{3},\beta_{3}\}$ (resp. $\{\gamma_{3,5},\ep_{3}+\ep_{4}\}$, $\{\ep_{1}+\ep_{3},\ep_{2}+\ep_{3}\}$, $\{\ep_{2}+\ep_{5},\gamma_{3,5}\}$) is a bad pair.

Suppose that $a_{2}=0$.
We turn to the condition that $2\in \im(\varphi_\lambda^J)$.
By taking $\Psi=\{\alpha_{2},\alpha_{1},\alpha_{6},\ep_{1}+\ep_{3}\}$ and the same argument as proof of Lemma \ref{2lem}, we obtain that exactly one of $a_{4}=2$ or $a_{1}=1$ or $a_{6}=1$ holds.
If $a_{1}=1$ (resp. $a_{6}=1$), we see that $\{\ep_{2}+\ep_{3},\beta_{3}\}$ (resp. $\{\ep_{1}+\ep_{4},-\ep_{2}+\ep_{5}\}$) is a bad pair.
Assume that $a_{4}=2$.
We turn to the condition that $4\in \im(\varphi_\lambda^J)$.
By taking $\Psi=\{\alpha_{2},\ep_{1}+\ep_{3},\alpha_{4},\alpha_{1},\alpha_{6},\ep_{1}+\ep_{4}\}$ and the same argument as proof of Lemma \ref{22lem}, we have that exactly one of $a_{1}=3$ or $a_{3}=3$ or $a_{5}=3$ holds.
If $a_{1}=3$ (resp. $a_{3}=3,\ a_{5}=3$), we can check that $\{\ep_{2}+\ep_{3},\beta_{3}\}$ (resp. $\{\gamma_{2,4},\gamma_{1,4}\}$, $\{\gamma_{4,5},\gamma_{3,5}\}$) is a bad pair.

Assume that $a_{4}=0$.
We turn to the condition that $2\in \im(\varphi_\lambda^J)$.
By taking $\Psi=\{\alpha_{4},\alpha_{1},\alpha_{6},\ep_{1}+\ep_{3},-\ep_{1}+\ep_{3},-\ep_{2}+\ep_{4}\}$ and the same argument as proof of Lemma \ref{2lem}, we obtain that exactly one of $a_{2}=2$ or $a_{1}=1$ or $a_{6}=1$ or $a_{3}=0$ or $a_{5}=0$ holds.
If $a_{1}=1$ (resp. $a_{6}=1$, $a_{3}=0$, $a_{5}=0$), we see that $\{\ep_{3}+\ep_{4},\gamma_{3,5}\}$ (resp. $\{\gamma_{3,4},\gamma_{2,5}\}$, $\{\ep_{1}+\ep_{4},\ep_{2}+\ep_{4}\}$, $\{\ep_{1}+\ep_{3},\ep_{1}+\ep_{4}\}$) is a bad pair.
Suppose that $a_{2}=2$.
We turn to the condition that $4\in \im(\varphi_\lambda^J)$.
By taking $\Psi=\{\alpha_{4},\ep_{1}+\ep_{3},\alpha_{2},\ep_{1}+\ep_{4},\ep_{2}+\ep_{3},\alpha_{6},\alpha_{1}\}$ and the same argument as proof of Lemma \ref{22lem}, we have that exactly one of $a_{1}=3$ or $a_{3}=3$ or $a_{5}=3$ or $a_{6}=3$ holds.
If $a_{1}=3$ (resp. $a_{3}=3,\ a_{5}=3,\ a_{6}=3$), we can check that $\{\ep_{2}+\ep_{4},\beta_{4}\}$ (resp. $\{\ep_{1}+\ep_{5},\ep_{2}+\ep_{5}\}$, $\{\gamma_{1,4},\gamma_{2,3}\}$, $\{\ep_{1}+\ep_{4},-\ep_{2}+\ep_{5}\}$) is a bad pair.

Assume that $a_{6}=0$.
We examine the condition that $2\in \im(\varphi_\lambda^J)$.
By taking $\Psi=\{\alpha_{6},\alpha_{1},\alpha_{2},\alpha_{4},-\ep_{3}+\ep_{5}\}$ and the same argument as proof of Lemma \ref{2lem}, we obtain that exactly one of $a_{1}=1$ or $a_{2}=1$ or $a_{4}=1$ or $a_{5}=0$ holds.
If $a_{1}=1$ (resp. $a_{2}=1$, $a_{4}=1$, $a_{5}=0$), we see that $\{\ep_{2}+\ep_{5},\gamma_{3,5}\}$ (resp. $\{\ep_{1}+\ep_{4},-\ep_{2}+\ep_{5}\}$, $\{\gamma_{3,4},\gamma_{2,5}\}$, $\{\ep_{1}+\ep_{3},\ep_{1}+\ep_{4}\}$) is a bad pair.
\end{proof}
\begin{prop}There are no initialized irreducible homogeneous Ulrich bundles on $E_{6}/P_{1,3,4,5}\cong E_{6}/P_{3,4,5,6}$.
\end{prop}
\begin{proof}Suppose for contradiction that there is an initialized irreducible homogeneous Ulrich bundle $E_{\lambda}$ on $E_{6}/P_{1,3,4,5}$ with $\lambda=\sum_{i=1}^{6}a_{i}\varpi_{i}$.
By the same argument as the proof of Lemma \ref{n1lem}, we deduce that exactly one of $a_{1}=0$ or $a_{3}=0$ or $a_{4}=0$ or $a_{5}=0$ holds.

Assume that $a_{1}=0$.
We examine the condition that $2\in \im(\varphi_\lambda^J)$.
By taking $\Psi=\{\alpha_{1},\beta_{2},\alpha_{4},\alpha_{5}\}$ and the same argument as proof of Lemma \ref{2lem}, we obtain that exactly one of $a_{3}=2$ or $a_{4}=1$ or $a_{5}=1$ holds.
If $a_{4}=1$ (resp. $a_{5}=1$), we see that $\{\ep_{3}+\ep_{4},\gamma_{3,5}\}$ (resp. $\{\beta_{3},-\ep_{1}+\ep_{4}\}$) is a bad pair.
Suppose that $a_{3}=2$.
By $\varphi_{\lambda}^{J}(-\ep_{1}+\ep_{3})=\frac{4+a_{4}}{2}$, and $\varphi_{\lambda}^{J}(\beta_{3})=\frac{a_{4}+5}{3}$, and
the same argument as the proof of Lemma \ref{22lem}, we deduce that $a_{4}$ is grater than or equal to $10$.
In this case, we obtain $a_{5}=3$ by taking $\Psi=\{\alpha_{1},\beta_{2},\alpha_{3},\alpha_{5}\}$ and the same argument as the proof of Lemma \ref{22lem}.
If $a_{5}=3$, we can check that $\{\ep_{1}+\ep_{4},\ep_{2}+\ep_{3}\}$ is a bad pair.

Suppose that $a_{3}=0$.
We turn to the condition that $2\in \im(\varphi_\lambda^J)$.
By taking $\Psi=\{\alpha_{3},\beta_{2},-\ep_{1}+\ep_{3},\alpha_{5}\}$ and the same argument as proof of Lemma \ref{2lem}, we obtain that exactly one of $a_{1}=2$ or $a_{4}=2$ or $a_{5}=1$ holds.
If $a_{5}=1$, we see that $\{\ep_{1}+\ep_{4},\ep_{2}+\ep_{3}\}$ is a bad pair.
Assume that $a_{1}=2$.
We examine the condition that $4\in \im(\varphi_\lambda^J)$.
By $\varphi_{\lambda}^{J}(-\ep_{1}+\ep_{3})=\frac{4+a_{4}}{2}$, and $\varphi_{\lambda}^{J}(\beta_{3})=\frac{a_{4}+5}{3}$, and
the same argument as the proof of Lemma \ref{22lem}, we deduce that $a_{4}$ is grater than or equal to $10$.
In this case, we obtain $a_{5}=3$ by taking $\Psi=\{\alpha_{3},\beta_{2},\alpha_{1},\alpha_{5}\}$ and the same argument as the proof of Lemma \ref{22lem}.
If $a_{5}=3$, we see that $\{\ep_{1}+\ep_{4},\ep_{2}+\ep_{3}\}$ is a bad pair.
Suppose that $a_{4}=2$.
We also investigate the condition that that $4\in \im(\varphi_\lambda^J)$.
By $\varphi_{\lambda}^{J}(\beta_{2})=\frac{2+a_{1}}{2}$, and $\varphi_{\lambda}^{J}(\beta_{3})=\frac{a_{1}+5}{3}$, and $\varphi_{\lambda}^{J}(-\ep_{2}+\ep_{4})=\frac{4+a_{5}}{2}$, and $\varphi_{\lambda}^{J}(-\ep_{1}+\ep_{4})=\frac{a_{5}+5}{3}$, and
the same argument as the proof of Lemma \ref{22lem},
we deduce that $a_{1}$ and $a_{5}$ are grater than or equal to $10$.
In this case, we obtain $a_{2}=3$ by taking $\Psi=\{\alpha_{3},-\ep_{1}+\ep_{3},\alpha_{4},\ep_{2}+\ep_{3}\}$ and the same argument as the proof of Lemma \ref{22lem}.
If $a_{2}=3$ , we can check that $\{\ep_{2}+\ep_{5},-\ep_{1}+\ep_{4}\}$ is a bad pair.

Assume that $a_{4}=0$.
We examine the condition that $2\in \im(\varphi_\lambda^J)$.
By taking $\Psi=\{\alpha_{4},\alpha_{1},\ep_{1}+\ep_{3},-\ep_{1}+\ep_{3},-\ep_{2}+\ep_{4}\}$ and the same argument as proof of Lemma \ref{2lem}, we obtain that exactly one of $a_{3}=2$ or $a_{5}=2$ or $a_{1}=1$ or $a_{2}=0$ holds.
If $a_{1}=1$ (resp. $a_{2}=0$), we see that $\{\ep_{3}+\ep_{4},\gamma_{3,5}\}$ (resp. $\{\ep_{2}+\ep_{4},-\ep_{1}+\ep_{4}\}$) is a bad pair.
Suppose that $a_{3}=2$.
We examine the condition that $4\in \im(\varphi_\lambda^J)$.
By $\varphi_{\lambda}^{J}(\beta_{2})=\frac{4+a_{1}}{2}$, and $\varphi_{\lambda}^{J}(\beta_{3})=\frac{a_{1}+5}{3}$, and $\varphi_{\lambda}^{J}(-\ep_{2}+\ep_{4})=\frac{2+a_{5}}{2}$, and $\varphi_{\lambda}^{J}(-\ep_{1}+\ep_{4})=\frac{a_{5}+5}{3}$, and
the same argument as the proof of Lemma \ref{22lem}, we deduce that $a_{1}$ and $a_{5}$ are grater than or equal to $10$.
In this case, we obtain $a_{2}=3$ by taking $\Psi=\{\alpha_{4},-\ep_{1}+\ep_{3},\alpha_{3},\ep_{2}+\ep_{3}\}$ and the same argument as the proof of Lemma \ref{22lem}.
If $a_{2}=3$, we can check that $\{\ep_{2}+\ep_{4},-\ep{1}+\ep_{4}\}$ is a bad pair.
Assume that $a_{5}=2$.
We investigate the condition that that $4\in \im(\varphi_\lambda^J)$.
By $\varphi_{\lambda}^{J}(-\ep_{1}+\ep_{3})=\frac{2+a_{3}}{2}$, and $\varphi_{\lambda}^{J}(-\ep_{1}+\ep_{4})=\frac{a_{3}+5}{3}$, and 
the same argument as the proof of Lemma \ref{22lem}, we deduce that $a_{3}$ is grater than or equal to $10$.
In this case, we obtain that exactly one of $a_{1}=3$ or $a_{2}=3$ or $a_{6}=3$ holds by taking $\Psi=\{\alpha_{4},-\ep_{2}+\ep_{4},\alpha_{5},\ep_{1}+\ep_{4},\alpha_{1},-\ep_{2}+\ep_{5}\}$ and the same argument as the proof of Lemma \ref{22lem}.
If $a_{1}=3$ (resp. $a_{2}=3,\ a_{6}=3$), we can check that $\{\ep_{2}+\ep_{4},\gamma_{4,5}\}$ (resp. $\{\ep_{2}+\ep_{4},-\ep_{1}+\ep_{4}\}$, $\{\ep_{2}+\ep_{4},\ep_{2}+\ep_{5}\}$) is a bad pair.

Suppose that $a_{5}=0$.
We turn to the condition that $2\in \im(\varphi_\lambda^J)$.
By taking $\Psi=\{\alpha_{5},\alpha_{1},\alpha_{3},-\ep_{3}+\ep_{5},-\ep_{2}+\ep_{4}\}$ and the same argument as proof of Lemma \ref{2lem}, we obtain that exactly one of  $a_{4}=2$ or $a_{1}=1$ or $a_{3}=1$ or $a_{6}=0$ holds.
If $a_{1}=1$ (resp. $a_{3}=1$, $a_{6}=0$), we check that $\{\gamma_{4,5},\ep_{2}+\ep_{4}\}$ (resp. $\{\ep_{1}+\ep_{4},\ep_{2}+\ep_{3}\}$, $\{\ep_{1}+\ep_{4},\ep_{1}+\ep_{5}\}$) is a bad pair.
Suppose that $a_{4}=2$.
We examine the condition that $4\in \im(\varphi_\lambda^J)$.
By $\varphi_{\lambda}^{J}(-\ep_{1}+\ep_{3})=\frac{4+a_{3}}{2}$, and $\varphi_{\lambda}^{J}(-\ep_{1}+\ep_{4})=\frac{a_{3}+5}{3}$, and 
the same argument as the proof of Lemma \ref{22lem}, we deduce that $a_{3}$ is grater than or equal to $10$.
In this case, we obtain that exactly one of $a_{1}=3$ or $a_{2}=3$ or $a_{6}=3$ holds by taking $\Psi=\{\alpha_{5},-\ep_{2}+\ep_{4},\alpha_{4},\ep_{1}+\ep_{4},-\ep_{2}+\ep_{5},\alpha_{1}\}$ and the same argument as the proof of Lemma \ref{22lem}.
If $a_{1}=3$ (resp. $a_{2}=3,\ a_{6}=3$), we can check that $\{-\ep_{1}+\ep_{3},\beta_{2}\}$ (resp. $\{\ep_{2}+\ep_{4},-\ep_{1}+\ep_{4}\}$, $\{\ep_{2}+\ep_{4},\ep_{2}+\ep_{5}\}$) is a bad pair.
\end{proof}
\begin{prop}There are no initialized irreducible homogeneous Ulrich bundles on $E_{6}/P_{1,3,4,6}\cong E_{6}/P_{1,4,5,6}$.
\end{prop}
\begin{proof}Suppose for contradiction that there is an initialized irreducible homogeneous Ulrich bundle $E_{\lambda}$ on $E_{6}/P_{1,3,4,6}$ with $\lambda=\sum_{i=1}^{6}a_{i}\varpi_{i}$.
By the same argument as the proof of Lemma \ref{n1lem}, we deduce that exactly one of $a_{1}=0$ or $a_{3}=0$ or $a_{4}=0$ or $a_{6}=0$ holds.

Let us consider the case $a_{1}=0$ first.
We examine the condition that $2\in \im(\varphi_\lambda^J)$.
By taking $\Psi=\{\alpha_{1},\beta_{2},\alpha_{4},\alpha_{6}\}$ and the same argument as proof of Lemma \ref{2lem}, we obtain that exactly one of $a_{3}=2$ or $a_{4}=1$ or $a_{6}=1$ holds.
If $a_{4}=1$ (resp. $a_{6}=1$), we see that $\{\ep_{2}+\ep_{5},\gamma_{3,4}\}$ (resp. $\{\beta_{4},-\ep_{1}+\ep_{5}\}$) is a bad pair.
Suppose that $a_{3}=2$.
We turn to the condition that $4\in \im(\varphi_\lambda^J)$.
By $\varphi_{\lambda}^{J}(-\ep_{1}+\ep_{3})=\frac{4+a_{4}}{2}$, and $\varphi_{\lambda}^{J}(\beta_{3})=\frac{a_{4}+5}{3}$, and 
the same argument as the proof of Lemma \ref{22lem}, we deduce that $a_{4}$ is grater than or equal to $10$.
In this case, we obtain $a_{6}=3$ by taking $\Psi=\{\alpha_{1},\beta_{2},\alpha_{3},\alpha_{6}\}$ and the same argument as the proof of Lemma \ref{22lem}.
If $a_{6}=3$, we see that $\{\ep_{1}+\ep_{5},\ep_{2}+\ep_{4}\}$ is a bad pair.

Next, let us consider the case $a_{3}=0$.
We examine the condition that $2\in \im(\varphi_\lambda^J)$.
By taking $\Psi=\{\alpha_{3},\beta_{2},-\ep_{1}+\ep_{3},\alpha_{6}\}$ and the same argument as proof of Lemma \ref{2lem}, we have that exactly one of $a_{1}=2$ or $a_{4}=2$ or $a_{6}=1$ holds.
If $a_{6}=1$, we see that $\{\ep_{1}+\ep_{5},\ep_{2}+\ep_{4}\}$ is a bad pair.
Assume that $a_{1}=2$.
We turn to the condition that $4\in \im(\varphi_\lambda^J)$.
By $\varphi_{\lambda}^{J}(-\ep_{1}+\ep_{3})=\frac{2+a_{4}}{2}$, and $\varphi_{\lambda}^{J}(\beta_{3})=\frac{a_{4}+5}{3}$, and 
the same argument as the proof of Lemma \ref{22lem}, we deduce that $a_{4}$ is grater than or equal to $10$.
In this case, we obtain $a_{6}=3$ by taking $\Psi=\{\alpha_{3},\beta_{2},\alpha_{1},\alpha_{6}\}$ and the same argument as the proof of Lemma \ref{22lem}.
If $a_{6}=3$, we see that $\{\ep_{1}+\ep_{5},\ep_{2}+\ep_{4}\}$ is a bad pair.
Suppose that $a_{4}=2$.
We also investigate the condition that that $4\in \im(\varphi_\lambda^J)$.
By $\varphi_{\lambda}^{J}(\beta_{2})=\frac{2+a_{1}}{2}$, and $\varphi_{\lambda}^{J}(\beta_{3})=\frac{a_{1}+5}{3}$, and 
the same argument as the proof of Lemma \ref{22lem}, we deduce that $a_{1}$ is grater than or equal to $10$.
In this case, we obtain that exactly one of $a_{2}=3$ or $a_{5}=3$ or $a_{6}=3$ holds by taking $\Psi=\{\alpha_{3},-\ep_{1}+\ep_{3},\alpha_{4},\ep_{2}+\ep_{3},-\ep_{1}+\ep_{4},\alpha_{6}\}$ and the same argument as the proof of Lemma \ref{22lem}.
If $a_{2}=3$ (resp. $a_{5}=3,\ a_{6}=3$), we can check that $\{\ep_{2}+\ep_{5},-\ep_{1}+\ep_{5}\}$ (resp. $\{\gamma_{3,5},\gamma_{4,5}\}$, $\{\ep_{1}+\ep_{5},\ep_{2}+\ep_{4}\}$) is a bad pair.

Next, let us consider the case $a_{4}=0$.
We turn to the condition that $2\in \im(\varphi_\lambda^J)$.
By taking $\Psi=\{\alpha_{4},\alpha_{1},\alpha_{6},-\ep_{1}+\ep_{3},-\ep_{2}+\ep_{4},\ep_{1}+\ep_{3}\}$ and the same argument as proof of Lemma \ref{2lem}, we have that exactly one of $a_{3}=2$ or $a_{1}=1$ or $a_{6}=1$ or $a_{2}=0$ or $a_{5}=0$ holds.
If $a_{1}=1$ (resp. $a_{6}=1$, $a_{2}=0$, $a_{5}=0$), we can check that $\{\ep_{3}+\ep_{4},\gamma_{3,5}\}$ (resp. $\{\gamma_{3,4},\gamma_{2,5}\}$, $\{\ep_{1}+\ep_{5},-\ep_{2}+\ep_{5}\}$, $\{\ep_{2}+\ep_{3},-\ep_{2}+\ep_{4}\}$) is a bad pair.
Suppose that $a_{3}=2$.
We turn to the condition that $4\in \im(\varphi_\lambda^J)$.
By $\varphi_{\lambda}^{J}(\beta_{2})=\frac{4+a_{1}}{2}$, and $\varphi_{\lambda}^{J}(\beta_{3})=\frac{a_{1}+5}{3}$, and 
the same argument as the proof of Lemma \ref{22lem}, we deduce that $a_{1}$ is grater than or equal to $10$.
In this case, we obtain that exactly one of $a_{2}=3$ or $a_{5}=3$ or $a_{6}=3$ holds by taking $\Psi=\{\alpha_{4},-\ep_{1}+\ep_{3},\alpha_{3},\ep_{2}+\ep_{3},-\ep_{1}+\ep_{4},\alpha_{6}\}$ and the same argument as the proof of Lemma \ref{22lem}.
If $a_{2}=3$ (resp. $a_{5}=3,\ a_{6}=3$), we can check that $\{\ep_{2}+\ep_{5},-\ep{1}+\ep_{5}\}$ (resp. $\{\gamma_{4,5},\gamma_{3,5}\}$, $\{\ep_{1}+\ep_{5},\ep_{2}+\ep_{4}\}$) is a bad pair.

Finally, let us consider the case $a_{6}=0$.
We examine the condition that $2\in \im(\varphi_\lambda^J)$.
By taking $\Psi=\{\alpha_{6},\alpha_{1},\alpha_{3},\alpha_{4},-\ep_{3}+\ep_{5}\}$ and the same argument as proof of Lemma \ref{n2lem}, we have that exactly one of $a_{1}=1$ or $a_{3}=1$ or $a_{4}=1$ or $a_{5}=0$ holds.
f $a_{1}=1$ (resp. $a_{3}=1$, $a_{4}=1$, $a_{5}=0$), we see that $\{-\ep_{1}+\ep_{5},\beta_{4}\}$ (resp. $\{\ep_{1}+\ep_{5},\ep_{2}+\ep_{4}\}$, $\{\gamma_{3,4},\gamma_{2,5}\}$, $\{\ep_{2}+\ep_{3},\ep_{2}+\ep_{4}\}$) is a bad pair. 
\end{proof}
\begin{prop}There are no initialized irreducible homogeneous Ulrich bundles on $E_{6}/P_{1,3,5,6}$.
\end{prop}
\begin{proof}Suppose for contradiction that there is an initialized irreducible homogeneous Ulrich bundle $E_{\lambda}$ on $E_{6}/P_{1,3,5,6}$ with $\lambda=\sum_{i=1}^{6}a_{i}\varpi_{i}$.
By the same argument as the proof of Lemma \ref{n1lem}, we deduce that exactly one of $a_{1}=0$ or $a_{3}=0$ or $a_{5}=0$ or $a_{6}=0$ holds.

Assume that $a_{1}=0$.
We examine the condition that $2\in \im(\varphi_\lambda^J)$.
By taking $\Psi=\{\alpha_{1},\beta_{2},\alpha_{5},\alpha_{6}\}$ and the same argument as proof of Lemma \ref{2lem}, we obtain that exactly one of $a_{3}=2$ or $a_{5}=1$ or $a_{6}=1$ holds.
If $a_{5}=1$ (resp. $a_{6}=1$), we see that $\{\ep_{2}+\ep_{4},\gamma_{4,5}\}$ (resp. $\{\beta_{4},-\ep_{1}+\ep_{5}\}$) is a bad pair.
Suppose that $a_{3}=2$.
We turn to the condition that $4\in \im(\varphi_\lambda^J)$.
We obtain that exactly one of $a_{4}=3$ or $a_{5}=3$ or $a_{6}=3$ holds by taking $\Psi=\{\alpha_{1},\beta_{2},\alpha_{3},\beta_{3},\alpha_{5},\alpha_{6}\}$ and the same argument as the proof of Lemma \ref{22lem}.
If $a_{4}=3$ (resp. $a_{5}=3,\ a_{6}=3$), we can check that $\{\ep_{2}+\ep_{5},\ep_{3}+\ep_{5}\}$ (resp. $\{\ep_{2}+\ep_{4},\gamma_{4,5}\}$, $\{\ep_{1}+\ep_{5},\ep_{2}+\ep_{4}\}$) is a bad pair.

Suppose that $a_{3}=0$.
We examine the condition that $2\in \im(\varphi_\lambda^J)$.
By taking $\Psi=\{\alpha_{3},\beta_{2},\alpha_{5},\alpha_{6},-\ep_{1}+\ep_{3}\}$ and the same argument as proof of Lemma \ref{2lem}, we obtain that exactly one of $a_{1}=2$ or $a_{5}=1$ or $a_{6}=1$ or $a_{4}=0$ holds.
If $a_{5}=1$ (resp. $a_{6}=1$, $a_{4}=0$), we see that $\{\gamma_{2,3},\gamma_{1,4}\}$ (resp. $\{\ep_{1}+\ep_{5},\ep_{2}+\ep_{4}\}$, $\{\ep_{2}+\ep_{5},\ep_{3}+\ep_{5}\}$) is a bad pair.
Assume that $a_{1}=2$.
We turn to the condition that $2\in \im(\varphi_\lambda^J)$.
By taking $\Psi=\{\alpha_{3},\beta_{2},\alpha_{1},\beta_{3},\alpha_{5},\alpha_{6}\}$ and the same argument as proof of Lemma \ref{22lem}, we obtain that exactly one of $a_{4}=3$ or $a_{5}=3$ or $a_{6}=3$ holds.
If $a_{4}=3$ (resp. $a_{5}=3,\ a_{6}=3$), we can check that $\{\ep_{2}+\ep_{5},\ep_{3}+\ep_{5}\}$ (resp. $\{\ep_{2}+\ep_{4},\gamma_{4,5}\}$, $\{-\ep_{1}+\ep_{4},-\ep_{2}+\ep_{5}\}$) is a bad pair.

Assume that $a_{5}=0$.
We examine the condition that $2\in \im(\varphi_\lambda^J)$.
By taking $\Psi=\{\alpha_{5},\alpha_{1},-\ep_{3}+\ep_{5},\alpha_{3}\}$ and the same argument as proof of Lemma \ref{2lem}, we obtain that exactly one of $a_{6}=2$ or $a_{1}=1$ or $a_{3}=1$ or $a_{4}=0$ holds.
If $a_{1}=1$ (resp. $a_{3}=1$, $a_{4}=0$), we see that $\{-\ep_{1}+\ep_{4},\beta_{3}\}$ (resp. $\{\gamma_{2,3},\gamma_{1,4}\}$, $\{\ep_{2}+\ep_{4},\ep_{3}+\ep_{4}\}$) is a bad pair.
Suppose that $a_{6}=2$.
We turn to the condition that $4\in \im(\varphi_\lambda^J)$.
By taking $\Psi=\{\alpha_{5},-\ep_{3}+\ep_{5},\alpha_{6},\alpha_{1},-\ep_{2}+\ep_{5},\alpha_{3}\}$ and the same argument as proof of Lemma \ref{22lem}, we obtain that exactly one of $a_{1}=3$ or $a_{3}=3$ or $a_{4}=3$ holds.
If $a_{1}=3$ (resp. $a_{3}=3,\ a_{4}=3$), we can check that $\{-\ep_{1}+\ep_{4},\beta_{3}\}$ (resp. $\{\ep_{1}+\ep_{5},\ep_{2}+\ep_{4}\}$, $\{\ep_{2}+\ep_{5},\ep_{3}+\ep_{5}\}$) is a bad pair.

Assume that $a_{6}=0$.
We examine the condition that $2\in \im(\varphi_\lambda^J)$.
By taking $\Psi=\{\alpha_{6},-\ep_{3}+\ep_{5},\alpha_{1},\alpha_{3}\}$ and the same argument as proof of Lemma \ref{2lem}, we obtain that exactly one of $a_{5}=2$ or $a_{1}=1$ or $a_{3}=1$ holds.
If $a_{1}=1$ (resp. $a_{3}=1$), we see that $\{-\ep_{1}+\ep_{5},\beta_{4}\}$ (resp. $\{\ep_{2}+\ep_{4},\ep_{1}+\ep_{5}\}$) is a bad pair.
Suppose that $a_{5}=2$.
We turn to the condition that $4\in \im(\varphi_\lambda^J)$.
By taking $\Psi=\{\alpha_{5},-\ep_{3}+\ep_{5},\alpha_{6},\alpha_{1},-\ep_{2}+\ep_{5},\alpha_{3}\}$ and the same argument as proof of Lemma \ref{22lem}, we obtain that exactly one of $a_{1}=3$ or $a_{3}=3$ or $a_{4}=3$ holds.
If $a_{1}=3$ (resp. $a_{3}=3,\ a_{4}=3$), we can check that $\{-\ep_{1}+\ep_{4},\beta_{3}\}$ (resp. $\{\ep_{1}+\ep_{5},\ep_{2}+\ep_{4}\}$, $\{\ep_{2}+\ep_{5},\ep_{3}+\ep_{5}\}$) is a bad pair.
\end{proof}
\begin{prop}There are no initialized irreducible homogeneous Ulrich bundles on $E_{6}/P_{2,3,4,5}$.
\end{prop}
\begin{proof}Suppose for contradiction that there is an initialized irreducible homogeneous Ulrich bundle $E_{\lambda}$ on $E_{6}/P_{2,3,4,5}$ with $\lambda=\sum_{i=1}^{6}a_{i}\varpi_{i}$.
By the same argument as the proof of Lemma \ref{n1lem}, we deduce that exactly one of $a_{2}=0$ or $a_{3}=0$ or $a_{4}=0$ or $a_{5}=0$ holds.

Let us consider the case $a_{2}=0$ first.
We examine the condition that $2\in \im(\varphi_\lambda^J)$.
By taking $\Psi=\{\alpha_{2},\alpha_{3},\alpha_{5},\ep_{1}+\ep_{3}\}$ and the same argument as proof of Lemma \ref{2lem}, we obtain that exactly one of $a_{4}=2$ or $a_{3}=1$ or $a_{5}=1$ holds.
If $a_{3}=1$ (resp. $a_{5}=1$), we see that $\{-\ep_{1}+\ep_{4},\ep_{1}+\ep_{4}\}$ (resp. $\{\ep_{2}+\ep_{3},-\ep_{1}+\ep_{4}\}$) is a bad pair.
When $a_{4}=2$, we see that this contradicts our hypothesis by Lemma \ref{lem6}.

Next, let us consider the case $a_{3}=0$.
We turn to the condition that $2\in \im(\varphi_\lambda^J)$.
By taking $\Psi=\{\alpha_{3},\alpha_{2},\alpha_{5},-\ep_{1}+\ep_{3},\beta_{2}\}$ and the same argument as proof of Lemma \ref{2lem}, we have that exactly one of $a_{4}=2$ or $a_{2}=1$ or $a_{5}=1$ or $a_{1}=0$ holds.
If $a_{2}=1$ (resp. $a_{5}=1$, $a_{1}=0$), we see that $\{\ep_{1}+\ep_{4},-\ep_{1}+\ep_{4}\}$ (resp. $\{\ep_{2}+\ep_{3},\ep_{1}+\ep_{4}\}$, $\{-\ep_{1}+\ep_{4},\beta_{4}\}$) is a bad pair, respectively.
Suppose that $a_{4}=2$.
We investigate the condition that $4\in \im(\varphi_\lambda^J)$.
By $\varphi_{\lambda}^{J}(\ep_{1}+\ep_{3})=\frac{4+a_{2}}{2}$, and $\varphi_{\lambda}^{J}(\ep_{2}+\ep_{3})=\frac{a_{2}+5}{3}$, and
$\varphi_{\lambda}^{J}(-\ep_{2}+\ep_{4})=\frac{4+a_{5}}{2}$, and $\varphi_{\lambda}^{J}(-\ep_{1}+\ep_{4})=\frac{a_{5}+5}{3}$, and 
the same argument as the proof of Lemma \ref{22lem}, we deduce that $a_{2}$ and $a_{5}$ are grater than or equal to $10$.
In this case, we have $a_{1}=3$ by taking $\Psi=\{\alpha_{3},-\ep_{1}+\ep_{3},\alpha_{4},\beta_{3}\}$ and the same argument as the proof of Lemma \ref{22lem}.
If $a_{1}=3$, we can check that $\{-\ep_{1}+\ep_{4},\beta_{4}\}$ is a bad pair.

Next, let us consider the case $a_{4}=0$.
We turn to the condition that $2\in \im(\varphi_\lambda^J)$.
By taking $\Psi=\{\alpha_{4},\ep_{1}+\ep_{3},-\ep_{1}+\ep_{3},-\ep_{2}+\ep_{4}\}$ and the same argument as proof of Lemma \ref{2lem}, we have that exactly one of $a_{2}=2$ or $a_{3}=2$ or $a_{5}=2$ holds.
Assume that $a_{2}=2$.
By $\varphi_{\lambda}^{J}(-\ep_{1}+\ep_{3})=\frac{2+a_{3}}{2}$, and $\varphi_{\lambda}^{J}(\ep_{2}+\ep_{3})=\frac{a_{3}+5}{3}$, and
$\varphi_{\lambda}^{J}(-\ep_{2}+\ep_{4})=\frac{2+a_{5}}{2}$, and $\varphi_{\lambda}^{J}(\ep_{1}+\ep_{4})=\frac{a_{5}+5}{3}$, and 
the same argument as the proof of Lemma \ref{22lem}, we deduce that $a_{3}$ and $a_{5}$ are grater than or equal to $10$.
By taking $\Psi=\{\alpha_{4},\ep_{1}+\ep_{3},\alpha_{2}\}$ and the same argument as Lemma \ref{lem6}, we see that this case contradicts our hypothesis.
Suppose that $a_{3}=2$.
We investigate the condition that $4\in \im(\varphi_\lambda^J)$.
By $\varphi_{\lambda}^{J}(\ep_{1}+\ep_{3})=\frac{2+a_{2}}{2}$, and $\varphi_{\lambda}^{J}(\ep_{2}+\ep_{3})=\frac{a_{2}+5}{3}$, and
$\varphi_{\lambda}^{J}(-\ep_{2}+\ep_{4})=\frac{2+a_{5}}{2}$, and $\varphi_{\lambda}^{J}(-\ep_{1}+\ep_{4})=\frac{a_{5}+5}{3}$, and 
the same argument as the proof of Lemma \ref{22lem},
we deduce that $a_{2}$ and $a_{5}$ is grater than or equal to $10$.
In this case, we have $a_{1}=3$ by taking $\Psi=\{\alpha_{4},\alpha_{3},\beta_{3}\}$ and the same argument as the proof of Lemma \ref{22lem}.
If $a_{1}=3$, we can check that $\{-\ep_{1}+\ep_{4},\beta_{4}\}$ is a bad pair.
Assume that $a_{5}=2$.
We investigate the condition that $4\in \im(\varphi_\lambda^J)$.
By $\varphi_{\lambda}^{J}(\ep_{1}+\ep_{3})=\frac{2+a_{2}}{2}$, and $\varphi_{\lambda}^{J}(\ep_{1}+\ep_{4})=\frac{a_{2}+5}{3}$, and
$\varphi_{\lambda}^{J}(-\ep_{1}+\ep_{3})=\frac{2+a_{3}}{2}$, and $\varphi_{\lambda}^{J}(-\ep_{1}+\ep_{4})=\frac{a_{3}+5}{3}$, and 
the same argument as the proof of Lemma \ref{22lem},
we deduce that $a_{2}$ and $a_{3}$ are grater than or equal to $10$.
In this case, we have $a_{6}=3$ by taking $\Psi=\{\alpha_{4},-\ep_{2}+\ep_{4},\alpha_{5},-\ep_{2}+\ep_{5}\}$ and the same argument as the proof of Lemma \ref{22lem}.
If $a_{6}=3$, we can check that $\{\ep_{1}+\ep_{4},\ep_{1}+\ep_{5}\}$ is a bad pair.

Finally, let us consider the case $a_{5}=0$.
We examine the condition that $2\in \im(\varphi_\lambda^J)$.
By taking $\Psi=\{\alpha_{5},\alpha_{2},\alpha_{3},-\ep_{3}+\ep_{5},-\ep_{2}+\ep_{4}\}$ and the same argument as proof of Lemma \ref{2lem}, we have that exactly one of $a_{4}=2$ or $a_{2}=1$ or $a_{3}=1$ or $a_{6}=0$ holds.
If $a_{2}=1$ (resp. $a_{3}=1$, $a_{6}=0$), we see that $\{\gamma_{4,5},\beta_{4}\}$ (resp. $\{\ep_{1}+\ep_{4},\ep_{2}+\ep_{3}\}$, $\{\ep_{1}+\ep_{4},\ep_{1}+\ep_{5}\}$) is a bad pair.
Suppose that $a_{4}=2$.
We investigate the condition that $4\in \im(\varphi_\lambda^J)$.
By $\varphi_{\lambda}^{J}(\ep_{1}+\ep_{3})=\frac{4+a_{2}}{2}$, and $\varphi_{\lambda}^{J}(\ep_{1}+\ep_{4})=\frac{a_{2}+5}{3}$, and
$\varphi_{\lambda}^{J}(-\ep_{1}+\ep_{3})=\frac{4+a_{3}}{2}$, and $\varphi_{\lambda}^{J}(-\ep_{1}+\ep_{4})=\frac{a_{3}+5}{3}$, and 
the same argument as the proof of Lemma \ref{22lem},
we deduce that $a_{2}$ and $a_{3}$ are grater than or equal to $10$.
In this case, we have $a_{6}=3$ by taking $\Psi=\{\alpha_{5},-\ep_{2}+\ep_{4},\alpha_{4},-\ep_{2}+\ep_{5}\}$ and the same argument as the proof of Lemma \ref{22lem}.
If $a_{6}=3$, we can check that $\{\ep_{1}+\ep_{4},\ep_{1}+\ep_{5}\}$ is a bad pair, respectively.
\end{proof}
\begin{lem}\label{lem6}Let $E_{\mu}$ be an initialized irreducible homogeneous vector bundle on $E_{6}/P_{2,3,4,5}$ with $\mu=a_{1}\varpi_{1}+a_{3}\varpi_{3}+a_{5}\varpi_{5}+a_{6}\varpi_{6}+2\cdot\varpi_{4}$.
Then $E_{\mu}$ is not Ulrich.
\end{lem}
\begin{proof}
By $\varphi_{\mu}^{J}(-\ep_{1}+\ep_{3})=\frac{4+a_{3}}{2}$, and $\varphi_{\mu}^{J}(\ep_{2}+\ep_{3})=\frac{a_{3}+5}{3}$, and
$\varphi_{\mu}^{J}(-\ep_{2}+\ep_{4})=\frac{2+a_{5}}{2}$, and $\varphi_{\mu}^{J}(\ep_{1}+\ep_{4})=\frac{a_{5}+5}{3}$, and 
the same argument as the proof of Lemma \ref{22lem}, we deduce that $a_{3}$ and $a_{5}$ are grater than or equal to $10$.
Let $\Psi=\{\alpha_{2},\ep_{1}+\ep_{3},\alpha_{4}\}$.
Then for every $\alpha\in\Phi_{2,3,4,5}^{+}\setminus\Psi$ we have $\varphi_{\mu}^{J}(\alpha)>4$.
In other words, there is not a positive root $\alpha\in\Phi_{2,3,4,5}^{+}\setminus\Psi$ such that $\varphi_{\mu}^{J}(\alpha)=4$.
By Lemma \ref{lem}, $E_{\mu}$ is not Ulrich.
\end{proof}
\subsection{Proof of theorem \ref{thmE6} when $|J|=5$}
\begin{prop}There are no initialized irreducible homogeneous Ulrich bundles on $E_{6}/P_{1,2,3,4,5}\cong E_{6}/P_{2,3,4,5,6}$.
\end{prop}
\begin{proof}Suppose for contradiction that there is an initialized irreducible homogeneous Ulrich bundle $E_{\lambda}$ on $E_{6}/P_{1,2,3,4,5}$ with $\lambda=\sum_{i=1}^{6}a_{i}\varpi_{i}$.
By the same argument as the proof of Lemma \ref{n1lem}, we deduce that exactly one of $a_{1}=0$ or $a_{2}=0$ or $a_{3}=0$ or $a_{4}=0$ or $a_{5}=0$ holds.

Assume that $a_{1}=0$.
We examine the condition that $2\in \im(\varphi_\lambda^J)$.
By taking $\Psi=\{\alpha_{1},\beta_{2},\alpha_{2},\alpha_{4},\alpha_{5}\}$ and the same argument as proof of Lemma \ref{2lem}, we have that exactly one of $a_{3}=2$ or $a_{2}=1$ or $a_{4}=1$ or $a_{5}=1$ holds.
If $a_{2}=1$ (resp. $a_{4}=1$, $a_{5}=1$), we see that $\{\ep_{2}+\ep_{3},\beta_{3}\}$ (resp. $\{\ep_{3}+\ep_{5},\gamma_{3,4}\}$, $\{-\ep_{1}+\ep_{4},\beta_{3}\}$) is a bad pair.
Suppose that $a_{3}=2$.
We investigate the condition that $4\in \im(\varphi_\lambda^J)$.
By $\varphi_{\mu}^{J}(-\ep_{1}+\ep_{3})=\frac{4+a_{4}}{2}$, and $\varphi_{\mu}^{J}(\beta_{3})=\frac{a_{4}+5}{3}$, and 
the same argument as the proof of Lemma \ref{22lem},
we deduce that $a_{4}$ is grater than or equal to $10$.
In this case, we have that exactly one of $a_{2}=3$ or $a_{5}=3$ holds by taking $\Psi=\{\alpha_{1},\beta_{2},\alpha_{3},\alpha_{2},\alpha_{5}\}$ and the same argument as the proof of Lemma \ref{22lem}.
If $a_{2}=3$ (resp. $a_{5}=3$), we can check that $\{\ep_{1}+\ep_{4},-\ep_{1}+\ep_{4}\}$ (resp. $\{\ep_{1}+\ep_{4},\ep_{2}+\ep_{3}\}$) is a bad pair.

Suppose that $a_{2}=0$.
We turn to the condition that $2\in \im(\varphi_\lambda^J)$.
By taking $\Psi=\{\alpha_{2},\ep_{1}+\ep_{3},\alpha_{1},\alpha_{3},\alpha_{5}\}$ and the same argument as proof of Lemma \ref{2lem}, we have that exactly one of $a_{4}=2$ or $a_{1}=1$ or $a_{3}=1$ or $a_{5}=1$ holds.
If $a_{1}=1$ (resp. $a_{3}=1$, $a_{5}=1$), we see that $\{\ep_{2}+\ep_{3},\beta_{3}\}$ (resp. $\{\ep_{1}+\ep_{4},-\ep_{1}+\ep_{4}\}$, $\{\ep_{2}+\ep_{3},-\ep_{1}+\ep_{4}\}$) is a bad pair.
Assume that $a_{4}=2$.
We investigate the condition that $4\in \im(\varphi_\lambda^J)$.
By $\varphi_{\lambda}^{J}(-\ep_{1}+\ep_{3})=\frac{4+a_{3}}{2}$, and $\varphi_{\lambda}^{J}(\ep_{2}+\ep_{3})=\frac{a_{3}+5}{3}$, and
$\varphi_{\lambda}^{J}(-\ep_{2}+\ep_{4})=\frac{4+a_{5}}{2}$, and $\varphi_{\lambda}^{J}(\ep_{1}+\ep_{4})=\frac{a_{4}+5}{3}$, and 
the same argument as the proof of Lemma \ref{22lem},
we deduce that $a_{3}$ and $a_{5}$ are grater than or equal to $10$.
In this case, we have $a_{1}=3$ by taking $\Psi=\{\alpha_{2},\ep_{1}+\ep_{3},\alpha_{4},\alpha_{1}\}$ and the same argument as the proof of Lemma \ref{22lem}.
If $a_{1}=3$, we can check that $\{\ep_{2}+\ep_{4},\beta_{4}\}$ is a bad pair.

Assume that $a_{3}=0$.
We examine the condition that $2\in \im(\varphi_\lambda^J)$.
By taking $\Psi=\{\alpha_{3},\beta_{2},-\ep_{1}+\ep_{3},\alpha_{2},\alpha_{5}\}$ and the same argument as proof of Lemma \ref{2lem}, we have that exactly one of $a_{1}=2$ or $a_{4}=2$ or $a_{2}=1$ or $a_{5}=1$ holds.
If $a_{2}=1$ (resp. $a_{5}=1$), we see that $\{\ep_{1}+\ep_{3},-\ep_{1}+\ep_{3}\}$ (resp. $\{-\ep_{1}+\ep_{3},-\ep_{2}+\ep_{4}\}$) is a bad pair.
Suppose that $a_{4}=2$.
By $\varphi_{\lambda}^{J}(\beta_{2})=\frac{2+a_{1}}{2}$, and $\varphi_{\lambda}^{J}(\beta_{3})=\frac{a_{1}+5}{3}$, and
$\varphi_{\lambda}^{J}(\ep_{1}+\ep_{3})=\frac{4+a_{2}}{2}$, and $\varphi_{\lambda}^{J}(\ep_{2}+\ep_{3})=\frac{a_{2}+5}{3}$, and
$\varphi_{\lambda}^{J}(-\ep_{2}+\ep_{4})=\frac{4+a_{5}}{2}$, and $\varphi_{\lambda}^{J}(-\ep_{1}+\ep_{4})=\frac{a_{5}+5}{3}$, and
the same argument as the proof of Lemma \ref{22lem},
we deduce that $a_{1}$, $a_{2}$ and $a_{5}$ are grater than or equal to $10$.
Therefore by taking $\Psi=\{\alpha_{3},-\ep_{1}+\ep_{3},\alpha_{4}\}$ and the same argument as proof of Lemma \ref{lem6}, we see that there is not positive root $\alpha\in\Phi_{1,2,3,4,5}^{+}\setminus\Psi$ such that $\varphi_{\lambda}^{J}(\alpha)=4$.
Assume that $a_{1}=2$.
We investigate the condition that $4\in \im(\varphi_\lambda^J)$.
By $\varphi_{\lambda}^{J}(-\ep_{1}+\ep_{3})=\frac{2+a_{4}}{2}$, and $\varphi_{\lambda}^{J}(\beta_{3})=\frac{a_{3}+5}{3}$, and 
the same argument as the proof of Lemma \ref{22lem},
we deduce that $a_{4}$ is grater than or equal to $10$.
In this case, we have that exactly one of $a_{2}=3$ or $a_{5}=3$ holds by taking $\Psi=\{\alpha_{3},\beta_{2},\alpha_{1},\alpha_{2},\alpha_{5}\}$ and the same argument as the proof of Lemma \ref{22lem}.
If $a_{2}=3$ (resp. $a_{5}=3$), we can check that $\{\ep_{2}+\ep_{3},\beta_{4}\}$ (resp. $\{\ep_{2}+\ep_{4},\gamma_{4,5}\}$) is a bad pair.

Suppose that $a_{4}=0$.
We turn to the condition that $2\in \im(\varphi_\lambda^J)$.
By taking $\Psi=\{\alpha_{4},\}$ and the same argument as proof of Lemma \ref{2lem}, we have that exactly one of $a_{2}=2$ or $a_{3}=2$ or $a_{5}=2$ or $a_{1}=1$ holds.
If $a_{1}=1$, we see that $\{-\ep_{1}+\ep_{3},\beta_{2}\}$ is a bad pair.
Suppose that $a_{2}=2$.
We investigate the condition that $4\in \im(\varphi_\lambda^J)$.
By $\varphi_{\lambda}^{J}(-\ep_{1}+\ep_{3})=\frac{2+a_{3}}{2}$, and $\varphi_{\lambda}^{J}(\ep_{2}+\ep_{3})=\frac{a_{3}+5}{3}$, and
$\varphi_{\lambda}^{J}(-\ep_{2}+\ep_{4})=\frac{2+a_{5}}{2}$, and $\varphi_{\lambda}^{J}(\ep_{1}+\ep_{4})=\frac{a_{4}+5}{3}$, and 
the same argument as the proof of Lemma \ref{22lem},
we deduce that $a_{3}$ and $a_{5}$ are grater than or equal to $10$.
In this case, we have $a_{1}=3$ by taking $\Psi=\{\alpha_{4},\ep_{1}+\ep_{3},\alpha_{2},\alpha_{1}\}$ and the same argument as the proof of Lemma \ref{22lem}.
If $a_{1}=3$, we can check that $\{-\ep_{1}+\ep_{3},\beta_{2}\}$ is a bad pair.
Suppose that $a_{3}=2$.
By $\varphi_{\lambda}^{J}(\beta_{2})=\frac{4+a_{1}}{2}$, and $\varphi_{\lambda}^{J}(\beta_{3})=\frac{a_{1}+5}{3}$, and
$\varphi_{\lambda}^{J}(\ep_{1}+\ep_{3})=\frac{2+a_{2}}{2}$, and $\varphi_{\lambda}^{J}(\ep_{2}+\ep_{3})=\frac{a_{2}+5}{3}$, and
$\varphi_{\lambda}^{J}(-\ep_{2}+\ep_{4})=\frac{2+a_{5}}{2}$, and $\varphi_{\lambda}^{J}(-\ep_{1}+\ep_{4})=\frac{a_{5}+5}{3}$, and
the same argument as the proof of Lemma \ref{22lem},
we deduce that $a_{1}$, $a_{2}$ and $a_{5}$ are grater than or equal to $10$.
Therefore by taking $\Psi=\{\alpha_{3},-\ep_{1}+\ep_{3},\alpha_{4}\}$ and the same argument as proof of Lemma \ref{lem6}, we see that there is not positive root $\alpha\in\Phi_{1,2,3,4,5}^{+}\setminus\Psi$ such that $\varphi_{\lambda}^{J}(\alpha)=4$.
Assume that $a_{5}=2$.
We investigate the condition that $4\in \im(\varphi_\lambda^J)$.
By $\varphi_{\lambda}^{J}(\ep_{1}+\ep_{3})=\frac{2+a_{2}}{2}$, and $\varphi_{\lambda}^{J}(\ep_{1}+\ep_{4})=\frac{a_{2}+5}{3}$, and
$\varphi_{\lambda}^{J}(-\ep_{1}+\ep_{3})=\frac{2+a_{3}}{2}$, and $\varphi_{\lambda}^{J}(-\ep_{1}+\ep_{4})=\frac{a_{3}+5}{3}$, and
the same argument as the proof of Lemma \ref{22lem},
we deduce that $a_{2}$ and $a_{3}$ are grater than or equal to $10$.
In this case, we have that exactly one of $a_{1}=3$ or $a_{6}=3$ holds by taking $\Psi=\{\alpha_{4},-\ep_{2}+\ep_{4},\alpha_{5},\alpha_{1},-\ep_{2}+\ep_{5}\}$ and the same argument as the proof of Lemma \ref{22lem}.
If $a_{1}=3$ (resp. $a_{6}=3$), we can check that $\{-\ep_{1}+\ep_{3},\beta_{2}\}$ (resp. $\{\ep_{1}+\ep_{4},\ep_{1}+\ep_{5}\}$) is a bad pair.

Assume that $a_{5}=0$.
We turn to the condition that $2\in \im(\varphi_\lambda^J)$.
By taking $\Psi=\{\alpha_{5},-\ep_{2}+\ep_{4},\alpha_{1},\alpha_{2},\alpha_{3},-\ep_{3}+\ep_{5}\}$ and the same argument as proof of Lemma \ref{2lem}, we have that exactly one of $a_{4}=2$ or $a_{1}=1$ or $a_{2}=1$ or $a_{3}=1$ or $a_{6}=0$ holds.
If $a_{1}=1$ (resp. $a_{2}=1$, $a_{3}=1$, $a_{6}=0$), we see that $\{-\ep_{1}+\ep_{4},\beta_{3}\}$ (resp. $\{\ep_{2}+\ep_{3},-\ep_{1}+\ep_{4}\}$, $\{\ep_{1}+\ep_{4},\ep_{2}+\ep_{3}\}$, $\{\ep_{1}+\ep_{4},\ep_{1}+\ep_{5}\}$) is a bad pair.
Suppose that $a_{4}=2$.
We investigate the condition that $4\in \im(\varphi_\lambda^J)$.
By $\varphi_{\lambda}^{J}(\ep_{1}+\ep_{3})=\frac{4+a_{2}}{2}$, and $\varphi_{\lambda}^{J}(\ep_{1}+\ep_{4})=\frac{a_{2}+5}{3}$, and
$\varphi_{\lambda}^{J}(-\ep_{1}+\ep_{3})=\frac{4+a_{3}}{2}$, and $\varphi_{\lambda}^{J}(-\ep_{1}+\ep_{4})=\frac{a_{3}+5}{3}$, and
the same argument as the proof of Lemma \ref{22lem},
we deduce that $a_{2}$ and $a_{3}$ are grater than or equal to $10$.
In this case, we have that exactly one of $a_{1}=3$ or $a_{6}=3$ holds by taking $\Psi=\{\alpha_{5},-\ep_{2}+\ep_{4},\alpha_{4},\alpha_{1},-\ep_{2}+\ep_{5}\}$ and the same argument as the proof of Lemma \ref{22lem}.
If $a_{1}=3$ (resp. $a_{6}=3$), we can check that $\{-\ep_{1}+\ep_{3},\beta_{2}\}$ (resp. $\{\ep_{3}+\ep_{4},\ep_{3}+\ep_{5}\}$) is a bad pair.
\end{proof}
\begin{prop}There are no initialized irreducible homogeneous Ulrich bundles on $E_{6}/P_{1,2,3,4,6}\cong E_{6}/P_{1,2,4,5,6}$.
\end{prop}
\begin{proof}Suppose for contradiction that there is an initialized irreducible homogeneous Ulrich bundle $E_{\lambda}$ on $E_{6}/P_{1,2,3,4,6}$ with $\lambda=\sum_{i=1}^{6}a_{i}\varpi_{i}$.
By the same argument as the proof of Lemma \ref{n1lem}, we deduce that exactly one of $a_{1}=0$ or $a_{2}=0$ or $a_{3}=0$ or $a_{4}=0$ or $a_{6}=0$ holds.

Assume that $a_{1}=0$.
We examine the condition that $2\in \im(\varphi_\lambda^J)$.
By taking $\Psi=\{\alpha_{1},\beta_{2},\alpha_{2},\alpha_{4},\alpha_{6}\}$ and the same argument as proof of Lemma \ref{2lem}, we have that exactly one of $a_{3}=2$ or $a_{2}=1$ or $a_{4}=1$ or $a_{6}=1$ holds.
If $a_{2}=1$ (resp. $a_{4}=1$, $a_{6}=1$), we see that $\{\ep_{2}+\ep_{3},\beta_{3}\}$ (resp. $\{-\ep_{1}+\ep_{3},\beta_{2}\}$, $\{-\ep_{1}+\ep_{5},\beta_{4}\}$) is a bad pair.
Suppose that $a_{3}=2$.
We investigate the condition that $4\in \im(\varphi_\lambda^J)$.
By $\varphi_{\mu}^{J}(-\ep_{1}+\ep_{3})=\frac{4+a_{4}}{2}$, and $\varphi_{\mu}^{J}(\beta_{3})=\frac{a_{4}+5}{3}$, and
the same argument as the proof of Lemma \ref{22lem},
we deduce that $a_{4}$ is grater than or equal to $10$.
In this case, we have that exactly one of $a_{2}=3$ or $a_{6}=3$ holds by taking $\Psi=\{\alpha_{1},\beta_{2},\alpha_{3},\alpha_{2},\alpha_{6}\}$ and the same argument as the proof of Lemma \ref{22lem}.
If $a_{2}=3$ (resp. $a_{6}=3$), we can check that $\{\ep_{1}+\ep_{3},-\ep_{1}+\ep_{3}\}$ (resp. $\{\ep_{1}+\ep_{5},\ep_{2}+\ep_{4}\}$) is a bad pair.

Suppose that $a_{2}=0$.
We turn to the condition that $2\in \im(\varphi_\lambda^J)$.
By taking $\Psi=\{\alpha_{2},\ep_{1}+\ep_{3},\alpha_{1},\alpha_{3},\alpha_{6}\}$ and the same argument as proof of Lemma \ref{2lem}, we have that exactly one of $a_{4}=2$ or $a_{1}=1$ or $a_{3}=1$ or $a_{6}=1$ holds.
If $a_{1}=1$ (resp. $a_{3}=1$, $a_{6}=1$), we see that $\{\ep_{2}+\ep_{3},\beta_{3}\}$ (resp. $\{\ep_{1}+\ep_{3},-\ep_{1}+\ep_{3}\}$, $\{\ep_{1}+\ep_{4},-\ep_{2}+\ep_{5}\}$) is a bad pair.
Assume that $a_{4}=2$.
We investigate the condition that $4\in \im(\varphi_\lambda^J)$.
By $\varphi_{\mu}^{J}(-\ep_{1}+\ep_{3})=\frac{4+a_{3}}{2}$, and $\varphi_{\mu}^{J}(\ep_{2}+\ep_{3})=\frac{a_{3}+5}{3}$, and
the same argument as the proof of Lemma \ref{22lem},
we deduce that $a_{3}$ is grater than or equal to $10$.
In this case, we have that exactly one of $a_{1}=3$ or $a_{5}=3$ or $a_{6}=3$ holds by taking $\Psi=\{\alpha_{2},\ep_{1}+\ep_{3},\alpha_{4},\alpha_{1},\ep_{1}+\ep_{4},\alpha_{6}\}$ and the same argument as the proof of Lemma \ref{22lem}.
If $a_{1}=3$ (resp. $a_{5}=3$, $a_{6}=3$), we can check that $\{-\ep_{1}+\ep_{3},\beta_{2}\}$ (resp. $\{\ep_{2}+\ep_{3},\ep_{2}+\ep_{4}\}$, $\{\ep_{1}+\ep_{4},-\ep_{2}+\ep_{5}\}$) is a bad pair.

Assume that $a_{3}=0$.
We turn to the condition that $2\in \im(\varphi_\lambda^J)$.
By taking $\Psi=\{\alpha_{3},\beta_{2},-\ep_{1}+\ep_{3},\alpha_{2},\alpha_{6}\}$ and the same argument as proof of Lemma \ref{2lem}, we have that exactly one of $a_{1}=2$ or $a_{4}=2$ or $a_{2}=1$ or $a_{6}=1$ holds.
If $a_{2}=1$ (resp. $a_{6}=1$), we see that $\{\ep_{1}+\ep_{3},\ep_{1}+\ep_{3}\}$ (resp. $\{-\ep_{1}+\ep_{4},-\ep_{2}+\ep_{5}\}$) is a bad pair.
Suppose that $a_{1}=2$.
We investigate the condition that $4\in \im(\varphi_\lambda^J)$.
By $\varphi_{\lambda}^{J}(-\ep_{1}+\ep_{3})=\frac{2+a_{4}}{2}$, and $\varphi_{\lambda}^{J}(\beta_{3})=\frac{a_{4}+5}{3}$, and
the same argument as the proof of Lemma \ref{22lem},
we deduce that $a_{4}$ is grater than or equal to $10$.
In this case, we have that exactly one of $a_{2}=3$ or $a_{6}=3$ holds by taking $\Psi=\{\alpha_{3},\beta_{2},\alpha_{1},\alpha_{2},\alpha_{6}\}$ and the same argument as the proof of Lemma \ref{22lem}.
If $a_{2}=3$ (resp. $a_{6}=3$), we can check that $\{\ep_{1}+\ep_{3},-\ep_{1}+\ep_{3}\}$ (resp. $\{-\ep_{1}+\ep_{4},-\ep_{2}+\ep_{5}\}$) is a bad pair.
Assume that $a_{4}=2$.
We also investigate the condition that $4\in \im(\varphi_\lambda^J)$.
By $\varphi_{\lambda}^{J}(\beta_{2})=\frac{2+a_{1}}{2}$, and $\varphi_{\lambda}^{J}(\beta_{3})=\frac{a_{1}+5}{3}$, and
$\varphi_{\lambda}^{J}(\ep_{1}+\ep_{3})=\frac{4+a_{2}}{2}$, and $\varphi_{\lambda}^{J}(\ep_{2}+\ep_{3})=\frac{a_{2}+5}{3}$, and
the same argument as the proof of Lemma \ref{22lem},
we deduce that $a_{1}$ and $a_{2}$ is grater than or equal to $10$.
In this case, we have that exactly one of $a_{5}=3$ or $a_{6}=3$ holds by taking $\Psi=\{\alpha_{3},-\ep_{1}+\ep_{3},\alpha_{4},-\ep_{1}+\ep_{4},\alpha_{6}\}$ and the same argument as the proof of Lemma \ref{22lem}.
If $a_{5}=3$ (resp. $a_{6}=3$), we can check that $\{\ep_{2}+\ep_{3},\ep_{2}+\ep_{4}\}$ (resp. $\{-\ep_{1}+\ep_{4},-\ep_{2}+\ep_{5}\}$) is a bad pair. 

Suppose that $a_{4}=0$.
We turn to the condition that $2\in \im(\varphi_\lambda^J)$.
By taking $\Psi=\{\alpha_{4},\ep_{1}+\ep_{3},-\ep_{1}+\ep_{3},\alpha_{1},\alpha_{6},-\ep{2}+\ep_{4}\}$ and the same argument as proof of Lemma \ref{2lem}, we have that exactly one of $a_{2}=2$ or $a_{3}=2$ or $a_{1}=1$ or $a_{6}=1$ or $a_{5}=0$ holds.
If $a_{1}=1$ (resp. $a_{6}=1$, $a_{5}=0$), we see that $\{-\ep_{1}+\ep_{3},\beta_{2}\}$ (resp. $\{\ep_{2}+\ep_{5},\ep_{3}+\ep_{4}\}$, $\{\ep_{1}+\ep_{3},\ep_{1}+\ep_{4}\}$) is a bad pair.
Assume that $a_{2}=2$.
We investigate the condition that $4\in \im(\varphi_\lambda^J)$.
By $\varphi_{\lambda}^{J}(-\ep_{1}+\ep_{3})=\frac{2+a_{3}}{2}$, and $\varphi_{\lambda}^{J}(\ep_{2}+\ep_{3})=\frac{a_{3}+5}{3}$, and
the same argument as the proof of Lemma \ref{22lem},
we deduce that $a_{3}$ is grater than or equal to $10$.
In this case, we have that exactly one of $a_{1}=3$ or $a_{5}=3$ or $a_{6}=3$ holds by taking $\Psi=\{\alpha_{4},\ep_{1}+\ep_{3},\alpha_{2},\alpha_{1},\ep_{1}+\ep_{4},\alpha_{6}\}$ and the same argument as the proof of Lemma \ref{22lem}.
If $a_{1}=3$ (resp. $a_{5}=3$, $a_{6}=3$), we can check that $\{-\ep_{1}+\ep_{3},\beta_{2}\}$ (resp. $\{\ep_{2}+\ep_{3},\ep_{2}+\ep_{4}\}$, $\{\ep_{1}+\ep_{4},-\ep_{2}+\ep_{5}\}$) is a bad pair.
Assume that $a_{3}=2$.
We investigate the condition that $4\in \im(\varphi_\lambda^J)$.
By $\varphi_{\lambda}^{J}(\beta_{2})=\frac{4+a_{1}}{2}$, and $\varphi_{\lambda}^{J}(\beta_{3})=\frac{a_{1}+5}{3}$, and
$\varphi_{\lambda}^{J}(\ep_{1}+\ep_{3})=\frac{2+a_{2}}{2}$, and $\varphi_{\lambda}^{J}(\ep_{2}+\ep_{3})=\frac{a_{2}+5}{3}$, and
the same argument as the proof of Lemma \ref{22lem},
we deduce that $a_{1}$ and $a_{2}$ is grater than or equal to $10$.
In this case, we have that exactly one of $a_{5}=3$ or $a_{6}=3$ holds by taking $\Psi=\{\alpha_{4},-\ep_{1}+\ep_{3},\alpha_{3},-\ep_{1}+\ep_{4},\alpha_{6}\}$ and the same argument as the proof of Lemma \ref{22lem}.
If $a_{5}=3$ (resp. $a_{6}=3$), we can check that $\{\ep_{2}+\ep_{3},\ep_{2}+\ep_{4}\}$ (resp. $\{-\ep_{1}+\ep_{4},-\ep_{2}+\ep_{5}\}$) is a bad pair. 

Assume that $a_{6}=0$.
We examine the condition that $2\in \im(\varphi_\lambda^J)$.
By taking $\Psi=\{\alpha_{6},\alpha_{1},\alpha_{2},\alpha_{3},\alpha_{4},-\ep_{3}+\ep_{5}\}$ and the same argument as proof of Lemma \ref{n2lem}, we have that exactly one of $a_{1}=1$ or $a_{2}=1$ or $a_{3}=1$, $a_{4}=1$ or $a_{5}=0$ holds.
If $a_{1}=1$ (resp. $a_{2}=1$, $a_{3}=1$, $a_{4}=1$, $a_{5}=0$), we can check that $\{\ep_{2}+\ep_{5},\gamma_{3,5}\}$ (resp. $\{\ep_{1}+\ep_{4},-\ep_{2}+\ep_{5}\}$, $\{-\ep_{1}+\ep_{4},-\ep_{2}+\ep_{5}\}$, $\{\ep_{2}+\ep_{5},\ep_{3}+\ep_{4}\}$, $\{\ep_{1}+\ep_{3},\ep_{2}+\ep_{4}\}$) is a bad pair.
\end{proof}
\begin{prop}There are no initialized irreducible homogeneous Ulrich bundles on $E_{6}/P_{1,2,3,5,6}$.
\end{prop}
\begin{proof}Suppose for contradiction that there is an initialized irreducible homogeneous Ulrich bundle $E_{\lambda}$ on $E_{6}/P_{1,2,3,5,6}$ with $\lambda=\sum_{i=1}^{6}a_{i}\varpi_{i}$.
By the same argument as the proof of Lemma \ref{n1lem}, we deduce that exactly one of $a_{1}=0$ or $a_{2}=0$ or $a_{3}=0$ or $a_{5}=0$ or $a_{6}=0$ holds.

Assume that $a_{1}=0$.
We examine the condition that $2\in \im(\varphi_\lambda^J)$.
By taking $\Psi=\{\alpha_{1},\beta_{2},\alpha_{2},\alpha_{5},\alpha_{6}\}$ and the same argument as proof of Lemma \ref{2lem}, we have that exactly one of $a_{3}=2$ or $a_{2}=1$ or $a_{5}=1$ or $a_{6}=1$ holds.
If $a_{2}=1$ (resp. $a_{5}=1$, $a_{6}=1$), we see that $\{\ep_{2}+\ep_{3},\beta_{3}\}$ (resp. $\{-\ep_{1}+\ep_{4},\beta_{3}\}$, $\{-\ep_{1}+\ep_{5},\beta_{4}\}$) is a bad pair.
Suppose that $a_{3}=2$.
We investigate the condition that $4\in \im(\varphi_\lambda^J)$.
By taking $\Psi=\{\alpha_{1},\beta_{2},\alpha_{3},\alpha_{2},\beta_{3},\alpha_{5},\alpha_{6}\}$ and the same argument as the proof of Lemma \ref{22lem}, we have that exactly one of $a_{2}=3$ or $a_{4}=3$ or $a_{5}=3$ or $a_{6}=3$ holds.
If $a_{2}=3$ (resp. $a_{4}=3$, $a_{5}=3$, $a_{6}=3$), we can check that $\{\ep_{2}+\ep_{3},\beta_{3}\}$ (resp. $\{\gamma_{3,4},\gamma_{2,4}\}$, $\{\ep_{1}+\ep_{4},\ep_{2}+\ep_{3}\}$, $\{-\ep_{1}+\ep_{4},-\ep_{2}+\ep_{3}\}$) is a bad pair.

Assume that $a_{2}=0$.
We turn to the condition that $2\in \im(\varphi_\lambda^J)$.
By taking $\Psi=\{\alpha_{2},\alpha_{1},\alpha_{3},\alpha_{5},\alpha_{6},\ep_{1}+\ep_{3}\}$ and the same argument as proof of Lemma \ref{n2lem}, we have that exactly one of $a_{1}=1$ or $a_{3}=1$ or $a_{5}=1$ or $a_{6}=1$ or $a_{4}=0$ holds.
If $a_{1}=1$ (resp. $a_{3}=1$, $a_{5}=1$, $a_{6}=1$, $a_{4}=0$), we see that $\{\ep_{2}+\ep_{4},\beta_{4}\}$ (resp. $\{\ep_{1}+\ep_{4},-\ep_{1}+\ep_{4}\}$, $\{-\ep_{1}+\ep_{4},\ep_{2}+\ep_{3}\}$, $\{\ep_{1}+\ep_{4},-\ep_{2}+\ep_{5}\}$, $\{\ep_{2}+\ep_{4},\ep_{3}+\ep_{4}\}$) is a bad pair.

Suppose that $a_{3}=0$.
We examine the condition that $2\in \im(\varphi_\lambda^J)$.
By taking $\Psi=\{\alpha_{3},\beta_{2},\alpha_{2},\alpha_{5},\alpha_{6},-\ep_{1}+\ep_{3}\}$ and the same argument as proof of Lemma \ref{2lem}, we have that exactly one of $a_{1}=2$ or $a_{2}=1$ or $a_{5}=1$ or $a_{6}=1$ or $a_{4}=0$ holds.
If $a_{2}=1$ (resp. $a_{5}=1$, $a_{6}=1$, $a_{4}=0$), we see that $\{\ep_{1}+\ep_{4},-\ep_{1}+\ep_{4}\}$ (resp. $\{\ep_{1}+\ep_{4},\ep_{2}+\ep_{3}\}$, $\{-\ep_{1}+\ep_{4},-\ep_{2}+\ep_{5}\}$, $\{\ep_{2}+\ep_{4},\ep_{3}+\ep_{4}\}$) is a bad pair.
Assume that $a_{1}=2$.
By taking $\Psi=\{\alpha_{3},\beta_{2},\alpha_{1},\alpha_{2},\beta_{3},\alpha_{5},\alpha_{6}\}$ and the same argument as proof of Lemma \ref{22lem}, we have that exactly one of $a_{2}=3$ or $a_{4}=3$ or $a_{5}=3$ or $a_{6}=3$ holds.
If $a_{2}=3$ (resp. $a_{4}=3$, $a_{5}=3$, $a_{6}=3$), we can check that $\{\ep_{2}+\ep_{3},\beta_{3}\}$ (resp. $\{\ep_{2}+\ep_{4},\ep_{3}+\ep_{4}\}$, $\{\ep_{1}+\ep_{4},\ep_{2}+\ep_{3}\}$, $\{-\ep_{1}+\ep_{5},\beta_{4}\}$) is a bad pair.

Suppose that $a_{5}=0$.
We turn to the condition that $2\in \im(\varphi_\lambda^J)$.
By taking $\Psi=\{\alpha_{5},-\ep_{3}+\ep_{5},\alpha_{1},\alpha_{2},\alpha_{3},-\ep_{2}+\ep_{4}\}$ and he same argument as proof of Lemma \ref{2lem}, we have that exactly one of $a_{6}=2$ or $a_{1}=1$ or $a_{2}=1$ or $a_{3}=1$ or $a_{4}=0$ holds.
If $a_{1}=1$ (resp. $a_{2}=1$, $a_{3}=1$, $a_{4}=0$), we see that $\{\ep_{2}+\ep_{4},\gamma_{4,5}\}$ (resp. $\{\ep_{2}+\ep_{3},-\ep_{1}+\ep_{4}\}$, $\{\ep_{1}+\ep_{4},\ep_{2}+\ep_{3}\}$, $\{\ep_{2}+\ep_{4},\ep_{3}+\ep_{4}\}$) is a bad pair.
Assume that $a_{6}=2$.
By taking $\Psi=\{\alpha_{5},-\ep_{3}+\ep_{5},\alpha_{6},\alpha_{1},\alpha_{2},\alpha_{3},-\ep_{2}+\ep_{5}\}$ and the same argument as proof of Lemma \ref{22lem}, we have that exactly one of $a_{1}=3$ or $a_{2}=3$ or $a_{3}=3$ or $a_{4}=3$ holds.
If $a_{1}=3$ (resp. $a_{2}=3$, $a_{3}=3$, $a_{4}=3$), we can check that $\{-\ep_{1}+\ep_{4},\beta_{3}\}$ (resp. $\{\ep_{1}+\ep_{4},-\ep_{2}+\ep_{5}\}$, $\{\ep_{1}+\ep_{4},\ep_{2}+\ep_{3}\}$, $\{\ep_{2}+\ep_{4},\ep_{3}+\ep_{4}\}$) is a bad pair.

Suppose that $a_{6}=0$.
We examine the condition that $2\in \im(\varphi_\lambda^J)$.
By taking $\Psi=\{\alpha_{6},-\ep_{3}+\ep_{5},\alpha_{1},\alpha_{2},\alpha_{3}\}$ and the same argument as proof of Lemma \ref{2lem}, we have that exactly one of $a_{5}=2$ or $a_{1}=1$ or $a_{2}=1$ or $a_{3}=1$ holds.
If $a_{1}=1$ (resp. $a_{2}=1$, $a_{3}=1$), we see that $\{-\ep_{1}+\ep_{5},\beta_{4}\}$ (resp. $\{\ep_{1}+\ep_{4},-\ep_{2}+\ep_{5}\}$, $\{-\ep_{1}+\ep_{4},-\ep_{2}+\ep_{5}\}$) is a bad pair.
Assume that $a_{5}=2$.
By taking $\Psi=\{\alpha_{6},-\ep_{3}+\ep_{5},\alpha_{5},\alpha_{1},\alpha_{2},\alpha_{3},-\ep_{2}+\ep_{5}\}$ and the same argument as proof of Lemma \ref{22lem}, we have that exactly one of $a_{1}=3$ or $a_{2}=3$ or $a_{3}=3$ or $a_{4}=3$ holds.
If $a_{1}=3$ (resp. $a_{2}=3$, $a_{3}=3$, $a_{4}=3$), we can check that $\{-\ep_{1}+\ep_{4},\beta_{3}\}$ (resp. $\{\ep_{2}+\ep_{3},-\ep_{1}+\ep_{4}\}$, $\{-\ep_{1}+\ep_{4},-\ep_{2}+\ep_{5}\}$, $\{\ep_{2}+\ep_{4},\ep_{3}+\ep_{4}\}$) is a bad pair.
\end{proof}
\begin{prop}There are no initialized irreducible homogeneous Ulrich bundles on $E_{6}/P_{1,3,4,5,6}$.
\end{prop}
\begin{proof}Suppose for contradiction that there is an initialized irreducible homogeneous Ulrich bundle $E_{\lambda}$ on $E_{6}/P_{1,3,4,5,6}$ with $\lambda=\sum_{i=1}^{6}a_{i}\varpi_{i}$.
By the same argument as the proof of Lemma \ref{n1lem}, we deduce that exactly one of $a_{1}=0$ or $a_{3}=0$ or $a_{4}=0$ or $a_{5}=0$ or $a_{6}=0$ holds.

Assume that $a_{1}=0$.
We examine the condition that $2\in \im(\varphi_\lambda^J)$.
By taking $\Psi=\{\alpha_{1},\beta_{2},\alpha_{4},\alpha_{5},\alpha_{6}\}$ and the same argument as proof of Lemma \ref{2lem}, we have that exactly one of $a_{3}=2$ or $a_{4}=1$ or $a_{5}=1$ or $a_{6}=1$ holds.
If $a_{4}=1$ (resp. $a_{5}=1$, $a_{6}=1$), we see that $\{\ep_{3}+\ep_{4},\gamma_{3,5}\}$ (resp. $\{-\ep_{1}+\ep_{4},\beta_{3}\}$, $\{-\ep_{1}+\ep_{5},\beta_{4}\}$) is a bad pair.
Suppose that $a_{3}=2$.
We investigate the condition that $4\in \im(\varphi_\lambda^J)$.
By $\varphi_{\lambda}^{J}(-\ep_{1}+\ep_{3})=\frac{4+a_{4}}{2}$, and $\varphi_{\lambda}^{J}(\beta_{3})=\frac{a_{4}+5}{3}$, and
the same argument as the proof of Lemma \ref{22lem},
we deduce that $a_{4}$ is grater than or equal to $10$.
In this case, we have that exactly one of $a_{5}=3$ or $a_{6}=3$ holds by taking $\Psi=\{\alpha_{1},\beta_{2},\alpha_{3},\alpha_{5},\alpha_{6}\}$ and the same argument as the proof of Lemma \ref{22lem}.
If $a_{5}=3$ (resp. $a_{6}=3$), we can check that $\{-\ep_{1}+\ep_{3},-\ep_{2}+\ep_{4}\}$ (resp. $\{-\ep_{1}+\ep_{4},-\ep_{2}+\ep_{5}\}$) is a bad pair.

Assume that $a_{3}=0$.
We turn to the condition that $2\in \im(\varphi_\lambda^J)$.
By taking $\Psi=\{\alpha_{3},\beta_{2},-\ep_{1}+\ep_{3},\alpha_{5},\alpha_{6}\}$ and the same argument as proof of Lemma \ref{2lem}, we have that exactly one of$a_{1}=2$ or $a_{4}=2$ or $a_{5}=1$ or $a_{6}=1$ holds.
If $a_{5}=1$ (resp. $a_{6}=1$), we see that $\{\ep_{1}+\ep_{4},-\ep_{2}+\ep_{3}\}$ (resp. $\{\ep_{1}+\ep_{5},\ep_{2}+\ep_{4}\}$) is a bad pair.
Suppose that $a_{1}=2$.
We investigate the condition that $4\in \im(\varphi_\lambda^J)$.
By $\varphi_{\mu}^{J}(-\ep_{1}+\ep_{3})=\frac{2+a_{4}}{2}$, and $\varphi_{\mu}^{J}(\beta_{3})=\frac{a_{4}+5}{3}$, and
the same argument as the proof of Lemma \ref{22lem},
we deduce that $a_{4}$ is grater than or equal to $10$.
In this case, we have that exactly one of $a_{5}=3$ or $a_{6}=3$ holds by taking $\Psi=\{\alpha_{3},\beta_{2},\alpha_{1},\alpha_{5},\alpha_{6}\}$ and the same argument as the proof of Lemma \ref{22lem}.
If $a_{5}=3$ (resp. $a_{6}=3$), we can check that $\{-\ep_{1}+\ep_{3},-\ep_{2}+\ep_{4}\}$ (resp. $\{\gamma_{2,5},\ep_{2}+\ep_{5}\}$) is a bad pair.
Suppose that $a_{4}=2$.
We investigate the condition that $4\in \im(\varphi_\lambda^J)$.
By $\varphi_{\lambda}^{J}(\beta_{2})=\frac{2+a_{1}}{2}$, and $\varphi_{\lambda}^{J}(\beta_{3})=\frac{a_{1}+5}{3}$, and
$\varphi_{\lambda}^{J}(-\ep_{2}+\ep_{4})=\frac{4+a_{5}}{2}$, and $\varphi_{\lambda}^{J}(-\ep_{1}+\ep_{4})=\frac{a_{5}+5}{3}$, and
the same argument as the proof of Lemma \ref{22lem},
we deduce that $a_{1}$ and $a_{5}$ are grater than or equal to $10$.
In this case, we have that exactly one of $a_{2}=3$ or $a_{6}=3$ holds by taking $\Psi=\{\alpha_{3},-\ep_{1}+\ep_{3},\alpha_{4},\ep_{2}+\ep_{3},\alpha_{6}\}$ and the same argument as the proof of Lemma \ref{22lem}.
If $a_{2}=3$ (resp. $a_{6}=3$), we can check that $\{-\ep_{1}+\ep_{4},\ep_{2}+\ep_{4}\}$ (resp. $\{\ep_{2}+\ep_{5},\ep_{3}+\ep_{4}\}$) is a bad pair. 

Assume that $a_{4}=0$.
We examine the condition that $2\in \im(\varphi_\lambda^J)$.
By taking $\Psi=\{\alpha_{4},-\ep_{1}+\ep_{3},-\ep_{2}+\ep_{4},\alpha_{1},\alpha_{6},\ep_{1}+\ep_{3}\}$ and the same argument as proof of Lemma \ref{2lem}, we have that exactly one of $a_{3}=2$ or $a_{5}=2$ or $a_{1}=1$ or $a_{6}=1$ or $a_{2}=0$ holds.
If $a_{1}=1$ (resp. $a_{6}=1$, $a_{2}=0$), we see that $\{-\ep_{1}+\ep_{3},\beta_{2}\}$ (resp. $\{-\ep_{2}+\ep_{4},-\ep_{3}+\ep_{5}\}$, $\{\ep_{2}+\ep_{3},-\ep_{2}+\ep_{3}\}$) is a bad pair.
Suppose that $a_{3}=2$.
We investigate the condition that $4\in \im(\varphi_\lambda^J)$.
By $\varphi_{\lambda}^{J}(\beta_{2})=\frac{4+a_{1}}{2}$, and $\varphi_{\lambda}^{J}(\beta_{3})=\frac{a_{1}+5}{3}$, and
$\varphi_{\lambda}^{J}(-\ep_{2}+\ep_{4})=\frac{2+a_{5}}{2}$, and $\varphi_{\lambda}^{J}(-\ep_{1}+\ep_{4})=\frac{a_{5}+5}{3}$, and
the same argument as the proof of Lemma \ref{22lem},
we deduce that $a_{1}$ and $a_{5}$ are grater than or equal to $10$.
In this case, we have that exactly one of $a_{2}=3$ or $a_{6}=3$ holds by taking $\Psi=\{\alpha_{4},-\ep_{1}+\ep_{3},\alpha_{3},\ep_{2}+\ep_{3},\alpha_{6}\}$ and the same argument as the proof of Lemma \ref{22lem}.
If $a_{2}=3$ (resp. $a_{6}=3$), we can check that $\{\ep_{2}+\ep_{4},-\ep_{1}+\ep_{4}\}$ (resp. $\{\ep_{1}+\ep_{5},\ep_{2}+\ep_{4}\}$) is a bad pair.
Assume that $a_{5}=2$.
We investigate the condition that $4\in \im(\varphi_\lambda^J)$.
By $\varphi_{\lambda}^{J}(-\ep_{1}+\ep_{3})=\frac{2+a_{3}}{2}$, and $\varphi_{\lambda}^{J}(-\ep_{1}+\ep_{4})=\frac{a_{3}+5}{3}$, and
$\varphi_{\lambda}^{J}(-\ep_{3}+\ep_{5})=\frac{4+a_{6}}{2}$, and $\varphi_{\lambda}^{J}(-\ep_{2}+\ep_{5})=\frac{a_{6}+5}{3}$, and
the same argument as the proof of Lemma \ref{22lem},
we deduce that $a_{3}$ and $a_{6}$ are grater than or equal to $10$.
In this case, we have that exactly one of $a_{1}=3$ or $a_{2}=3$ holds by taking $\Psi=\{\alpha_{4},-\ep_{2}+\ep_{4},\alpha_{5},\alpha_{1},\ep_{1}+\ep_{4}\}$ and the same argument as the proof of Lemma \ref{22lem}.
If $a_{1}=3$ (resp. $a_{2}=3$), we can check that $\{-\ep_{1}+\ep_{3},\beta_{2}\}$ (resp. $\{-\ep_{1}+\ep_{4},\ep_{2}+\ep_{4}\}$) is a bad pair. 

Suppose that $a_{5}=0$.
We turn to the condition that $2\in \im(\varphi_\lambda^J)$.
By taking $\Psi=\{\alpha_{5},-\ep_{2}+\ep_{4},-\ep_{3}+\ep_{5},\alpha_{1},\alpha_{3}\}$ and the same argument as proof of Lemma \ref{2lem}, we have that exactly one of $a_{4}=2$ or $a_{6}=2$ or $a_{1}=1$ or $a_{3}=1$ holds.
If $a_{1}=1$ (resp. $a_{3}=1$), we see that $\{-\ep_{1}+\ep_{4},\beta_{3}\}$ (resp. $\{\ep_{1}+\ep_{4},\ep_{2}+\ep_{3}\}$) is a bad pair.
Assume that $a_{4}=2$.
We investigate the condition that $4\in \im(\varphi_\lambda^J)$.
By $\varphi_{\lambda}^{J}(-\ep_{1}+\ep_{3})=\frac{4+a_{3}}{2}$, and $\varphi_{\lambda}^{J}(-\ep_{1}+\ep_{4})=\frac{a_{3}+5}{3}$, and
$\varphi_{\lambda}^{J}(-\ep_{3}+\ep_{5})=\frac{2+a_{6}}{2}$, and $\varphi_{\lambda}^{J}(-\ep_{2}+\ep_{5})=\frac{a_{6}+5}{3}$, and
the same argument as the proof of Lemma \ref{22lem},
we deduce that $a_{3}$ and $a_{6}$ are grater than or equal to $10$.
In this case, we have that exactly one of $a_{1}=3$ or $a_{2}=3$ holds by taking $\Psi=\{\alpha_{5},-\ep_{2}+\ep_{4},\alpha_{4},\alpha_{1},\ep_{1}+\ep_{4}\}$ and the same argument as the proof of Lemma \ref{22lem}.
If $a_{1}=3$ (resp. $a_{2}=3$), we can check that $\{-\ep_{1}+\ep_{3},\beta_{2}\}$ (resp. $\{\ep_{2}+\ep_{4},-\ep_{1}+\ep_{4}\}$) is a bad pair.
Suppose that $a_{6}=2$.
We investigate the condition that $4\in \im(\varphi_\lambda^J)$.
By $\varphi_{\lambda}^{J}(-\ep_{2}+\ep_{4})=\frac{2+a_{4}}{2}$, and $\varphi_{\lambda}^{J}(-\ep_{2}+\ep_{5})=\frac{a_{4}+5}{3}$, and
the same argument as the proof of Lemma \ref{22lem},
we deduce that $a_{4}$ is grater than or equal to $10$.
In this case, we have that exactly one of $a_{1}=3$ or $a_{3}=3$ holds by taking $\Psi=\{\alpha_{5},-\ep_{3}+\ep_{5},\alpha_{6},\alpha_{1},\alpha_{3}\}$ and the same argument as the proof of Lemma \ref{22lem}.
If $a_{1}=3$ (resp. $a_{3}=3$), we can check that $\{-\ep_{1}+\ep_{5},\beta_{4}\}$ (resp. $\{\ep_{1}+\ep_{4},\ep_{2}+\ep_{3}\}$) is a bad pair. 

Suppose that $a_{6}=0$.
We turn to the condition that $2\in \im(\varphi_\lambda^J)$.
By taking $\Psi=\{\alpha_{6},-\ep_{3}+\ep_{5},\alpha_{1},\alpha_{3},\alpha_{4}\}$ and the same argument as proof of Lemma \ref{2lem}, we have that exactly one of $a_{5}=2$ or $a_{1}=1$ or $a_{3}=1$ or $a_{4}=1$ holds.
If $a_{1}=1$ (resp. $a_{3}=1$, $a_{4}=1$), we see that $\{-\ep_{1}+\ep_{5},\beta_{4}\}$ (resp. $\{\ep_{1}+\ep_{5},\ep_{2}+\ep_{4}\}$, $\{-\ep_{2}+\ep_{4},-\ep_{3}+\ep_{5}\}$) is a bad pair.
Assume that $a_{5}=2$.
We investigate the condition that $4\in \im(\varphi_\lambda^J)$.
By $\varphi_{\lambda}^{J}(-\ep_{2}+\ep_{4})=\frac{4+a_{4}}{2}$, and $\varphi_{\lambda}^{J}(-\ep_{2}+\ep_{5})=\frac{a_{4}+5}{3}$, and
the same argument as the proof of Lemma \ref{22lem},
we deduce that $a_{4}$ is grater than or equal to $10$.
In this case, we have that exactly one of $a_{1}=3$ or $a_{3}=3$ holds by taking $\Psi=\{\alpha_{6},-\ep_{3}+\ep_{5},\alpha_{5},\alpha_{1},\alpha_{3}\}$ and the same argument as the proof of Lemma \ref{22lem}.
If $a_{1}=3$ (resp. $a_{3}=3$), we can check that $\{-\ep_{1}+\ep_{4},\beta_{3}\}$ (resp. $\{\ep_{1}+\ep_{4},\ep_{2}+\ep_{3}\}$) is a bad pair. 
\end{proof}
\subsection{Proof of theorem \ref{thmE6} when $|J|=6$}
Finally, we consider the case of $J=I$. Then we have $P_{J}=B$, where $B$ is a Borel subgroup.
\begin{prop}There are no initialized irreducible homogeneous Ulrich bundles on $E_{6}/B$.
\end{prop}
\begin{proof}Suppose for contradiction that there is an initialized irreducible homogeneous Ulrich bundle $E_{\lambda}$ on $E_{6}/B$ with $\lambda=\sum_{i=1}^{6}a_{i}\varpi_{i}$.
By the same argument as the proof of Lemma \ref{n1lem}, we deduce that exactly one of $a_{1}=0$ or $a_{2}=0$ or $a_{3}=0$ or $a_{4}=0$ or $a_{5}=0$ or $a_{6}=0$ holds.

Let us consider the case $a_{1}=0$ first.
We examine the condition that $2\in \im(\varphi_\lambda^J)$.
By taking $\Psi=\{\alpha_{1},\beta_{2},\alpha_{2},\alpha_{4},\alpha_{5},\alpha_{6}\}$ and the same argument as proof of Lemma \ref{2lem}, we have that exactly one of $a_{3}=2$ or $a_{2}=1$ or $a_{4}=1$ or $a_{5}=1$ or $a_{6}=1$ holds.
If $a_{2}=1$ (resp. $a_{4}=1$, $a_{5}=1$, $a_{6}=1$), we see that $\{\ep_{2}+\ep_{3},\beta_{3}\}$ (resp. $\{-\ep_{1}+\ep_{3},\beta_{2}\}$, $\{-\ep_{1}+\ep_{4},\beta_{3}\}$, $\{-\ep_{1}+\ep_{5},\beta_{4}\}$) is a bad pair.
Suppose that $a_{3}=2$.
We investigate the condition that $4\in \im(\varphi_\lambda^J)$.
By $\varphi_{\lambda}^{J}(-\ep_{1}+\ep_{3})=\frac{4+a_{4}}{2}$, and $\varphi_{\lambda}^{J}(\beta_{3})=\frac{a_{4}+5}{3}$, and
the same argument as the proof of Lemma \ref{22lem},
we deduce that $a_{4}$ is grater than or equal to $10$.
In this case, we have that exactly one of $a_{2}=3$ or $a_{5}=3$ or $a_{6}=3$ holds by taking $\Psi=\{\alpha_{1},\beta_{2},\alpha_{3},\alpha_{2},\alpha_{5},\alpha_{6}\}$ and the same argument as the proof of Lemma \ref{22lem}.
If $a_{2}=3$ (resp. $a_{5}=3$, $a_{6}=3$), we can check that $\{-\ep_{1}+\ep_{3},\ep_{1}+\ep_{3}\}$ (resp. $\{\ep_{1}+\ep_{4},\ep_{2}+\ep_{3}\}$, $\{-\ep_{1}+\ep_{5},-\ep_{2}+\ep_{5}\}$) is a bad pair.

Next, let us consider the case $a_{2}=0$.
We examine the condition that $2\in \im(\varphi_\lambda^J)$.
By taking $\Psi=\{\alpha_{2},\ep_{1}+\ep_{3},\alpha_{1},\alpha_{3},\alpha_{5},\alpha_{6}\}$ and the same argument as proof of Lemma \ref{2lem}, we have that exactly one of $a_{4}=2$ or $a_{1}=1$ or $a_{3}=1$ or $a_{5}=1$ or $a_{6}=1$ holds.
If $a_{1}=1$ (resp. $a_{3}=1$, $a_{5}=1$, $a_{6}=1$), we see that $\{\ep_{2}+\ep_{4},\beta_{4}\}$ (resp. $\{\ep_{1}+\ep_{3},-\ep_{1}+\ep_{3}\}$, $\{-\ep_{2}+\ep_{4},\ep_{1}+\ep_{3}\}$, $\{-\ep_{2}+\ep_{5},\ep_{1}+\ep_{4}\}$) is a bad pair.
Suppose that $a_{4}=2$.
We investigate the condition that $4\in \im(\varphi_\lambda^J)$.
By $\varphi_{\lambda}^{J}(-\ep_{1}+\ep_{3})=\frac{4+a_{3}}{2}$, and $\varphi_{\lambda}^{J}(\ep_{2}+\ep_{3})=\frac{a_{3}+5}{3}$, and
$\varphi_{\lambda}^{J}(-\ep_{2}+\ep_{4})=\frac{4+a_{5}}{2}$, and $\varphi_{\lambda}^{J}(\ep_{1}+\ep_{4})=\frac{a_{5}+5}{3}$, and
the same argument as the proof of Lemma \ref{22lem},
we deduce that $a_{3}\geq10$ and $a_{5}\geq10$.
In this case, we have that exactly one of $a_{1}=3$ or $a_{6}=3$ holds by taking $\Psi=\{\alpha_{2},\ep_{1}+\ep_{3},\alpha_{4},\alpha_{6},\alpha_{1},\}$ and the same argument as the proof of Lemma \ref{22lem}.
If $a_{1}=3$ (resp. $a_{6}=3$), we can check that $\{\ep_{2}+\ep_{3},\beta_{3}\}$ (resp. $\{\ep_{1}+\ep_{4},-\ep_{2}+\ep_{5}\}$) is a bad pair.

Next, let us consider the case $a_{3}=0$.
We examine the condition that $2\in \im(\varphi_\lambda^J)$.
By taking $\Psi=\{\alpha_{3},\beta_{2},-\ep_{1}+\ep_{3},\alpha_{2},\alpha_{5},\alpha_{6}\}$ and the same argument as proof of Lemma \ref{2lem}, we have that exactly one of $a_{1}=2$ or $a_{4}=2$ or $a_{2}=1$ or $a_{5}=1$ or $a_{6}=1$ holds.
If $a_{2}=1$ (resp. $a_{5}=1$, $a_{6}=1$), we see that $\{\ep_{1}+\ep_{3},-\ep_{1}+\ep_{3}\}$ (resp. $\{\ep_{1}+\ep_{4},\ep_{2}+\ep_{3}\}$, $\{-\ep_{1}+\ep_{4},-\ep_{2}+\ep_{5}\}$) is a bad pair.
Assume that $a_{1}=2$.
We investigate the condition that $4\in \im(\varphi_\lambda^J)$.
By $\varphi_{\lambda}^{J}(-\ep_{1}+\ep_{3})=\frac{2+a_{4}}{2}$, and $\varphi_{\lambda}^{J}(\beta_{3})=\frac{a_{4}+5}{3}$, and
the same argument as the proof of Lemma \ref{22lem},
we deduce that $a_{4}$ is grater than or equal to $10$.
In this case, we have that exactly one of $a_{2}=3$ or $a_{5}=3$ or $a_{6}=3$ holds by taking $\Psi=\{\alpha_{3},\beta_{2},\alpha_{1},\alpha_{2},\alpha_{5},\alpha_{6}\}$ and the same argument as the proof of Lemma \ref{22lem}.
If $a_{2}=3$ (resp. $a_{5}=3$, $a_{6}=3$), we can check that $\{-\ep_{1}+\ep_{3},\ep_{1}+\ep_{3}\}$ (resp. $\{-\ep_{1}+\ep_{3},-\ep_{2}+\ep_{4}\}$, $\{-\ep_{1}+\ep_{5},\beta_{5}\}$) is a bad pair.
Suppose that $a_{4}=2$.
We investigate the condition that $4\in \im(\varphi_\lambda^J)$.
By $\varphi_{\lambda}^{J}(\beta_{2})=\frac{2+a_{1}}{2}$, and $\varphi_{\lambda}^{J}(\beta_{3})=\frac{a_{1}+5}{3}$, and
$\varphi_{\lambda}^{J}(\ep_{1}+\ep_{3})=\frac{4+a_{2}}{2}$, and $\varphi_{\lambda}^{J}(\ep_{2}+\ep_{3})=\frac{a_{2}+5}{3}$, and
$\varphi_{\lambda}^{J}(-\ep_{2}+\ep_{4})=\frac{4+a_{5}}{2}$, and $\varphi_{\lambda}^{J}(-\ep_{1}+\ep_{4})=\frac{a_{5}+5}{3}$, and
the same argument as the proof of Lemma \ref{22lem},
we deduce that $a_{1}\geq10$, $a_{2}\geq10$ and $a_{5}\geq10$.
In this case, we have $a_{6}=3$ by taking $\Psi=\{\alpha_{3},-\ep_{1}+\ep_{3},\alpha_{4},\alpha_{6}\}$ and the same argument as the proof of Lemma \ref{22lem}.
If $a_{6}=3$, we can check that $\{\ep_{1}+\ep_{5},\ep_{2}+\ep_{4}\}$ is a bad pair.

Next, let us consider the case $a_{4}=0$.
We examine the condition that $2\in \im(\varphi_\lambda^J)$.
By taking $\Psi=\{\alpha_{4},\ep_{1}+\ep_{3},-\ep_{1}+\ep_{3},-\ep_{2}+\ep_{4},\alpha_{1},\alpha_{6}\}$ and the same argument as proof of Lemma \ref{2lem}, we have that exactly one of $a_{2}=2$ or $a_{3}=2$ or $a_{5}=2$ or $a_{1}=1$ or $a_{6}=1$ holds.
If $a_{1}=1$ (resp. $a_{6}=1$), we see that $\{-\ep_{1}+\ep_{3},\beta_{2}\}$ (resp. $\{-\ep_{2}+\ep_{4},-\ep_{3}+\ep_{5}\}$) is a bad pair.
Suppose that $a_{2}=2$.
We investigate the condition that $4\in \im(\varphi_\lambda^J)$.
By $\varphi_{\lambda}^{J}(-\ep_{1}+\ep_{3})=\frac{2+a_{3}}{2}$, and $\varphi_{\lambda}^{J}(\ep_{2}+\ep_{3})=\frac{a_{3}+5}{3}$, and
$\varphi_{\lambda}^{J}(-\ep_{2}+\ep_{4})=\frac{2+a_{5}}{2}$, and $\varphi_{\lambda}^{J}(\ep_{1}+\ep_{4})=\frac{a_{5}+5}{3}$, and
the same argument as the proof of Lemma \ref{22lem},
we deduce that $a_{3}\geq10$ and $a_{5}\geq10$.
In this case, we have that exactly one of $a_{1}=3$ or $a_{6}=3$ holds by taking $\Psi=\{\alpha_{4},\ep_{1}+\ep_{3},\alpha_{2},\alpha_{1},\alpha_{6}\}$ and the same argument as the proof of Lemma \ref{22lem}.
If $a_{1}=3$ (resp. $a_{6}=3$), we can check that $\{-\ep_{1}+\ep_{3},\beta_{2}\}$ (resp. $\{-\ep_{2}+\ep_{4},-\ep_{3}+\ep_{5}\}$) is a bad pair.
Assume that $a_{3}=2$.
We investigate the condition that $4\in \im(\varphi_\lambda^J)$.
By $\varphi_{\lambda}^{J}(\beta_{2})=\frac{4+a_{1}}{2}$, and $\varphi_{\lambda}^{J}(\beta_{3})=\frac{a_{1}+5}{3}$, and
$\varphi_{\lambda}^{J}(\ep_{1}+\ep_{3})=\frac{2+a_{2}}{2}$, and $\varphi_{\lambda}^{J}(\ep_{2}+\ep_{3})=\frac{a_{2}+5}{3}$, and
$\varphi_{\lambda}^{J}(-\ep_{2}+\ep_{4})=\frac{2+a_{5}}{2}$, and $\varphi_{\lambda}^{J}(-\ep_{1}+\ep_{4})=\frac{a_{5}+5}{3}$, and
the same argument as the proof of Lemma \ref{22lem},
we deduce that $a_{1}\geq10$, $a_{2}\geq10$ and $a_{5}\geq10$.
In this case, we have $a_{6}=3$ by taking $\Psi=\{\alpha_{4},-\ep_{1}+\ep_{3},\alpha_{3},\alpha_{6}\}$ and the same argument as the proof of Lemma \ref{22lem}.
If $a_{6}=3$, we can check that $\{-\ep_{1}+\ep_{4},-\ep_{2}+\ep_{5}\}$ is a bad pair.
Assume that $a_{5}=2$.
We investigate the condition that $4\in \im(\varphi_\lambda^J)$.
By $\varphi_{\lambda}^{J}(\ep_{1}+\ep_{3})=\frac{2+a_{2}}{2}$, and $\varphi_{\lambda}^{J}(\ep_{1}+\ep_{4})=\frac{a_{2}+5}{3}$, and
$\varphi_{\lambda}^{J}(-\ep_{1}+\ep_{3})=\frac{2+a_{3}}{2}$, and $\varphi_{\lambda}^{J}(-\ep_{1}+\ep_{4})=\frac{a_{3}+5}{3}$, and
$\varphi_{\lambda}^{J}(-\ep_{3}+\ep_{5})=\frac{4+a_{6}}{2}$, and $\varphi_{\lambda}^{J}(-\ep_{2}+\ep_{5})=\frac{a_{6}+5}{3}$, and
the same argument as the proof of Lemma \ref{22lem},
we deduce that $a_{2}\geq10$, $a_{3}\geq10$ and $a_{6}\geq10$.
In this case, we have $a_{1}=3$ by taking $\Psi=\{\alpha_{4},-\ep_{2}+\ep_{4},\alpha_{5},\alpha_{1}\}$ and the same argument as the proof of Lemma \ref{22lem}.
If $a_{1}=3$, we can check that $\{-\ep_{1}+\ep_{3},\beta_{2}\}$ is a bad pair.

Next, let us consider the case $a_{5}=0$.
We examine the condition that $2\in \im(\varphi_\lambda^J)$.
By taking $\Psi=\{\alpha_{5},-\ep_{2}+\ep_{4},-\ep_{3}+\ep_{5},\alpha_{1},\alpha_{2},\alpha_{3}\}$ and the same argument as proof of Lemma \ref{2lem}, we have that exactly one of $a_{4}=2$ or $a_{6}=2$ or $a_{1}=1$ or $a_{2}=1$ or $a_{3}=1$ holds.
If $a_{1}=1$ (resp. $a_{2}=1$, $a_{3}=1$), we see that $\{-\ep_{1}+\ep_{4},\beta_{3}\}$ (resp. $\{\ep_{1}+\ep_{3},-\ep_{2}+\ep_{4}\}$, $\{-\ep_{1}+\ep_{3},-\ep_{2}+\ep_{4}\}$) is a bad pair.
Suppose that $a_{4}=2$.
We investigate the condition that $4\in \im(\varphi_\lambda^J)$.
By $\varphi_{\lambda}^{J}(\ep_{1}+\ep_{3})=\frac{4+a_{2}}{2}$, and $\varphi_{\lambda}^{J}(\ep_{1}+\ep_{4})=\frac{a_{2}+5}{3}$, and
$\varphi_{\lambda}^{J}(-\ep_{1}+\ep_{3})=\frac{4+a_{3}}{2}$, and $\varphi_{\lambda}^{J}(-\ep_{1}+\ep_{4})=\frac{a_{3}+5}{3}$, and
$\varphi_{\lambda}^{J}(-\ep_{3}+\ep_{5})=\frac{2+a_{6}}{2}$, and $\varphi_{\lambda}^{J}(-\ep_{2}+\ep_{5})=\frac{a_{6}+5}{3}$, and
the same argument as the proof of Lemma \ref{22lem},
we deduce that $a_{2}\geq10$, $a_{3}\geq10$ and $a_{6}\geq10$.
In this case, we have $a_{1}=3$ by taking $\Psi=\{\alpha_{5},-\ep_{2}+\ep_{4},\alpha_{4},\alpha_{1}\}$ and the same argument as the proof of Lemma \ref{22lem}.
If $a_{1}=3$, we can check that $\{-\ep_{1}+\ep_{3},\beta_{2}\}$ is a bad pair.
Assume that $a_{6}=2$.
We investigate the condition that $4\in \im(\varphi_\lambda^J)$.
By $\varphi_{\lambda}^{J}(-\ep_{2}+\ep_{4})=\frac{2+a_{4}}{2}$, and $\varphi_{\lambda}^{J}(-\ep_{2}+\ep_{5})=\frac{a_{4}+5}{3}$, and
the same argument as the proof of Lemma \ref{22lem},
we deduce that $a_{4}\geq10$.
In this case, we obtain that exactly one of $a_{1}=3$ or $a_{2}=3$ or $a_{3}=3$ holds by taking $\Psi=\{\alpha_{5},-\ep_{3}+\ep_{5},\alpha_{6},\alpha_{1},\alpha_{2},\alpha_{3}\}$ and the same argument as the proof of Lemma \ref{22lem}.
If $a_{1}=3$ (resp. $a_{2}=3$, $a_{3}=3$), we can check that $\{-\ep_{1}+\ep_{5},\beta_{4}\}$ (resp. $\{\ep_{1}+\ep_{3},-\ep_{2}+\ep_{4}\}$, $\{-\ep_{1}+\ep_{3},-\ep_{2}+\ep_{4}\}$) is a bad pair. 

Finally, let us consider the case $a_{6}=0$.
We examine the condition that $2\in \im(\varphi_\lambda^J)$.
By taking $\Psi=\{\alpha_{6},-\ep_{3}+\ep_{5},\alpha_{1},\alpha_{2},\alpha_{3},\alpha_{4}\}$ and the same argument as proof of Lemma \ref{2lem}, we have that exactly one of $a_{5}=2$ or $a_{1}=1$ or $a_{2}=1$ or $a_{3}=1$ or $a_{4}=1$ holds.
If $a_{1}=1$ (resp. $a_{2}=1$, $a_{3}=1$, $a_{4}=1$), we see that $\{-\ep_{1}+\ep_{5},\beta_{4}\}$ (resp. $\{\ep_{1}+\ep_{4},-\ep_{2}+\ep_{5}\}$, $\{-\ep_{1}+\ep_{4},-\ep_{2}+\ep_{5}\}$, $\{-\ep_{2}+\ep_{4},-\ep_{3}+\ep_{5}\}$) is a bad pair.
Suppose that $a_{5}=2$.
We investigate the condition that $4\in \im(\varphi_\lambda^J)$.
By $\varphi_{\lambda}^{J}(-\ep_{2}+\ep_{4})=\frac{4+a_{4}}{2}$, and $\varphi_{\lambda}^{J}(-\ep_{2}+\ep_{5})=\frac{a_{4}+5}{3}$, and
the same argument as the proof of Lemma \ref{22lem},
we deduce that $a_{4}\geq10$.
In this case, we obtain that exactly one of $a_{1}=3$ or $a_{2}=3$ or $a_{3}=3$ holds by taking $\Psi=\{\alpha_{6},-\ep_{3}+\ep_{5},\alpha_{5},\alpha_{1},\alpha_{2},\alpha_{3}\}$ and the same argument as the proof of Lemma \ref{22lem}.
If $a_{1}=3$ (resp. $a_{2}=3$, $a_{3}=3$), we can check that $\{-\ep_{1}+\ep_{4},\beta_{3}\}$ (resp. $\{\ep_{1}+\ep_{3},-\ep_{2}+\ep_{4}\}$, $\{-\ep_{1}+\ep_{3},-\ep_{2}+\ep_{4}\}$) is a bad pair.
\end{proof}
\section{Type $F_{4}$}
In this section, we prove Theorem \ref{thm} of the case in which $G$ is type $F_{4}$. In other words, the goal of this section is to show the following theorem.
\begin{thm}\label{thmF4}For every $J\subset I$ with $|J|\geq2$, there are no initialized irreducible homogeneous Ulrich bundles with respect to $\mathcal{O}(1)$ on $F_{4}/P_{J}$.
\end{thm}

Let $G$ be a simply connected simple Lie group with a Dynkin diagram of type $F_{4}$, as follows.
\begin{center}
$F_{4}$:\dynkin[label]{F}{4}
\end{center}

We denote $I^{\prime}_{4}$ by $I_{4}+\mathbb{Z}((\epsilon_{1}+\epsilon_{2}+\epsilon_{3}+\epsilon_{4})/2)$ 
and define $\Phi_{F_{4}}=\{\alpha\in I^{\prime}_{4}\ |\ (\alpha,\alpha)\in \{1,2\}\}$. Then, $\Phi_{F_{4}}$ consists of all $\pm\epsilon_{i}$, all $\pm(\epsilon\pm\epsilon_{j})$, $i\neq j$, and $\pm\frac{1}{2}(\epsilon_{1}\pm\epsilon_{2}\pm\epsilon_{3}\pm\epsilon_{4})$, where the signs may be chosen independently. As a base, we take
$$\Delta_{4}:=\{\alpha_{1}:=\epsilon_{2}-\epsilon_{3},\ \alpha_{2}:=\epsilon_{3}-\epsilon_{4},\ \alpha_{3}:=\epsilon_{4},\ \alpha_{4}:=\frac{1}{2}(\epsilon_{1}-\epsilon_{2}-\epsilon_{3}-\epsilon_{4})\}.$$
Let $\beta_{i}$ for $1\leq i\leq 4$ denotes the positive root such that the coefficient of $\ep_{i}$ is $-\frac{1}{2}$ and those of the remaining are $\frac{1}{2}$.
Moreover, let $\gamma_{i,j}$ for $2\leq i<j \leq4$ be the positive root such that the coefficients of $\ep_{i}$ and $e_{j}$ are $-\frac{1}{2}$ and those of the remaining are $\frac{1}{2}$ and $\beta_{0}$ be the positive root such that all coefficients of $\ep_{i}$ are $\frac{1}{2}$.
The fundamental weights are as follows:
\begin{eqnarray*}
\varpi_{1}&=&\epsilon_{1}+\epsilon_{2}\\
\varpi_{2}&=&2\epsilon_{1}+\epsilon_{2}+\epsilon_{3}\\
\varpi_{3}&=&\frac{1}{2}(3\epsilon_{1}+\epsilon_{2}+\epsilon_{3}+\epsilon_{4})\\
\varpi_{4}&=&\epsilon_{1}.\\
\end{eqnarray*}

\subsection{Proof of theorem \ref{thmF4} when $|J|=2$}
\begin{prop}There are no initialized irreducible homogeneous Ulrich bundles on $F_{4}/P_{1,2}$.
\end{prop}
\begin{proof}Suppose for contradiction that there is an initialized irreducible homogeneous Ulrich bundle $E_{\lambda}$ on $F_{4}/P_{1,2}$ with $\lambda=\sum_{i=1}^{4}a_{i}\varpi_{i}$.
By the same argument as the proof of Lemma \ref{n1lem}, we deduce that exactly one of $a_{1}=0$ or $a_{2}=0$ holds.

Assume that $a_{1}=0$.
We examine the condition that $2\in \im(\varphi_\lambda^J)$.
By taking $\Psi=\{\alpha_{1},\ep_{2}-\ep_{4}\}$ and the same argument as the proof of Lemma \ref{2lem}, we have $a_{2}=2$.
Suppose that $a_{2}=2$.
By the same argument as the proof of Lemma \ref{maxlem}, the function $\varphi_{\lambda}^{J}$ takes the maximal value $\dim(F_{4}/P_{1,2})=21$ and second maximal value $25$ at $\ep_{1}-\ep_{2}$ and $\beta_{2}$.
In other words, we obtain $\varphi_{\lambda}^{J}(\ep_{1}-\ep_{2})=5+a_{3}+a_{4}=21$ and $\varphi_{\lambda}^{J}(\beta_{2})=\frac{9}{2}+a_{3}+\frac{a_{4}}{2}=20$.
Therefore, we have $a_{4}=1$.
Then we see that $\{\ep_{2}+\ep_{3},\ep_{1}-\ep_{4}\}$ is a bad pair.

Suppose that $a_{2}=0$.
We examine the condition that $2\in \im(\varphi_\lambda^J)$.
In this case, we see that $a_{3}\geq1$.
By taking $\Psi=\{\alpha_{2},\ep_{2}-\ep_{4},\ep_{3}\}$ and the same argument as the proof of Lemma \ref{2lem}, we have that exactly one of $a_{1}=2$ or $a_{3}=1$ holds.
If $a_{1}=2$ (resp. $a_{3}=1$), we can check that $\{\beta_{3},\ep_{1}+\ep_{3}\}$ (resp. $\{\ep_{1}-\ep_{4},\ep_{1}+\ep_{4}\}$) is a bad pair.
\end{proof}
\begin{prop}There are no initialized irreducible homogeneous Ulrich bundles on $F_{4}/P_{1,3}$.
\end{prop}
\begin{proof}Suppose for contradiction that there is an initialized irreducible homogeneous Ulrich bundle $E_{\lambda}$ on $F_{4}/P_{1,3}$ with $\lambda=\sum_{i=1}^{4}a_{i}\varpi_{i}$.
By the same argument as the proof of Lemma \ref{n1lem}, we deduce that exactly one of $a_{1}=0$ or $a_{3}=0$ holds.

Assume that $a_{1}=0$.
We examine the condition that $2\in \im(\varphi_\lambda^J)$.
By taking $\Psi=\{\alpha_{1},\ep_{2}-\ep_{4},\ep_{3}\}$ and the same argument as the proof of Lemma \ref{2lem}, we have that exactly one of $a_{3}=1$ or $a_{2}=0$ holds.
If $a_{3}=1$ (resp. $a_{2}=0$), we see that $\{\ep_{1}+\ep_{4},\gamma_{3,4}\}$ (resp. $\{\ep_{1}+\ep_{3},\ep_{1}+\ep_{4}\}$) is a bad pair.

Suppose that $a_{3}=0$.
We examine the condition that $2\in \im(\varphi_\lambda^J)$.
By taking $\Psi=\{\alpha_{1},\ep_{3}+\ep_{4},\alpha_{1},\gamma_{2,3}\}$ and the same argument as the proof of Lemma \ref{2lem}, we have that exactly one of $a_{1}=1$ or $a_{2}=0$ or $a_{4}=0$ holds.
If $a_{1}=1$ (resp. $a_{2}=0$, $a_{4}=0$), we can check that $\{\beta_{2},\ep_{1}-\ep_{4}\}$ (resp. $\{\ep_{1}+\ep_{3},\ep_{1}+\ep_{4}\}$, $\{\ep_{2}+\ep_{4},\ep_{1}-\ep_{4}\}$) is a bad pair.
\end{proof}
\begin{prop}There are no initialized irreducible homogeneous Ulrich bundles on $F_{4}/P_{1,4}$.
\end{prop}
\begin{proof}Suppose for contradiction that there is an initialized irreducible homogeneous Ulrich bundle $E_{\lambda}$ on $F_{4}/P_{1,4}$ with $\lambda=\sum_{i=1}^{4}a_{i}\varpi_{i}$.
By the same argument as the proof of Lemma \ref{n1lem}, we deduce that exactly one of $a_{1}=0$ or $a_{4}=0$ holds.

Assume that $a_{1}=0$.
We examine the condition that $2\in \im(\varphi_\lambda^J)$.
By taking $\Psi=\{\alpha_{1},\ep_{2}-\ep_{4},\alpha_{4}\}$ and the same argument as the proof of Lemma \ref{2lem}, we have that exactly one of $a_{4}=1$ or $a_{2}=0$ holds.
If $a_{4}=1$ (resp. $a_{2}=0$), we can check that $\{\ep_{2},\ep_{1}-\ep_{4}\}$ ($\{\ep_{1}+\ep_{3},\ep_{1}+\ep_{4}\}$) is a bad pair.

Suppose that $a_{4}=0$.
We turn to the condition that $2\in \im(\varphi_\lambda^J)$.
By taking $\Psi=\{\alpha_{1},\ep_{2}-\ep_{3},\gamma_{2,3}\}$ and the same argument as the proof of Lemma \ref{2lem}, we have that exactly one of $a_{1}=1$ or $a_{3}=0$ holds.
If $a_{1}=1$ (resp. $a_{3}=0$), we see that $\{\ep_{2},\ep_{1}-\ep_{4}\}$ (resp. $\{\beta_{3},\gamma_{3,4}\}$) is a bad pair.
\end{proof}
\begin{prop}There are no initialized irreducible homogeneous Ulrich bundles on $F_{4}/P_{2,3}$.
\end{prop}
\begin{proof}Suppose for contradiction that there is an initialized irreducible homogeneous Ulrich bundle $E_{\lambda}$ on $F_{4}/P_{2,3}$ with $\lambda=\sum_{i=1}^{4}a_{i}\varpi_{i}$.
By the same argument as the proof of Lemma \ref{n1lem}, we deduce that exactly one of $a_{2}=0$ or $a_{3}=0$ holds.

Suppose that $a_{2}=0$.
We examine the condition that $2\in \im(\varphi_\lambda^J)$.
By $\varphi_{\lambda}^{J}(\ep_{3}+\ep_{4})=\frac{2+a_{3}}{2}$, and $\varphi_{\lambda}^{J}(\ep_{3})=\frac{a_{3}+3}{3}$, and
the same argument as the proof of Lemma \ref{22lem},
we deduce that $a_{3}\geq6$.
In this case, we have $a_{1}=0$ by taking $\Psi=\{\alpha_{2},\ep_{2}-\ep_{4}\}$ and the same argument as the proof of Lemma \ref{2lem}.
If $a_{1}=0$, we see that $\{\ep_{2},\ep_{3}\}$ is a bad pair.

Assume that $a_{3}=0$.
We examine the condition that $2\in \im(\varphi_\lambda^J)$.
By $\varphi_{\lambda}^{J}(\ep_{3}+\ep_{4})=\frac{2+a_{2}}{2}$, and $\varphi_{\lambda}^{J}(\ep_{3})=\frac{2a_{2}+3}{3}$, and
the same argument as the proof of Lemma \ref{22lem},
we deduce that $a_{2}\geq6$.
In this case, we have $a_{4}=0$ by taking $\Psi=\{\alpha_{3},\gamma_{2,3}\}$ and the same argument as the proof of Lemma \ref{2lem}.
If $a_{4}=0$, we see that $\{\ep_{2}+\ep_{3},\ep_{1}-\ep_{4}\}$ is a bad pair.
\end{proof}
\begin{prop}There are no initialized irreducible homogeneous Ulrich bundles on $F_{4}/P_{2,4}$.
\end{prop}
\begin{proof}Suppose for contradiction that there is an initialized irreducible homogeneous Ulrich bundle $E_{\lambda}$ on $F_{4}/P_{2,4}$ with $\lambda=\sum_{i=1}^{4}a_{i}\varpi_{i}$.
By the same argument as the proof of Lemma \ref{n1lem}, we deduce that exactly one of $a_{2}=0$ or $a_{4}=0$ holds.

Assume that $a_{2}=0$.
We examine the condition that $2\in \im(\varphi_\lambda^J)$.
By taking $\Psi=\{\alpha_{2},\ep_{2}-\ep_{4},\alpha_{4},\ep_{3}\}$ and the same argument as the proof of Lemma \ref{2lem}, we have that exactly one of $a_{3}=1$ or $a_{4}=1$ or $a_{1}=0$ holds.
If $a_{3}=1$ (resp. $a_{4}=1$, $a_{1}=0$), we see that $\{\beta_{2},\gamma_{2,4}\}$ (resp. $\{\ep_{2}+\ep_{3},\ep_{1}-\ep_{3}\}$, $\{\ep_{1}+\ep_{2},\ep_{1}+\ep_{3}\}$) is a bad pair.

Suppose that $a_{4}=0$.
We examine the condition that $2\in \im(\varphi_\lambda^J)$.
By taking $\Psi=\{\alpha_{4},\ep_{3}-\ep_{4},\gamma_{2,3}\}$ and the same argument as the proof of Lemma \ref{2lem}, we have that exactly one of $a_{3}=0$ or $a_{2}=1$ holds.
If $a_{3}=0$ (resp. $a_{2}=1$), the positive roots $\{\beta_{0},\beta_{4}\}$ (resp. $\{\ep_{2}+\ep_{3},\ep_{1}-\ep_{3}\}$) is a bad pair.
\end{proof}
\begin{prop}There are no initialized irreducible homogeneous Ulrich bundles on $F_{4}/P_{3,4}$.
\end{prop}
\begin{proof}Suppose for contradiction that there is an initialized irreducible homogeneous Ulrich bundle $E_{\lambda}$ on $F_{4}/P_{3,4}$ with $\lambda=\sum_{i=1}^{4}a_{i}\varpi_{i}$.
By the same argument as the proof of Lemma \ref{n1lem}, we deduce that exactly one of $a_{3}=0$ or $a_{4}=0$ holds.

Assume that $a_{3}=0$.
We examine the condition that $2\in \im(\varphi_\lambda^J)$.
By taking $\Psi=\{\alpha_{3},\ep_{3},\gamma_{2,3}\}$ and the same argument as the proof of Lemma \ref{2lem}, we have that exactly one of $a_{4}=2$ or $a_{2}=1$ holds.
If $a_{2}=0$, we see that $\{\ep_{1}+\ep_{3},\ep_{1}+\ep_{4}\}$ is a bad pair.
Suppose that $a_{4}=2$.
By the same argument as the proof of Lemma \ref{maxlem}, the function $\varphi_{\lambda}^{J}$ takes the maximal value $\dim(F_{4}/P_{1,2})=21$ and second maximal value $25$ at $\ep_{2}$ and $\ep_{2}+\ep_{3}$.
In other words, we obtain $\varphi_{\lambda}^{J}(\ep_{2})=5+2a_{1}+2a_{2}=21$ and $\varphi_{\lambda}^{J}(\ep_{2}+\ep_{3})=4+a_{1}+2a_{2}=20$.
Therefore, we have $a_{1}=0$.
If $a_{1}=0$, we can check that $\{\ep_{1}+\ep_{2},\ep_{1}+\ep_{3}\}$ is a bad pair.

Suppose that $a_{4}=0$.
We examine the condition that $2\in \im(\varphi_\lambda^J)$.
By taking $\Psi=\{\alpha_{4},\gamma_{2,3}\}$ and the same argument as the proof of Lemma \ref{2lem}, we have $a_{3}=2$.
Assume that $a_{3}=2$.
By the same argument as the proof of Lemma \ref{maxlem}, we obtain $\varphi_{\lambda}^{J}(\ep_{2})=7+2a_{1}+2a_{2}=21$ and $\varphi_{\lambda}^{J}(\ep_{2}+\ep_{3})=6+a_{1}+2a_{2}=20$.
Therefore, we have $a_{1}=0$.
If $a_{1}=0$, we can check that $\{\ep_{1}+\ep_{3},\ep_{1}+\ep_{2}\}$ is a bad pair.
\end{proof}
\subsection{Proof of theorem \ref{thmF4} when $|J|=3$}
\begin{prop}There are no initialized irreducible homogeneous Ulrich bundles on $F_{4}/P_{1,2,3}$.
\end{prop}
\begin{proof}Suppose for contradiction that there is an initialized irreducible homogeneous Ulrich bundle $E_{\lambda}$ on $F_{4}/P_{1,2,3}$ with $\lambda=\sum_{i=1}^{4}a_{i}\varpi_{i}$.
By the same argument as the proof of Lemma \ref{n1lem}, we deduce that exactly one of $a_{1}=0$ or $a_{2}=0$ or $a_{3}=0$ holds.

Assume that $a_{1}=0$.
We examine the condition that $2\in \im(\varphi_\lambda^J)$.
By taking $\Psi=\{\alpha_{1},\ep_{2}-\ep_{4},\alpha_{3}\}$ and the same argument as the proof of Lemma \ref{2lem}, we have that exactly one of $a_{2}=2$ or $a_{3}=1$ holds.
If $a_{2}=2$ (resp. $a_{3}=1$), we see that $\{\ep_{3},\ep_{2}+\ep_{4}\}$ (resp. $\{\ep_{3}+\ep_{4},\ep_{2}-\ep_{4}\}$) is a bad pair.

Suppose that $a_{2}=0$.
We turn to the condition that $2\in \im(\varphi_\lambda^J)$.
By $\varphi_{\lambda}^{J}(\ep_{3}+\ep_{4})=\frac{2+a_{3}}{2}$, and $\varphi_{\lambda}^{J}(\ep_{3})=\frac{a_{3}+3}{3}$, and
the same argument as the proof of Lemma \ref{22lem},
we deduce that $a_{3}\geq6$.
By taking $\Psi=\{\alpha_{2},\ep_{2}-\ep_{4}\}$ and the same argument as the proof of Lemma \ref{2lem}, we have $a_{1}=2$.
When $a_{1}=2$, we see that $\{\ep_{3},\ep_{2}+\ep_{4}\}$ is a bad pair.

Assume that $a_{3}=0$.
We examine the condition that $2\in \im(\varphi_\lambda^J)$.
By $\varphi_{\lambda}^{J}(\ep_{3}+\ep_{4})=\frac{2+a_{2}}{2}$, and $\varphi_{\lambda}^{J}(\ep_{3})=\frac{2a_{2}+3}{3}$, and
the same argument as the proof of Lemma \ref{22lem},
we deduce that $a_{2}\geq6$.
By taking $\Psi=\{\alpha_{3},\alpha_{1},\gamma_{2,3}\}$ and the same argument as the proof of Lemma \ref{2lem}, we have that exactly one of $a_{1}=1$ or $a_{4}=0$ holds.
If $a_{1}=1$ (resp. $a_{4}=0$), we see that $\{\ep_{3}+\ep_{4},\ep_{2}-\ep_{4}\}$ (resp. $\{\ep_{2}+\ep_{4},\ep_{1}-\ep_{3}\}$) is a bad pair.
\end{proof}
\begin{prop}There are no initialized irreducible homogeneous Ulrich bundles on $F_{4}/P_{1,2,4}$.
\end{prop}
\begin{proof}Suppose for contradiction that there is an initialized irreducible homogeneous Ulrich bundle $E_{\lambda}$ on $F_{4}/P_{1,2,4}$ with $\lambda=\sum_{i=1}^{4}a_{i}\varpi_{i}$.
By the same argument as the proof of Lemma \ref{n1lem}, we deduce that exactly one of $a_{1}=0$ or $a_{2}=0$ or $a_{4}=0$ holds.

Assume that $a_{1}=0$.
We examine the condition that $2\in \im(\varphi_\lambda^J)$.
By taking $\Psi=\{\alpha_{1},\ep_{2}-\ep_{4},\gamma_{2,3}\}$ and the same argument as the proof of Lemma \ref{2lem}, we have that exactly one of $a_{2}=2$ or $a_{4}=1$ holds.
If $a_{2}=2$ (resp. $a_{4}=1$), we see that $\{\ep_{1}-\ep_{3},\gamma_{2,4}\}$ (resp. $\{\ep_{2}+\ep_{3},\gamma_{2,4}\}$) is a bad pair.

Suppose that $a_{2}=0$.
We examine the condition that $2\in \im(\varphi_\lambda^J)$.
By taking $\Psi=\{\alpha_{2},\ep_{2}-\ep_{4},\gamma_{2,3},\ep_{3}\}$ and the same argument as the proof of Lemma \ref{2lem}, we have that exactly one of $a_{1}=2$ or $a_{3}=1$ or $a_{4}=1$ holds.
If $a_{1}=2$ (resp. $a_{3}=1$, $a_{4}=1$), we can check that $\{\ep_{1}+\ep_{3},\beta_{3}\}$ (resp. $\{\beta_{3},\gamma_{3,4}\}$, $\{\ep_{3},\ep_{1}-\ep_{2}\}$) is a bad pair.

Assume that $a_{4}=0$.
We turn to the condition that $2\in \im(\varphi_\lambda^J)$.
By taking $\Psi=\{\alpha_{4},\alpha_{1},\alpha_{2},\gamma_{2,3}\}$ and the same argument as the proof of Lemma \ref{n2lem}, we have that exactly one of $a_{1}=1$ or $a_{2}=1$ or $a_{3}=0$ holds.
If $a_{1}=1$ (resp. $a_{2}=1$, $a_{3}=0$), we can check that $\{\ep_{2}+\ep_{4},\ep_{1}-\ep_{2}\}$ (resp. $\{\ep_{2}+\ep_{3},\ep_{1}-\ep_{3}\}$, $\{\ep_{2}+\ep_{4},\ep_{2}-\ep_{4}\}$) is a bad pair.
\end{proof}
\begin{prop}There are no initialized irreducible homogeneous Ulrich bundles on $F_{4}/P_{1,3,4}$.
\end{prop}
\begin{proof}Suppose for contradiction that there is an initialized irreducible homogeneous Ulrich bundle $E_{\lambda}$ on $F_{4}/P_{1,3,4}$ with $\lambda=\sum_{i=1}^{4}a_{i}\varpi_{i}$.
By the same argument as the proof of Lemma \ref{n1lem}, we deduce that exactly one of $a_{1}=0$ or $a_{3}=0$ or $a_{4}=0$ holds.

Assume that $a_{1}=0$.
We examine the condition that $2\in \im(\varphi_\lambda^J)$.
By taking $\Psi=\{\alpha_{1},\ep_{2}-\ep_{4},\alpha_{4},\alpha_{3}\}$ and the same argument as the proof of Lemma \ref{n2lem}, we have that exactly one of $a_{4}=1$ or $a_{2}=0$ or $a_{3}=1$ holds.
If $a_{4}=1$ (resp. $a_{2}=0$, $a_{3}=1$), we see that $\{\ep_{2}+\ep_{4},\ep_{1}-\ep_{2}\}$ (resp. $\{\ep_{1}+\ep_{3},\ep_{1}+\ep_{4}\}$, $\{\ep_{1}-\ep_{4},\beta_{2}\}$) is a bad pair.

Suppose that $a_{3}=0$.
We turn to the condition that $2\in \im(\varphi_\lambda^J)$.
By taking $\Psi=\{\alpha_{3},\ep_{3}+\ep_{4},\alpha_{1},\gamma_{2,3}\}$ and the same argument as the proof of Lemma \ref{n2lem}, we have that exactly one of $a_{4}=2$ or $a_{1}=1$ or $a_{2}=0$ holds.
When $a_{2}=0$, we see that $\{\ep_{1}+\ep_{3},\ep_{1}+\ep_{4}\}$ is a bad pair.
If $a_{1}=1$, then this contradicts our hypothesis by Lemma \ref{lemF}.
Assume that $a_{4}=2$.
We investigate the condition that $4\in \im(\varphi_\lambda^J)$.
By $\varphi_{\lambda}^{J}(\gamma_{2,4})=\frac{6+2a_{2}}{2}$, and $\varphi_{\lambda}^{J}(\beta_{2})=\frac{2a_{2}+7}{3}$, and
the same argument as the proof of Lemma \ref{22lem},
we deduce that $a_{2}\geq4$.
In this case, we obtain $a_{1}=3$ holds by taking $\Psi=\{\alpha_{3},\gamma_{2,3},\alpha_{4},\alpha_{1}\}$ and the same argument as the proof of Lemma \ref{22lem}.
If $a_{1}=3$, we can check that $\{\ep_{2}+\ep_{4},\ep_{1}-\ep_{2}\}$ is a bad pair.

Assume that $a_{4}=0$.
We examine the condition that $2\in \im(\varphi_\lambda^J)$.
By taking $\Psi=\{\alpha_{4},\gamma_{2,3},\alpha_{1}\}$ and the same argument as the proof of Lemma \ref{2lem}, we have that exactly one of $a_{3}=2$ or $a_{1}=1$ holds.
If $a_{1}=2$, we see that $\{\ep_{2}+\ep_{3},\gamma_{2,3}\}$ is a bad pair.
Suppose that $a_{3}=2$.
We investigate the condition that $4\in \im(\varphi_\lambda^J)$.
By taking $\Psi=\{\alpha_{4},\gamma_{2,3},\alpha_{3},\ep_{1}-\ep_{2},\alpha_{1}\}$ and the same argument as the proof of Lemma \ref{22lem}, we have that exactly one of $a_{1}=3$ or $a_{2}=3$ holds.
When $a_{1}=3$, we have $\varphi_{\lambda}^{J}(\ep_{2}+\ep_{3})=\frac{9+2a_{2}}{2}$.
This contradicts our hypothesis by the same argument as the proof of Lemma \ref{lemF}.
If $a_{2}=3$, we see that $\{\beta_{4},\beta_{3}\}$ is a bad pair. 
\end{proof}
\begin{lem}\label{lemF}Let $E_{\lambda}$ be an an initialized irreducible homogeneous vector bundle on $F_{4}/P_{1,3,4}$ with $\lambda=\varpi_{1}+a_{2}\varpi_{2}+a_{4}\varpi_{4}$.
Then $E_{\lambda}$ is not Ulrich.
\end{lem}
\begin{proof}In this case, we have $\varphi_{\lambda}^{J}(\ep_{2}+\ep_{3})=\frac{5+2a_{2}}{2}$. Since $5$ is congruent to $1$ modulo $2$ and $2a_{2}$ is congruent to $0$ modulo $2$ for every non-negative integer $a_{2}$, we see that $5+2a_{2}$ is congruent to $1$ modulo $2$ for every non-negative integer $a_{2}$.
Therefore, we have $\varphi_{\lambda}^{J}(\ep_{2}+\ep_{3})\notin[1,\dim(F_{4}/P_{1,3,4})]$.
\end{proof}
\begin{prop}There are no initialized irreducible homogeneous Ulrich bundles on $F_{4}/P_{2,3,4}$.
\end{prop}
\begin{proof}Suppose for contradiction that there is an initialized irreducible homogeneous Ulrich bundle $E_{\lambda}$ on $F_{4}/P_{2,3,4}$ with $\lambda=\sum_{i=1}^{4}a_{i}\varpi_{i}$.
By the same argument as the proof of Lemma \ref{n1lem}, we deduce that exactly one of $a_{2}=0$ or $a_{3}=0$ or $a_{4}=0$ holds.

Assume that $a_{2}=0$.
We investigate the condition that $4\in \im(\varphi_\lambda^J)$.
By $\varphi_{\lambda}^{J}(\ep_{3}+\ep_{4})=\frac{2+a_{3}}{2}$, and $\varphi_{\lambda}^{J}(\ep_{3})=\frac{a_{3}+3}{3}$, and
the same argument as the proof of Lemma \ref{22lem},
we deduce that $a_{3}\geq6$.
In this case, we obtain that exactly one of $a_{1}=0$ or $a_{4}=1$ holds by taking $\Psi=\{\alpha_{2},\ep_{2}-\ep_{4},\alpha_{4}\}$ and the same argument as the proof of Lemma \ref{22lem}.
If $a_{1}=0$ (resp. $a_{4}=1$), we see that $\{\ep_{3}+\ep_{4},\gamma_{2,3}\}$ (resp. $\{\ep_{2}+\ep_{4},\ep_{3}+\ep_{4}\}$) is a bad pair.

Suppose that $a_{3}=0$.
We turn to the condition that $4\in \im(\varphi_\lambda^J)$.
By $\varphi_{\lambda}^{J}(\ep_{3}+\ep_{4})=\frac{2+a_{2}}{2}$, and $\varphi_{\lambda}^{J}(\ep_{3})=\frac{2a_{2}+3}{3}$, and
the same argument as the proof of Lemma \ref{22lem},
we deduce that $a_{3}\geq6$.
In this case, we obtain $a_{4}=2$ by taking $\Psi=\{\alpha_{3},\gamma_{2,3}\}$ and the same argument as the proof of Lemma \ref{22lem}.
When $a_{4}=2$, this contradicts our hypothesis by Lemma \ref{lemF2}.

Assume that $a_{4}=0$.
We examine the condition that $2\in \im(\varphi_\lambda^J)$.
By taking $\Psi=\{\alpha_{4},\gamma_{2,3},\alpha_{2}\}$ and the same argument as the proof of Lemma \ref{2lem}, we have that exactly one of $a_{3}=2$ or $a_{2}=1$ holds.
If $a_{2}=1$, we can check that $\{\ep_{3},-\ep_{2}+\ep_{1}\}$ is a bad pair.
When $a_{3}=2$, we have $\varphi_{\lambda}^{J}(\ep_{3})=\frac{2a_{2}+5}{3}$ and $\varphi_{\lambda}^{J}(\ep_{1}-\ep_{2})=\frac{5+a_{2}}{3}$.
In this case, we see that this contradicts our hypothesis by same argument as the proof of Lemma \ref{lemF2}.
\end{proof}
\begin{lem}\label{lemF2}Let $E_{\lambda}$ be an an initialized irreducible homogeneous vector bundle on $F_{4}/P_{2,3,4}$ with $\lambda=a_{1}\varpi_{1}+a_{2}\varpi_{2}+2\varpi_{4}$.
Then $E_{\lambda}$ is not Ulrich.
\end{lem}
\begin{proof}Suppose for contradiction that $E_{\lambda}$ is Ulrich bundle.
By $\varphi_{\lambda}^{J}(\ep_{3})=\frac{2a_{2}+3}{3}$, we see that $a_{2}$ is congruent to $0$ modulo $3$. 
On the other hand, we have that $a_{2}$ is congruent to $1$ modulo $3$ by $\varphi_{\lambda}^{J}(\ep_{1}-\ep_{2})=\frac{5+a_{2}}{3}$.
Therefore, at least one of the value $\varphi_{\lambda}^{J}(\ep_{3})$ or $\varphi_{\lambda}^{J}(\ep_{1}-\ep_{2})$ is not integer.
By Lemma \ref{lem}, this contradicts $\im(\varphi_{\lambda}^{J})=[1,\dim(F_{4}/P_{2,3,4}])$.
\end{proof}
\subsection{Proof of theorem \ref{thmF4} when $|J|=4$}
We consider the case of full flag of type $F_{4}$. In other words, $I$ equals the set of simple roots $\Delta_{F_{4}}$.
\begin{prop}There are no initialized irreducible homogeneous Ulrich bundles on $F_{4}/B$.
\end{prop}
\begin{proof}Suppose for contradiction that there is an initialized irreducible homogeneous Ulrich bundle $E_{\lambda}$ on $F_{4}/B$ with $\lambda=\sum_{i=1}^{4}a_{i}\varpi_{i}$.
By the same argument as the proof of Lemma \ref{n1lem}, we deduce that exactly one of $a_{1}=0$ or $a_{2}=0$ or $a_{3}=0$ or $a_{4}=0$ holds.

Let us consider the case $a_{1}=0$ first.
We examine the condition that $2\in \im(\varphi_\lambda^J)$.
By taking $\Psi=\{\alpha_{1},\ep_{2}-\ep_{4},\alpha_{3},\alpha_{4}\}$ and the same argument as the proof of Lemma \ref{2lem}, we have that exactly one of $a_{2}=2$ or $a_{3}=1$ or $a_{4}=1$ holds.
If $a_{2}=2$ ($a_{4}=1$), we see that $\{\ep_{2}+\ep_{4},\ep_{3}\}$ (resp. $\{\ep_{2}+\ep_{4},\ep_{1}-\ep_{2}\}$) is a bad pair.
When $a_{3}=1$, we have $\varphi_{\lambda}^{J}(\ep_{3})=\frac{2a_{2}+4}{3}$ and $\varphi_{\lambda}^{J}(\ep_{2}+\ep_{4})=\frac{4+a_{2}}{3}$.
In this case, we see that this contradicts our hypothesis by same argument as the proof of Lemma \ref{lemF2}.

Next, let us consider the case $a_{2}=0$.
We examine the condition that $2\in \im(\varphi_\lambda^J)$.
By $\varphi_{\lambda}^{J}(\ep_{3}+\ep_{4})=\frac{2+a_{3}}{2}$, and $\varphi_{\lambda}^{J}(\ep_{3})=\frac{a_{3}+3}{3}$, and
the same argument as the proof of Lemma \ref{22lem},
we deduce that $a_{3}\geq6$.
In this case, we obtain that exactly one of $a_{1}=2$ or $a_{4}=1$ holds by taking $\Psi=\{\alpha_{2},\alpha_{4},\ep_{2}-\ep_{4}\}$ and the same argument as the proof of Lemma \ref{22lem}.
If $a_{1}=2$ (resp. $a_{4}=1$), we can check that $\{\ep_{2}+\ep_{4},\ep_{3}\}$ (resp. $\{\ep_{3},-\ep_{2}+\ep_{1}\}$) is a bad pair.

Next, let us consider the case $a_{3}=0$.
We turn to the condition that $2\in \im(\varphi_\lambda^J)$.
By $\varphi_{\lambda}^{J}(\ep_{3}+\ep_{4})=\frac{2+a_{2}}{2}$, and $\varphi_{\lambda}^{J}(\ep_{3})=\frac{2a_{2}+3}{3}$, and
the same argument as the proof of Lemma \ref{22lem},
we deduce that $a_{2}\geq6$.
In this case, we obtain that exactly one of $a_{1}=1$ or $a_{4}=2$ holds by taking $\Psi=\{\alpha_{3},\alpha_{1},\gamma_{2,3}\}$ and the same argument as the proof of Lemma \ref{22lem}.
If $a_{1}=2$ (resp. $a_{4}=2$), we can check that $\{\ep_{3}+\ep_{4},\ep_{2}-\ep_{4}\}$ (resp. $\{\ep_{3},-\ep_{2}+\ep_{1}\}$) is a bad pair.

Finally, let us consider the case $a_{4}=0$.
We investigate the condition that $2\in \im(\varphi_\lambda^J)$.
By taking $\Psi=\{\alpha_{4},\gamma_{2,3},\alpha_{1},\alpha_{2}\}$ and the same argument as the proof of Lemma \ref{2lem}, we have that exactly one of $a_{1}=1$ or $a_{2}=1$ or $a_{3}=2$ holds.
If $a_{1}=2$ (resp. $a_{2}=1$, $a_{3}=2$), we can check that $\{\ep_{2}+\ep_{4},\ep_{1}-\ep_{2}\}$ (resp. $\{\ep_{3},-\ep_{2}+\ep_{1}\}$, $\{\ep_{3},-\ep_{2}+\ep_{1}\}$) is a bad pair.
\end{proof}
\section{type $G_{2}$}
Finally, we prove the following theorem in this section.
\begin{thm}\label{thmG2}There are no initialized irreducible homogeneous Ulrich bundles on $G_{2}/B$, where $B$ is a Borel subgroup of $G_{2}$.
\end{thm}
We prepare notations before the proof. Let $G$ be a simply connected simple Lie group with a Dynkin diagram of type $G_{2}$, as follows.
\begin{center}
$G_{2}$:\dynkin[label]{G}{2}
\end{center}

Let $V_{G_{2}}$ be a vector subspace of $\mathbb{R}^{3}$ consisting of $\sum_{i=1}^{3}a_{i}\epsilon_{i}$ such that $\sum_{i}^{3}a_{i}=0$. We define $\Phi_{G_{2}}=V_{G_{2}}\cap I_{3}$. As a set, $\Phi_{G_{2}}$ comprises
$$\{\pm(\epsilon_{i}-\epsilon_{j})\ |\ 1\leq i<j\leq3\} \cup \{\pm(2\epsilon_{i}-\epsilon_{j}-\epsilon_{k})\ |\ \{1,2,3\}=\{i,j,k\}\}.$$ 
As a base, we take
$$\Delta_{G_{2}}=\{\alpha_{1}:=\epsilon_{2}-\epsilon_{3},\ \alpha_{2}:=\epsilon_{1}-2\epsilon_{2}+\epsilon_{3}\}.$$
The fundamental weights are as follows:
\begin{eqnarray*}
\varpi_{1}&=&\epsilon_{1}-\epsilon_{3}\\
\varpi_{2}&=&2\epsilon_{1}-\epsilon_{2}-\epsilon_{3}.
\end{eqnarray*}
The set of positive roots, $\Phi_{G_{2}}^{+}$, is
$$\{\epsilon_{1}-\epsilon_{2},\ep_{1}-\ep_{3},\alpha_{1},-\beta_{1},\alpha_{2},\beta_{3}\},$$
where $\beta_{i}\ (1\leq i\leq3)$ is a  root such that the coefficient of $\epsilon_{i}$ is $-2$ and those of the remaining are $1$.

Let $E_{\lambda}$ be an initialized irreducible homogeneous vector bundle on $G_{2}/B$ with $\lambda=a_{1}\varpi_{1}+a_{2}\varpi_{2}$.
Then we have
\begin{eqnarray*}
\varphi_{\lambda}^{1,2}(\ep_{1}-\ep_{2})&=&\frac{a_{1}+3a_{2}+4}{4}\\
\varphi_{\lambda}^{1,2}(\ep_{1}-\ep_{3})&=&\frac{2a_{1}+3a_{2}+5}{5}\\
\varphi_{\lambda}^{1,2}(\alpha_{1})&=&a_{1}+1\\
\varphi_{\lambda}^{1,2}(-\beta_{1})&=&\frac{a_{1}+2a_{2}+3}{3}\\
\varphi_{\lambda}^{1,2}(\alpha_{2})&=&a_{2}+1\\
\varphi_{\lambda}^{1,2}(\beta_{3})&=&\frac{a_{1}+a_{2}+2}{2}
\end{eqnarray*}

\begin{proof}[Proof of Theorem \ref{thmG2}]Suppose for contradiction that there is an initialized irreducible homogeneous Ulrich bundle $E_{\lambda}$ on $G_{2}/B$ with $\lambda=a_{1}\varpi_{1}+a_{2}\varpi_{2}$.
By the same argument as the proof of Lemma \ref{n1lem}, we deduce that exactly one of $a_{1}=0$ or $a_{2}=0$ holds.

Let us consider the case $a_{1}=0$.
Then, we have $\varphi_{\lambda}^{1,2}(-\beta_{1})=\frac{2a_{2}+3}{3}$, and $\varphi_{\lambda}^{1,2}(\beta_{3})=\frac{a_{2}+2}{2}$.
For these to be integers, we deduce that $a_{2}$ must be congruent to $0$ modulo $6$.
On the other hand, by $\dim(G_{2}/B)=6$ and $\varphi_{\lambda}^{1,2}(\alpha_{2})=a_{2}+1$,
we see that $a_{2}$ is less than or equal to $5$.
Therefore, we obtain $a_{2}=0$.
However, if $a_{2}=0$, we have $\varphi_{\lambda}^{1,2}(\alpha_{1})=\varphi_{\lambda}^{1,2}(\alpha_{2})=1$.
This contradicts injectivity of $\varphi_{\lambda}^{1,2}$.

Let us consider the $a_{2}=0$.
Then, we have $\varphi_{\lambda}^{1,2}(-\beta_{1})=\frac{a_{1}+3}{3}$, and $\varphi_{\lambda}^{1,2}(\beta_{3})=\frac{a_{1}+2}{2}$.
For these to be integers, we deduce that $a_{1}$ must be congruent to $0$ modulo $6$.
On the other hand, by $\dim(G_{2}/B)=6$ and $\varphi_{\lambda}^{1,2}(\alpha_{2})=a_{1}+1$,
we see that $a_{1}$ is less than or equal to $5$.
Therefore, we obtain $a_{1}=0$.
However, if $a_{1}=0$, we have $\varphi_{\lambda}^{1,2}(\alpha_{1})=\varphi_{\lambda}^{1,2}(\alpha_{2})=1$.
This contradicts injectivity of $\varphi_{\lambda}^{1,2}$.
\end{proof}
\par\ \par
\scriptsize{School of Fundamental Science and Engineering, Waseda University, 3-4-1, Okubo, Shinjuku, Tokyo, 169-8555.\par
Email address: yusuke216144@akane.waseda.jp.}


\begin{thebibliography}{MMMN}
\bibitem{ACCMT}M. Aprodu, G. Casnati, L. Costa, R. M. Mir\'{o}-Roig, and M. Teixidor I Bigas. Theta divisors and Ulrich bundles on geometrically ruled surfaces. Ann. Mat. Pura Appl. (4), $\mathbf{199}$ 2020 (1):199-216.
\bibitem{ACMR}M. Aprodu, L. Costa, and R. M. Mir\'{o}-Roig, Ulrich bundles on ruled surfaces. J. Pure Appl. Algebra, $\mathbf{222}$ (2018) 1 :131-138.
\bibitem{AFO}M. Aprodu, G. Farkas, and A. Ortega. Minimal resolutions, Chow forms and Ulrich bundles on K3 surfaces. J. Reine Angew. Math., $\mathbf{730}$ (2017) :225-249.
\bibitem{HUB}J. Herzog, B. Ulrich and J. Backelin, Linear maximal Cohen-Macaulay modules over strict complete intersections, J. Pure Appl. Algebra $\mathbf{71}$ (1991), no. $2$-$3$, $187$-$202$.
\bibitem{Bea2}A. Beauville,  Ulrich bundles on abelian surfaces. Proc. Amer. Math. Soc., $\mathbf{144}$ (2016) 11:4609-4611.
\bibitem{Bea}A. Beauville, An introduction to Ulrich bundles. Eur. J. Math., $\mathbf{4}$ 2018, (1):26-36.
\bibitem{BMPMT}V. Benedetti, P. Montero, Y. Prieto-Monta\~{n}ez and S. Troncoso, Projective manifolds whose tangent bundle is Ulrich, J. Algebra $\mathbf{630}$ ($2023$), $248$-$273$.
\bibitem{BHU}J. P. Brennan, J. Herzog and B. Ulrich, Maximally generated Cohen-Macaulay modules, Math. Scand. $\mathbf{61}$ ($1987$), no. $2$, $181$-$203$.
\bibitem{BN}L. Borisov and H. Nuer. Ulrich bundles on Enriques surfaces. Int. Math. Res. Not. IMRN,  2018 (13):4171-4189.
\bibitem{Cas1}G. Casnati. Special Ulrich bundles on non-special surfaces with $p_{g}=q= 0$. Internat. J. Math., $\mathbf{28}$ (2017) 8:1750061, 18.
\bibitem{Cas2}G. Casnati. Ulrich bundles on non-special surfaces with $p_{g} = 0$ and $q =1$. Rev. Mat. Complut., $\mathbf{32}$ 2019 (2):559-574.
\bibitem{CFM}G. Casnati, D. Faenzi, and F. Malaspina, Rank two aCM bundles on the del Pezzo threefold with Picard number $3$, J. Algebra, $\mathbf{429}$ ($2015$),$413$-$446$.
\bibitem{CKM}E. Coskun, R. S. Kulkarni and Y. Mustopa, The geometry of Ulrich bundles on del Pezzo surfaces, J. Algebra $\mathbf{375}$ ($2013$), $280$-$301$.
\bibitem{CJ}I. Coskun and L. Jaskowiak, Ulrich partion for two-step flag varieties, Involve $\mathbf{10}(3)$ (2017): 531-539.
\bibitem{CCHMW}I. Coskun, L. Costa, J. Huizenga, R. M. Mir\'{o}-Roig and M. Woolf, Ulrich Schur bundles on flag varieties, J. Algebra $\mathbf{474}$ ($2017$), $49$-$96$.
\bibitem{CM1}L. Costa and R. M. Mir\'{o}-Roig, Homogeneous ACM bundles on a Grassmannian, Adv. in Math., $\mathbf{289}$ (2016) :$95$-$113$.
\bibitem{CM2}L. Costa and R. M. Mir\'{o}-Roig, GL($V$)-invariant Ulrich bundles on Grassmannians, Math. Ann. $\mathbf{361}$ 2015.$1$-$2$, pp. $443$-$457$.
\bibitem{Dem}M. Demazure, A very simple proof of Bott's theorem, {\it Inventiones mathematicae\/}, $33$ 1976, $271$-$272$.
\bibitem{DFR}R. Du, X. Fang, and P. Ren, Homogeneous ACM bundles on isotropic Grassmannians, Forum math. $\mathbf{35}(3)$ (2023), pp. 763-782.
\bibitem{ES}D. Eisenbud and F.-O. Schreyer, Boij-S\"{o}derberg theory, {\it in\ Combinatorial\ aspects\ of\ commutative\ algebra\ and\ algebraic\ geometry}, 35-48, Abel Symp., 6, Springer, Berlin.
\bibitem{ESW}D. Eisenbud, F.-O. Schreyer and J. Weyman, Resultants and Chow forms via exterior syzygies, J. Amer. Math. Soc. $\mathbf{16}$ (2003), no. 3, 537-579.
\bibitem{Fae}D. Faenzi. Ulrich bundles on K3 surfaces. Algebra Number Theory, 13 2019 (6):1443-1454.
\bibitem{FNR}X. Fang, Y. Nakayama and P. Ren, Homogeneous ACM bundles on Grassmannians of exceptional types, preprint.
\bibitem{Fon}A. Fonarev, Irreducible Ulrich bundles on isotropic Grassmannians, {\it Moscow Mathematical Journal}, $\mathbf{16}$ (2016),no. $4$, $711-726$. 
\bibitem{Hor}G. Horrocks, Vector bundles on the punctured spectrum of local ring, {\it Proceedings\ of\ the\ London\ Mathematical\ Society\/}, $3\ 1964\ (4)$:$689-713$.
\bibitem{Hum}J. E. Humphreys, Introduction to Lie Algebras and Representation Theory, {\it Graduate\ Texts\ in\ Mathematics\/}, $9.$\ Springer,\ New\ York,\ $1978.$  
\bibitem{Jan} J. C. Jantzen, Representations of algebraic groups, Second edition. Mathematical Surveys and Monographs, $107$. {\it American\ Mathematical\ Society\/}, Providence, RI, $2003. $
\bibitem{LP}K.-S. Lee and K.-D. Park, Equivariant Ulrich bundles on exceptional homogeneous varieties, {\it Advances in Geometry\/}, $\mathbf{21}$ (2021) no.$2$, $187-205$.
\bibitem{Lop1}A. F. Lopez. On the existence of Ulrich vector bundles on some surfaces of maximal Albanese dimension. Eur. J. Math., $\mathbf{5}$ 2019 (3):958-963.
\bibitem{Lop}A. F. Lopez. On the existence of Ulrich vector bundles on some irregular surfaces. Proc. Amer. Math. Soc., $\mathbf{149}$ 2021 (1):13-26.
\bibitem{MRPM}R. M. R. M. Mir\'{o}-Roig and J. Pons-Llopis. Special Ulrich bundles on regular Weierstrass fibrations. Math. Z., $\mathbf{293}$ 2019 (3-4):1431-1441.
\bibitem{YN}Y. Nakayama, Homogeneous Ulrich bundles on partial flag varieties of types $E_{6}$ and $F_{4}$. preprints.
\bibitem{Ott}G. Ottaviani, Rational homogeneous varieties. {\it Lecture\ notes\ for\ the\ summer\ school\ in Algebraic\ Geometry\ in\ Cortona\/},\ $1995$.
\bibitem{Ulrich}B. Ulrich, Gorenstein rings and modules with high numbers of generators, Math. Z. $\mathbf{188}$ (1984), $23-32$.
\end{thebibliography}
\end{document}